\documentclass[VANCOUVER,STIX1COL]{WileyNJD-v2}

\articletype{Article Type}
\received{26 April 2016}
\revised{6 June 2016}
\accepted{6 June 2016}

\usepackage{placeins}
\usepackage{tikz}
\usetikzlibrary{shapes,patterns,snakes}
\usepackage[belowskip=-0pt]{subcaption}
\usepackage{tabularx,booktabs}
	\newcolumntype{Y}{>{\centering\arraybackslash}X}
\usepackage{makecell}
\usepackage{hhline}

\usepackage{caption}
\usepackage{lineno}
\makeatletter
\let\LN@align\align
\let\LN@endalign\endalign
\renewcommand{\align}{\linenomath\LN@align}
\renewcommand{\endalign}{\LN@endalign\endlinenomath}
\let\LN@gather\gather
\let\LN@endgather\endgather
\renewcommand{\gather}{\linenomath\LN@gather}
\renewcommand{\endgather}{\LN@endgather\endlinenomath}
\makeatother

\DeclareMathOperator{\e}{e}
\DeclareMathOperator{\Real}{Re}

\newcommand{\revision}[1]{{\color{black} #1}}

\begin{document}
\linenumbers 

\title{On ``Optimal'' $h$-Independent Convergence of Parareal and MGRIT using Runge-Kutta Time Integration}

\author[1]{Stephanie Friedhoff}

\author[2]{Ben S. Southworth*}

\authormark{Friedhoff and Southworth}

\address[1]{\orgdiv{Department of Mathematics}, \orgname{Bergische Universit\"at Wuppertal}, \orgaddress{\state{Wuppertal}, \country{Germany}}}
\address[2]{\orgdiv{Department of Applied Mathematics}, \orgname{University of Colorado at Boulder}, \orgaddress{\state{Colorado}, \country{USA}}}
\corres{*Ben Southworth, Department of Applied Mathematics, University of Colorado at Boulder; 526 UCB Boulder, CO 80309. \email{ben.s.southworth@gmail.com}}

\abstract[Summary]{
\revision{Although convergence of the Parareal and multigrid-reduction-in-time (MGRIT) parallel-in-time algorithms
is well studied, results on their optimality is limited. Appealling to recently derived tight bounds of
two-level Parareal and MGRIT convergence, this paper proves (or disproves) $h_x$- and $h_t$-independent
convergence of two-level Parareal and MGRIT, for linear problems of the form $\mathbf{u}'(t) +
\mathcal{L}\mathbf{u}(t) = f(t)$, where $\mathcal{L}$ is symmetric positive definite and Runge-Kutta
time integration is used. The theory presented in this paper also encompasses analysis of some modified 
Parareal algorithms, such as the $\theta$-Parareal method, and shows that not all Runge-Kutta schemes are equal
from the perspective of parallel-in-time. Some schemes, particularly L-stable methods, offer significantly better
convergence than others as they are guaranteed to converge rapidly at both limits of small and large $h_t\xi$, 
where $\xi$ denotes an eigenvalue of $\mathcal{L}$ and $h_t$ time-step size.
On the other hand, some schemes do not obtain $h$-optimal
convergence, and two-level convergence is restricted to certain regimes. In certain cases, an $\mathcal{O}(1)$
factor change in time step $h_t$ or coarsening factor $k$ can be the difference between convergence factors
$\rho\approx0.02$ and divergence! The analysis is extended to skew symmetric operators as well, which
cannot obtain $h$-independent convergence and, in fact, will generally not converge for a sufficiently
large number of time steps. Numerical results confirm the analysis in practice and
emphasize the importance of a priori analysis in choosing an effective coarse-grid scheme and coarsening factor.
A Mathematica notebook to perform a priori two-grid analysis is available at}
\url{https://github.com/XBraid/xbraid-convergence-est}. }


\jnlcitation{\cname{\author{S. Friedhoff} and 
\author{B.S. Southworth}} (\cyear{2019}), 
\ctitle{On ``Optimal'' $h$-Independent Convergence of Parareal and 
MGRIT using Runge-Kutta Time Integration}, \cjournal{Numer. Linear Alg. Appl.}, \cvol{2019;00:0--0}.}

\maketitle

\section{Introduction}\label{sec:intro}
Starting with the seminal works of Runge \cite{CRunge_1895} and Kutta \cite{WKutta_1901} around 1900, many
Runge-Kutta methods have been developed for solving large, stiff and non-stiff systems of ordinary differential equations (ODEs)
\cite{EHairer_etal_1993,EHairer_GWanner_2010,JCButcher_2016}. For stiff problems, implicit Runge-Kutta (IRK) methods
\cite{EHairer_GWanner_2010} are of interest due to their good stability properties. The computational cost of IRK methods,
however, is high, because they require the solution of a large coupled system of equations for each time step. Diagonally or
singly implicit Runge-Kutta (DIRK or SDIRK) methods \cite{KBurrage_1995,EHairer_GWanner_2010,JCButcher_2016} make
this cost more tractable by reducing the single large system to a smaller system for each Runge-Kutta stage. Nevertheless,
traditional time stepping remains a sequential process. With emerging computing architectures largely increasing in number
of total processors, it is desirable to be able to distribute the process of time integration across multiple processors. Space-time
and time-parallel methods offer a level of parallelism to this process\cite{Gander2015_Review} by decomposing the temporal
domain into subintervals and performing computations in different subintervals simultaneously.

In this paper, we consider the Parareal \cite{JLLions_etal_2001a} and multigrid-reduction-in-time (MGRIT) \cite{RDFalgout_etal_2014}
algorithms. Both methods have been extensively studied
\cite{GBal_2005,MJGander_SVandewalle_2007,MJGander_EHairer_2008,SFriedhoff_SMacLachlan_2015a,VADobrev_etal_2017,BSSouthworth_2018a}
and applied to a variety of problems including fluid dynamics applications\cite{PFFischer_etal_2005}, plasma simulations \cite{DSamaddar_etal_2010},
and linear elasticity \cite{AHessenthaler_etal_2018}. While Parareal can be interpreted in a variety of frameworks of numerical methods \cite{MJGander_SVandewalle_2007,MJGander_etal_2018} and while there are clear and important differences in comparison with MGRIT,
in this paper, we consider Parareal as a two-level time-multigrid scheme and, thus, we view Parareal and MGRIT as two methods within the
same broader family. Both algorithms combine time stepping on the initially discretized temporal domain, referred to as the fine grid, with time
stepping on a coarse temporal mesh that uses a larger time step. One interesting property of both methods is that convergence is guaranteed
in a finite number of steps \cite{MJGander_SVandewalle_2007,RDFalgout_etal_2014,MJGander_etal_2018}. From the multigrid perspective,
this property is important, but taking it a step further, namely showing that Parareal and MGRIT are optimal in the sense that they can obtain
convergence for a discrete problem that is independent of the spatial or temporal mesh spacing and problem size, is of great interest. Such
$h$-independent multigrid convergence was shown for elliptic problems as early as the work by Brandt in 1977\cite{brandt1977multi}, and
has since been generalized to other contexts (for example, see \cite{SCBrenner_1999,SSerra_Capizzano_CTablino_Possio_2004,AArico_MDonatelli_2007}).

For linear problems  of the form $\mathbf{u}'(t) + \mathcal{L}\mathbf{u}(t) = f(t)$, where $\mathcal{L}$ is symmetric positive definite (SPD),
numerical experiments have suggested optimality of Parareal and MGRIT using backward Euler time integration\cite{GAStaff_EMRonquist_2005,MJGander_SVandewalle_2007,RDFalgout_etal_2014}.
Furthermore, for such problems, $h$-independent convergence of Parareal has been considered previously in several papers, where
specific combinations of coarse- and fine-grid time propagators are analyzed. In particular, Mathew et al.\cite{TPMathew_etal_2010}
showed $h$-independent convergence of Parareal for backward Euler time integration on the coarse grid and either backward Euler
or Crank-Nicolson on the fine grid. For backward Euler time integration on the coarse grid, Wu\cite{SLWu_2015} proved $h$-independent
convergence of Parareal for the parameter-dependent TR/BDF2 scheme\cite{Bank_etal_1985}, provided that the parameter of the TR/BDF2
method is chosen in a specific regime, and the two-stage DIRK method under the assumption that the coarse-grid time-step size is an even
multiple of the fine-grid time-step size. Wu and Zhou\cite{SLWu_TZhou_2015} analyzed optimality of three Parareal solvers, all using
backward Euler time integration on the coarse grid and the trapezoidal rule, third-order DIRK method, and fourth-order Gauss Runge-Kutta
method on the fine grid. For these three schemes, $h$-independent convergence can be obtained only if the ratio between the fine- and
coarse-grid time-steps is even and greater than or equal to a method-specific threshold, given in terms of the coarse- and fine-grid time-step sizes,
and the maximum eigenvalue of the SPD operator $\mathcal{L}$. Although each of these results is interesting, each also required somewhat
lengthy analyses to derive, and there remains a lack in broader understanding of $h$-independent convergence of MGRIT and Parareal. 

Recently, two-level theory was developed by Southworth \cite{BSSouthworth_2018a}
that provides tight bounds on linear Parareal and two-level MGRIT for arbitrary time-propagation operators.
\revision{For diagonalizable operators, the upper bounds derived for the convergence of Parareal and MGRIT in
\cite{BSSouthworth_2018a} are effectively equivalent to those first developed in \cite{VADobrev_etal_2017}. A lower
bound derived in \cite{BSSouthworth_2018a} proves the upper bounds in \cite{VADobrev_etal_2017,BSSouthworth_2018a}
are tight to order $\mathcal{O}(1/N_c)$, for $N_c$ coarse-grid time steps. For SPD operators, the difference in upper and
lower bounds is very small. However, as we will see later, for complex/purely imaginary eigenvalues, the number of
coarse-grid time steps, $N_c$, matters. } The central aim of this paper
is to use these tight bounds to (i) analytically show $h$-independent convergence of two-level Parareal and MGRIT for linear
problems of the form $\mathbf{u}'(t) + \mathcal{L}\mathbf{u}(t) = f(t)$, where $\mathcal{L}$ is SPD and Runge-Kutta time
integration is used, \revision{and (ii) show that $h$-independent convergence can largely not be obtained for skew-symmetric
spatial operators.} Thus, this paper highlights the practical implications of the new theory, particularly for parabolic-type
PDEs with SPD spatial operators, \revision{and some purely hyperbolic PDEs represented by a skew-symmetric operator.}

The relatively clean formulae derived in \cite{VADobrev_etal_2017,BSSouthworth_2018a} allow for straightforward
convergence analysis of arbitrary coarse- and
fine-grid propagators, coarsening factors, and time-step sizes. Therefore, this paper can be seen as a substantial generalization of previous
works by Mathew et al.\cite{TPMathew_etal_2010} and by Wu and Zhou\cite{SLWu_2015,SLWu_TZhou_2015} on $h$-independent convergence
of Parareal. Our theoretical framework also encompasses analysis of various modified Parareal algorithms, such as the (scalar) $\theta$-Parareal
method \cite{GAriel_etal_2018a} and modified A-/L-stable fine-grid propagators introduced in \cite{SLWu_2016}.

Because the underlying two-grid bounds used in this work are tight, they allow us to easily demonstrate that not all Runge-Kutta schemes are
equal from the perspective of parallel-in-time. The dynamics of ``good'' coarse- and fine-grid time integration schemes for Parareal and MGRIT
is complicated, and in certain cases seemingly small changes in coarsening factor or time-step size can have dramatic effects on convergence
of MGRIT. Furthermore, properties that yield fast Parareal or MGRIT convergence are not necessarily consistent with the properties desired in
a serial integration scheme, making the design of an effective Parareal/MGRIT algorithm not necessarily intuitive. The work in this paper presents a
straightforward method to develop effective MGRIT and Parareal methods using a priori analysis, including publicly available Mathematica
code to test arbitrary Runge-Kutta schemes. 

The remainder of this paper is organized as follows. Section \ref{sec:pint_rk} reviews important definitions and properties of Runge-Kutta
methods and describes the Parareal and MGRIT algorithms. Two-level convergence results for Parareal and MGRIT are summarized in
the beginning of Section \ref{sub:parareal_mgrit_conv}, and we proceed to state and prove new theoretical results relating convergence
of Parareal and MGRIT with general Runge-Kutta theory. Sections \ref{sec:erk} and \ref{sec:irk} build on these results to present tight
convergence bounds for a wide variety of explicit and implicit Runge-Kutta methods, respectively. Numerical experiments, presented in
Section \ref{sec:ex} demonstrate that the analysis applies in the practical setting.  \revision{All numerical results in this paper use the \textit{same}
integration scheme on the fine and coarse grid, unless otherwise specified.} \revision{The difficulties posed by imaginary eigenvalues
and corresponding lack of optimal convergence of Parareal and MGRIT is discussed in Section \ref{sec:imag},} and
concluding remarks are given in Section \ref{sec:conclusions}.

\section{Parallel-in-time algorithms using Runge-Kutta time integration}\label{sec:pint_rk}

\subsection{Model problem and Runge-Kutta methods}\label{sub:rk_theory}

Consider the linear system of ODEs
\begin{align}\label{eq:problem}
	\mathbf{u}'(t) + \mathcal{L}\mathbf{u}(t) = f(t) \quad\text{in }(0,T], \quad \mathbf{u}(0) = \mathbf{u}_0,
\end{align}
where $\mathcal{L}\in\mathbb{R}^{M\times M}$ is a discrete, time-independent,  SPD operator on some domain $\Omega$ and
$f$ a solution-independent forcing term $f: [0,T]\to\mathbb{R}^M$. Such problems arise, for example, in the numerical approximation
of partial differential equations (PDEs), where $\mathcal{L}$ is the discretization operator of an open spatial domain $\Omega\subseteq\mathbb{R}^d$
using $M$ degrees of freedom. We discretize \eqref{eq:problem} on an equidistant time grid consisting 
of $N\in\mathbb{N}$ time intervals with a time-step size $h_t=T/N$, using an $s$-stage Runge-Kutta scheme
characterized by the Butcher tableaux 
\[
	\renewcommand\arraystretch{1.2}
	\begin{array}
	{c|c}
	\mathbf{c}_0 & A_0\\
	\hline
	& \mathbf{b}_0^T
	\end{array}.
\]
Given the Runge-Kutta matrix $A_0 = (a_{ij})$, weight vector $\mathbf{b}_0^T = (b_1, \ldots, b_s)^T$, and 
nodes $\mathbf{c}_0 = (c_0, \ldots, c_s)$, the discrete solution is propagated in time via
\begin{align}\label{eq:rk_time_int}
	\mathbf{u}_{n+1} = \Phi\mathbf{u}_n + \sum_{j=1}^s\mathcal{R}_j(-h_t\mathcal{L})h_t f(t_n + c_jh_t), \quad n = 0, \ldots, N-1.
\end{align}
Here, the time-propagation operator is given by
\begin{align}\label{eq:phi}
	\Phi = I - h_t\mathbf{b}_0^T\otimes I \left(I + h_t A_0\otimes \mathcal{L}\right)^{-1} (\mathbf{1}_s \otimes \mathcal{L})
\end{align}
for rational functions $\mathcal{R}_j$, defined as
\[
	\mathcal{R}_j(z) = \mathbf{b}^T(I-zA_0)^{-1}\mathbf{e}_j, \quad j = 1, \ldots, s,
\]
where $\mathbf{1}_s$ and $\mathbf{e}_j$ denote a vector of ones of length $s$ and the $j$-th unit vector, respectively.

Note that the time-stepping operator $\Phi$ is a rational function of $\mathcal{L}$ and, thus, it can be diagonalized with the eigenvectors of $\mathcal{L}$. 
If $\xi$ is an eigenvalue of $\mathcal{L}$, the corresponding eigenvalue of $\Phi$ is given by
\begin{align}\label{eq:lambda}
	\lambda(h_t\xi) = 1 - h_t\xi\mathbf{b}_0^T(I + h_t\xi A_0)^{-1}\mathbf{1}.
\end{align}
Interestingly, this is exactly the stability function for a given Runge-Kutta method. Recall the definition of 
stability for general Runge-Kutta schemes \cite{JCButcher_2016,EHairer_GWanner_2010}:
\begin{definition}[Stability]\label{def:stability}
	Let $A$ be the Runge-Kutta matrix and $\mathbf{b}$ be the weight vector of a Runge-Kutta method. The complex-valued function
	\begin{align}\label{eq:stability_fcn}
		R_0(z) = \frac{\det(I-zA+z\mathbf{1}_s\mathbf{b}^T)}{\det(I-zA)} = \frac{P(z)}{Q(z)}
	\end{align}
	is called the \emph{stability function} of the Runge-Kutta method, where a method is stable for $z\in\mathbb{C}$ when $|R_0(z)|<1$.
	It can be interpreted as the numerical solution after one step for Dahlquist's test equation, given by
	\[
		y'(t) = \lambda y(t), \quad y(0) = 1, \quad z = h_t\lambda.
	\]
\end{definition}
Because our model problem is effectively a matrix-valued variant of the Dahlquist equation \eqref{eq:problem}, the time-stepping operator
$\Phi$ \eqref{eq:phi} indeed is given by the stability function evaluated at $z=-h_t\mathcal{L}$, that is, $\Phi = R_0(-h_t\mathcal{L})$
(in rational function form, the denominator can be thought of in terms of matrix inverses). 

For explicit Runge-Kutta (ERK) schemes, $A$ is strictly lower triangular and, thus, $\det(I-zA)=1$. As a consequence, the stability
function of $s$-stage ERK schemes is a polynomial of degree $\leq s$. For diagonally implicit Runge-Kutta (DIRK) methods, $A$
is lower triangular, and for singly diagonally implicit Runge-Kutta (SDIRK) methods, $A$ is lower triangular with constant diagonal.
Thus, the stability function of implicit $s$-stage Runge-Kutta methods is a rational function with polynomials $P$ and $Q$ of degree
$\leq s$ in the numerator and denominator.

We conclude this section by recalling two definitions and a theorem regarding Runge-Kutta methods, which prove
useful for further analysis in this paper \cite{JCButcher_2016,EHairer_GWanner_2010}. 

\begin{definition}[A-/L-stablity]
	A Runge-Kutta method with stability function $R_0(z)$ is called \emph{A-stable} if
	\[
		|R_0(z)| \leq 1, \quad\text{for all } z\in\mathbb{C} \text{ with } \Real(z) \leq 0.
	\]
	A method is called \emph{L-stable} if it is A-stable and if, in addition,
	\[
		\lim_{|z|\to\infty}R_0(z) = 0.
	\]
\end{definition}
\begin{definition}[Stiffly accurate]
A Runge-Kutta method of $s$ stages is called \emph{stiffly accurate} if $a_{sj} = b_j$ for all $j$\cite{prothero1974stability}.
\end{definition}
\begin{theorem}\label{th:exp_p}
	If the Runge-Kutta method with stability function $R_0(z)$ is of order $p$, then
	\[
		R_0(z) =1 - z + \tfrac{z^2}{2!} - \ldots + \tfrac{z^p}{p!} + \mathcal{O}(z^{p+1}),
	\]
	i.\,e., the stability function $R_0(z)$ is a rational approximation to $e^{-z}$ of order $p$.
\end{theorem}
The importance of stiffly accurate methods for very stiff ODEs is discussed in Section 2.3 
of Kennedy and Carpenter\cite{CAKennedy_MACarpenter_2016}.

\subsection{Parareal and MGRIT}\label{sub:parareal_mgrit}

Parareal \cite{JLLions_etal_2001a} can be interpreted in a variety of frameworks of numerical methods \cite{MJGander_SVandewalle_2007}.
In this paper, we consider Parareal as a two-level multigrid scheme and, thus, as being a method within the same broader family as MGRIT
\cite{RDFalgout_etal_2014}. More precisely, Parareal and MGRIT are reduction-based time-multigrid algorithms for solving time-dependent
problems. A ``reduction-based'' method attempts to reduce the solving of one problem to equivalently solving two smaller problems.
Reduction-based multigrid methods are iterative solvers that consist of two parts: relaxation and coarse-grid correction, which are, in
the spirit of reduction, designed to be complementary in reducing error associated with different degrees of freedom (DoFs). Applying
this idea in the time domain, we combine local time stepping on the discretized temporal domain, the fine grid, for a relaxation scheme,
with time stepping on a coarse temporal mesh that uses a larger time step for the coarse-grid correction. More precisely, consider the
model problem \eqref{eq:problem} on a discrete time grid consisting of $N\in\mathbb{N}$ time intervals with a time-step size $h_t = T/N$.
A coarse temporal grid is derived from this fine grid by considering only every $k$-th temporal point, where $k>1$ is an integer called the
coarsening factor. Thus, the coarse temporal grid consists of $N_c=N/k$ time intervals with time-step size $H_t = kh_t$. Using multigrid
terminology, the time points on the coarse grid define the set of C-points and the remaining temporal points are called F-points, as
visualized in Figure \ref{fig:time_grids}.

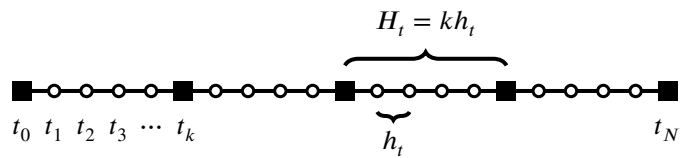
\begin{figure}[h!t]
	\centering
	\begin{tikzpicture}[scale=.85]
		\draw[line width=1.5pt] (0,0) -- (10,0);
		\foreach \i in {0,.5,1,...,10}{
			\fill[color=white] (\i,0) circle (2.5pt);
			\draw[line width=1.15pt] (\i,0) circle (2.5pt);
		}
		\foreach \i in {0,2.5,5,7.5,10}{
			\fill (\i-.15,.15) rectangle (\i+.15,-.15);
		}
		\foreach \i/\n in {0/0,.5/1,1/2,1.5/3}{
			\draw (\i,-.6) node {$t_\n$};
		}
		\draw (2,-.6) node {$\cdots$};
		\draw (2.55,-.6) node {$t_k$};
		\draw (10,-.6) node {$t_N$};
			
		\draw[line width=1.5pt,,decorate,decoration={brace,amplitude=3pt},yshift=5pt] (6,-.5) -- (5.5,-.5) node[below=3pt,midway] {$h_t$};
		\draw[line width=1.5pt,,decorate,decoration={brace,amplitude=5pt},yshift=-3pt] (5,.5) -- (7.5,.5) node[above=7pt,midway] {$H_t = kh_t$};
	\end{tikzpicture}
	\caption{Fine and coarse temporal grids; squares ${\scriptstyle\blacksquare}$ represent C-points and F-points are marked by circles $\circ$.\label{fig:time_grids}}
\end{figure}

Parareal uses F-relaxation, which refers to an exact local solve on F-points, i.\,e., updating the DoFs at F-points by time-stepping the current
solution values from each C-point across a coarse-scale time interval to the following $k-1$ F-points. After F-relaxation, the residual is zero
at F-points and the local (parallel) fine-grid time stepping is coupled with a global (sequential) time stepping on the coarse grid. More precisely,
the residual is evaluated at C-points and restricted by value to the coarse grid. Then, the coarse-grid DoFs are updated by time-stepping using
a coarse-grid time propagator. After the coarse-grid time stepping, a correction is interpolated to the fine grid, using the so-called ideal interpolation
operator, defined by taking the corrected approximate solution values at each C-point and integrating the values forward in time over the following $k-1$ F-points.

MGRIT uses the same restriction and interpolation operators as Parareal, with an additional option for relaxation. MGRIT augments F-relaxation
by using also C-relaxation, which refers to an exact solve on C-points. More precisely, MGRIT typically uses FCF-relaxation, consisting of
F-relaxation, followed by C-relaxation, and again F-relaxation. While this relaxation scheme is necessary to obtain a scalable multilevel
method, it also improves convergence of the two-level scheme and it can be important for convergence on difficult problems.

In this paper, we apply two-level versions of Parareal and MGRIT to the model problem \eqref{eq:problem}. Thus, only the two-level linear
setting is considered; multilevel schemes result from applying the two-level methods recursively and the full approximation storage (FAS)
\cite{ABrandt_1977b} approach can be used to accommodate nonlinear problems. For the fine- and coarse-grid time stepping schemes,
we consider $s$-stage Runge-Kutta methods with the time propagator \eqref{eq:phi},
\[
	\Phi (h_t) := I - h_t\mathbf{b}_0^T\otimes I \left(I + h_t A_0\otimes \mathcal{L}\right)^{-1} (\mathbf{1}_s \otimes \mathcal{L}),
\]
on the fine grid, and (possibly different) $s$-stage Runge-Kutta methods with coarse time propagator
\begin{align}\label{eq:psi}
	\Psi(kh_t) := I - kh_t\mathbf{b}_1^T\otimes I \left(I + k h_t A_1\otimes \mathcal{L}\right)^{-1} (\mathbf{1}_s \otimes \mathcal{L}),
\end{align}
where $A_1$ and $\mathbf{b}_1$ denote the Runge-Kutta matrix and the weight vector of the Runge-Kutta method on the coarse grid, respectively. 
Note that the time-stepping operator $\Psi$ is of a similar form as $\Phi$ and can also be diagonalized by the eigenvectors of $\mathcal{L}$.
If $\xi$ is an eigenvalue of $\mathcal{L}$, the corresponding eigenvalue of $\Psi$ is given by
\begin{align}\label{eq:mu}
	\mu(\xi) = 1 - kh_t\xi\mathbf{b}_1^T(I + kh_t\xi A_1)^{-1}\mathbf{1}.
\end{align}

\subsection{Two-level convergence and some observations}\label{sub:parareal_mgrit_conv}

Convergence of Parareal and two-level MGRIT have been considered in a number of papers\cite{GBal_2005,MJGander_SVandewalle_2007,MJGander_EHairer_2008,SFriedhoff_SMacLachlan_2015a,VADobrev_etal_2017,BSSouthworth_2018a,MJGander_etal_2018}. For the most part,
previous works derived sufficient conditions and approximate upper bounds on convergence. Recently, under certain assumptions,
\textit{necessary} and sufficient conditions for convergence were derived, including tight bounds on two-grid convergence factors. 
These bounds provide the basis for this work. To review, let $\Phi$ denote the fine-grid time propagation operator, $\Psi$ denote
the coarse-grid time-propagation operator, and assume that:
\begin{enumerate}
  \setlength\itemsep{0ex}
\item $\Phi$ and $\Psi$ are linear and independent of time,
\item $\Phi$ and $\Psi$ commute,
\item $\Phi$ and $\Psi$ are (simultaneously) diagonalizable, with eigenvectors $U$. 
\end{enumerate}
The first assumption includes most linear PDEs that are discretized using a method-of-lines approach, among other problems.
The second assumption holds nearly always if the same spatial discretization (or underlying problem being propagated through
time), say $\mathcal{L}$, is used on the fine and coarse grid. If a different coarse-grid spatial discretization is used, this assumption is
less likely to hold. Finally, given assumption two, the third assumption holds if $\Phi$ and $\Psi$ are diagonalizable, which is
typically equivalent to $\mathcal{L}$ being diagonalizable. This holds for most parabolic-type PDEs, and some hyperbolic as
well. 

Under these assumptions, the following theorem gives exact bounds on convergence in the $(UU^*)^{-1}$-norm based
on the relation of eigenmodes of $\Phi$ and $\Psi$. Note, if $\Phi$ and $\Psi$ are unitarily diaognalizable, such as when
$\mathcal{L}$ is symmetric positive definite (SPD) or, more generally, when $\mathcal{L}$ is normal, then the $(UU^*)^{-1}$-norm
is simply the $\ell^2$-norm. 

\begin{theorem}[Tight two-grid bound (Southworth\cite{BSSouthworth_2018a} Theorem 30)]\label{th:diag_tight}
Assume $\Phi$ and $\Psi$ commute and are diagonalizable, with eigenvectors $U$, and eigenvalues $\{\lambda\}$ and
$\{\mu\}$, respectively. Assume we coarsen in time by a factor of $k$, take $N_c$ time steps on the coarse grid, and let
$\mathcal{E}_F$ and $\mathcal{E}_{FCF}$ denote error propagation of two-level MGRIT with F-relaxation and
FCF-relaxation, respectively. Then, for iterations $i>1$,
\begin{align}\label{eq:v_bounds}
\begin{split}
\sup_i \frac{|\mu_i - \lambda_i^k|}{\sqrt{ (1 - |\mu_i|)^2 + \frac{\pi^2|\mu_i|}{N_c^2}}}
	& \leq \|\mathcal{E}_F\|_{(UU^*)^{-1}}
	\leq \sup_i  \frac{|\mu_i - \lambda_i^k|}{\sqrt{(1 - |\mu_i|)^2 + \frac{\pi^2|\mu_i|}{6N_c^2}} }, \\
\sup_i \frac{|\lambda_i^k||\mu_i - \lambda_i^k|}{\sqrt{ (1 - |\mu_i|)^2 + \frac{\pi^2|\mu_i|}{N_c^2}}}
	& \leq \|\mathcal{E}_{FCF}\|_{(UU^*)^{-1}}
	\leq \sup_i  \frac{|\lambda_i^k||\mu_i - \lambda_i^k|}{\sqrt{(1 - |\mu_i|)^2 + \frac{\pi^2|\mu_i|}{6N_c^2}} }.
\end{split}
\end{align}
\end{theorem}
The qualifier for iterations $i>1$ is a subtle detail which is discussed in detail in \cite{BSSouthworth_2018a}. In short, the first iteration
may have a slightly larger (but still bounded) convergence factor. 
These bounds can be relaxed for the simpler expressions
\begin{align}\label{eq:v_bounds2}
\begin{split}
\sup_i \frac{|\mu_i - \lambda_i^k|}{1 - |\mu_i| + \mathcal{O}(1/N_c))}
	& < \|\mathcal{E}_F\|_{(UU^*)^{-1}}
	< \sup_i  \frac{|\mu_i - \lambda_i^k|}{1 - |\mu_i| }, \\
\sup_i \frac{|\lambda_i^k||\mu_i - \lambda_i^k|}{1 - |\mu_i| + \mathcal{O}(1/N_c)}
	& < \|\mathcal{E}_{FCF}\|_{(UU^*)^{-1}}
	< \sup_i  \frac{|\lambda_i^k||\mu_i - \lambda_i^k|}{1 - |\mu_i| }.
\end{split}
\end{align}
The only place of significant difference between bounds in \eqref{eq:v_bounds} and \eqref{eq:v_bounds2}
is when $|\mu_i| \to 1$. However, when $|\mu_i| \to 1$, it is also the case that $|\lambda_i|\to 1$, and Corollary
\ref{cor:limits} proves below, when $\Phi$ and $\Psi$ are derived from Runge-Kutta schemes,
\begin{align*}
\lim_{|\mu_i|,|\lambda_i| \to 1} \frac{|\lambda_i^k||\mu_i - \lambda_i^k|}{1 - |\mu_i| } = 0.
\end{align*}
Furthermore, bounds on $\mathcal{E}_F$ and $\mathcal{E}_{FCF}$ are a supremum over $i$, and it
turns out the maximum is not obtained when $|\mu_i|$ is close to one for any of the many Runge-Kutta
schemes tested here (or that we have tested, but not included in this paper).
For this reason, results in this paper are based on the simpler bounds $\varphi_F$ and $\varphi_{FCF}$,
where (Southworth\cite{BSSouthworth_2018a} Theorems 12 and 13)
\begin{align}\label{eq:varphi}
\begin{split}
\|\mathcal{E}_F\|^2 & = \varphi_F^2 - \mathcal{O}(1/N_c^2) < \varphi_F^2 \hspace{9.5ex}\textnormal{where}\hspace{2ex}
\varphi_F := \sup_i  \frac{|\mu_i - \lambda_i^k|}{1 - |\mu_i| }, \\
\|\mathcal{E}_{FCF}\|^2 & = \varphi_{FCF}^2 - \mathcal{O}(1/N_c^2) < \varphi_{FCF}^2 \hspace{5ex}\textnormal{where}\hspace{2ex}
\varphi_{FCF} := \sup_i  \frac{|\lambda_i^k||\mu_i - \lambda_i^k|}{1 - |\mu_i| },
\end{split}
\end{align}
which eliminates the need to to pick a specific $N_c$ for bounds. 
Note, the equality with perturbation $\mathcal{O}(1/N_c^2)$ holds only when $|\mu_1| \not\approx 1$, which,
as mentioned above, is a valid assumption when taking the supremum over eigenvalues for Runge-Kutta schemes. 

The following lemma first introduces a relatively general result on $\Psi - \Phi^k$ as a rational function of the
spatial discretization $\mathcal{L}$ and time step $h_t$, related to the limiting case of $|\mu_i|,|\lambda_i|\to 1$.

\begin{lemma}\label{lem:psi_phi}
Let $\mathcal{L}$ be a nonsingular operator and $\Phi$ be a fine-grid time-propagation scheme, corresponding to a Runge-Kutta
scheme of order $p_1$, applied to $\mathcal{L}$ with time step $h_t$. Assume we coarsen in time by a factor of $k$,
and let $\Psi$ denote the coarse-grid time-propagation operator, corresponding to a Runge-Kutta scheme of order $p_2$
with time step $kh_t$. Then, $\Psi - \Phi^k$ is a rational function of $h_t\mathcal{L}$, of \textit{minimal} polynomial degree
\begin{align*}
\Psi - \Phi^k & = \mathcal{O}\left((h_t\mathcal{L})^{\min\{p_1,p_2\}+1}\right).
\end{align*}
\end{lemma}
\begin{proof}
From Theorem \ref{th:exp_p}, the stability function of any Runge-Kutta scheme of order $p$ can be written in the form
$R_0(z) := e^{z} + Cz^{p+1} + \mathcal{O}(z^{p+2})$, with error constant $C$. This result extends to operators $\Phi$
and $\Psi$, where, for any Runge-Kutta scheme of order $p$,
\begin{align*}
\Phi & = e^{-h_t\mathcal{L}} + \mathcal{O}((h_t\mathcal{L})^{p+1}),
\end{align*}
and a similar expression for $\Psi$. Thus, let $\Phi$ and $\Psi$ correspond to arbitrary Runge-Kutta schemes of orders $p_1$ and $p_2$ with time steps
$h_t$ and $kh_t$, respectively. Then,
\begin{align*}
\Phi^k - \Psi & = \left(e^{-h_t\mathcal{L}} + \sum_{\ell=p_1+1}^\infty \alpha_\ell (-h_t \mathcal{L})^\ell \right)^k - 
	\left(e^{-kh_t\mathcal{L}} + \sum_{\ell=p_2+1}^\infty \beta_\ell (-kh_t \mathcal{L})^\ell \right),
\end{align*}
for coefficients $\{\alpha_\ell\}$ and $\{\beta_\ell\}$ describing the error in the approximation of the exponential. Note, these 
coefficients may be rational functions of $h_t\mathcal{L}$, but it must be that
$\alpha_\ell(-h_t\mathcal{L})^\ell = \mathcal{O}((h_t\mathcal{L})^{q})$, for some $q\geq p+1$, and similarly for $\{\beta_\ell\}$. Applying
the binomial theorem to expand the power in $k$ yields
\begin{align}
\Phi^k - \Psi & = \sum_{j=0}^k \binom{k}{j} e^{-(k-j)h_t\mathcal{L}} \left[\sum_{\ell=p_1+1}^\infty \alpha_\ell (-h_t \mathcal{L})^\ell\right]^j -
	\left(e^{-kh_t\mathcal{L}} + \sum_{\ell=p_2+1}^\infty \beta_\ell (-kh_t \mathcal{L})^\ell \right) \\
& = \sum_{j=1}^k \binom{k}{j} e^{-(k-j)h_t\mathcal{L}} \left[\sum_{\ell=p_1+1}^\infty \alpha_\ell (-h_t \mathcal{L})^\ell\right]^j
	- \sum_{\ell=p_2+1}^\infty \beta_\ell (-kh_t \mathcal{L})^\ell  \\
& = \mathcal{O}\left((h_t\mathcal{L})^{\min\{p_1,p_2\}+1}\right).
\end{align}
\end{proof}

Lemma \ref{lem:psi_phi} is a natural extension of the well-known Runge-Kutta result that for a method $\Phi$ of
order $p$, $e^{\mathcal{L}h_t} = \Phi + \mathcal{O}((\mathcal{L}h_t)^{p+1})$. However, it does lead to a nice
Corollary that for arbitrary Runge-Kutta schemes and coarsening factors, $\Psi$ approximates $\Phi^k$ exactly
as $h_t\xi\to 0$ .

\begin{corollary}\label{cor:limits}
Let $\Phi$ and $\Psi$ correspond to arbitrary Runge-Kutta schemes to propagate the same diagonalizable
operator, $\mathcal{L}$, with real-valued eigenvalues, $\xi$. Then,
\begin{align*}
\lim_{h_t\xi\to 0^+} \frac{|\mu(h_t\xi) - \lambda(h_t\xi)^k|}{1 - |\mu(h_t\xi)|}  & = 0, \hspace{5ex}
\lim_{h_t\xi\to 0^+} \frac{|\lambda(h_t\xi)^k||\mu(h_t\xi) - \lambda(h_t\xi)^k|}{1 - |\mu(h_t\xi)|} = 0.
\end{align*} 
Equivalently, these limits apply when $\lim_{h_t\xi\to 0^+} \Longleftrightarrow \lim_{\mu_i,\lambda_i \to 1^-}$.
\end{corollary}
\begin{proof}
\revision{
Note that $\lim_{h_t\xi\to0^+} \mu(h_t\xi) = 1$ and, by continuity of $\mu$ within the region of stability and
because $\mu$ is a rational function, $|\mu(h_t\xi)|$ is
differentiable in some open interval $h_t\xi \in (0,\epsilon)$. Furthermore, from Theorem \ref{th:exp_p} and 
\eqref{eq:mu}, $\lim_{h_t\xi\to0^+} \frac{\partial}{\partial (h_t\xi)} \mu(h_t\xi) = -k \neq 0$.
From Lemma \ref{lem:psi_phi}, $|\mu(h_t\xi) - \lambda(h_t\xi)^k|$ is the absolute value of a rational function in
$h_t\xi$, with minimum polynomial degree $\geq \min\{p_1,p_2\}+1 := {d}$. Then,
\begin{align*}
|\mu(h_t\xi) - \lambda(h_t\xi)^k| = (h_t\xi)^2 |\mathcal{R}_{d-2}(h_t\xi)|,
\end{align*}
where $\mathcal{R}_{d-2}(h_t\xi)$ is a rational function of minimum polynomial degree $d-2 \geq 0$. 
A rational function is differentiable at all points except those with zero denominator. Because
$\lim_{h_t\xi\to0^+} \mu(h_t\xi), \lambda(h_t\xi)^k = 1$, there cannot be a singularity in 
$|\mu(h_t\xi) - \lambda(h_t\xi)^k|$ at $h_t\xi = 0$. Thus, $(h_t\xi)^2 |\mathcal{R}_{d-2}(h_t\xi)|$ is
continuous in some open interval $(-\epsilon,\epsilon)$ and differentiable except at a potential
sign change in $\mathcal{R}_{d-2}(h_t\xi)$. Noting that $\mathcal{R}_{d-2}(h_t\xi)$ has a fixed and
finite number of roots, there exists some open interval $(0,\hat{\epsilon})$ where
$(h_t\xi)^2 |\mathcal{R}_{d-2}(h_t\xi)|$ has fixed sign and is differentiable. 
Without loss of generality, assume $\mathcal{R}_{d-2}(h_t\xi) \geq 0$ for $h_t\xi\in(0,\hat\epsilon)$. Then
$\frac{\partial}{\partial (h_t\xi)} |\mu(h_t\xi) - \lambda(h_t\xi)^k| = 2h_t\xi\mathcal{R}_{d-2}(h_t\xi) +
h_t\xi\frac{\partial}{\partial (h_t\xi)}\mathcal{R}_{d-2}(h_t\xi)$ and $\lim_{h_t\xi\to0^+} 
\frac{\partial}{\partial (h_t\xi)} |\mu(h_t\xi) - \lambda(h_t\xi)^k|  = 0$. Applying L'H\^opital's rule to
the one-sided limit completes the proof.
}
\end{proof}

\begin{remark}[Complex eigenvalues]\label{rem:complex}
\revision{In the complex-valued case, limits and derivatives require a holomorphic function to be well-defined, which we
do not have. A variation on Corollary \ref{cor:limits} holds in the complex eigenvalue case by simply
including the square root and tight $\mathcal{O}(1/N_c^2)$ term in \eqref{eq:v_bounds}, and evaluating explicitly
at $h_t \xi = 0$. However, the lack of a well-defined limit has subtle but important implications, which are
discussed in Section \ref{sec:imag}.}
\end{remark}

Finally, we introduce the following simple proposition regarding large time steps and L-stable coarse-grid integration
schemes.

\begin{proposition}[L-stable schemes]\label{prop:L}
Let $\Psi$ correspond to an L-stable Runge-Kutta scheme and $\Phi$ an A-stable Runge-Kutta scheme. Then,
\begin{align*}
\lim_{h_t\xi\to \infty} \frac{|\mu(h_t\xi) - \lambda(h_t\xi)^k|}{1 - |\mu(h_t\xi)|} \leq 1, \hspace{5ex}
\lim_{h_t\xi\to \infty} \frac{|\lambda(h_t\xi)^k||\mu(h_t\xi) - \lambda(h_t\xi)^k|}{1 - |\mu(h_t\xi)|} \leq \lim_{h_t\xi\to \infty} |\lambda(h_t\xi)^k| \leq 1.
\end{align*} 
Now, suppose $\Psi$ and $\Phi$ correspond to L-stable Runge-Kutta schemes. Then,
\begin{align*}
\lim_{h_t\xi\to \infty} \frac{|\mu(h_t\xi) - \lambda(h_t\xi)^k|}{1 - |\mu(h_t\xi)|} = 
	\lim_{h_t\xi\to \infty} \frac{|\lambda(h_t\xi)^k||\mu(h_t\xi) - \lambda(h_t\xi)^k|}{1 - |\mu(h_t\xi)|} = 0.
\end{align*}
\end{proposition}
\begin{proof}
The proofs follow immediately from the definition of A-stability and L-stability (Definition \ref{def:stability}).
\end{proof}

The limit of $h_t\xi \to \infty$ is relevant for large time steps relative to the spatial mesh. Proposition \ref{prop:L}
proves that using an L-stable coarse-grid operator $\Psi$ with an A-stable \textit{or} L-stable fine-grid operator $\Phi$
cannot result in divergence for very large time steps. Moreover, using L-stable integration schemes for both $\Psi$
and $\Phi$ results in a method that is guaranteed to converge rapidly for large $h_t\xi$. Combining with Corollary
\ref{cor:limits}, which guarantees rapid convergence for small $h_t\xi$, suggests that L-stable schemes are particularly
amenable to Parareal and MGRIT, because they are guaranteed to converge rapidly at both limits of small and large
$h_t\xi$. This is confirmed (theoretically and numerically) in Section \ref{sec:irk}, where L-stable schemes typically
provide the best convergence factors.

\begin{remark}
In fact, there may be conceptual intuition for why L-stable coarse-grid operators are beneficial -- from Hairer\cite{EHairer_GWanner_2010},
we know that A-stable schemes (may) dampen stiff components very slowly. However, we do not rely on the
coarse-grid to integrate stiff components. We expect these to be dealt with on the fine grid (because stiff
components require small time steps to be accurately captured), and on the coarse grid, in the spirit of reduction,
we want these modes to be damped quickly.
\end{remark}

\section{Explicit Runge-Kutta}\label{sec:erk}

This section explores explicit Runge-Kutta schemes and the application of Parareal and MGRIT. First, we consider
some theoretical relationships between $\Psi$ and $\Phi^k$, showing that (i) $\Psi$ is an optimal polynomial
approximation to $\Phi^k$ in a Taylor sense, and (ii) in general, $\Psi - \Phi^k$ is invertible
and, thus, $\Psi$ approximates no eigenmodes of $\Phi^k$ exactly. We then consider convergence as a
function of $h_t\xi$ for various two-level schemes. 

As a consequence of Theorem \ref{th:exp_p}, the stability function of an explicit Runge-Kutta scheme of order 
$p=s$ is given by the exponential series truncated at the $z^s$-term\cite{JCButcher_2016},
	\[
		R_0(z) = 1 + z + \tfrac{z^2}{2!} + \ldots + \tfrac{z^s}{s!}.
	\]
This applies for explicit Runge-Kutta schemes of order $p\leq 4$. For $p\geq 5$, no explicit Runge-Kutta
method exists with $s=p$ stages\cite{JCButcher_2016}. The following lemma directly extends this to the matrix-valued case, and
proves that $\Psi$ is an optimal Taylor approximation to $\Phi^k$. 

\begin{lemma}\label{lem:opt}
For explicit Runge-Kutta schemes of less than order 5, or schemes of higher order with a stability function given by a truncated
Taylor series of the exponential, $\Psi$ is an optimal, in a Taylor-sense, approximation to $\Phi^k$. That is, for an $s$-stage method,
$\Phi^k$ is a matrix polynomial, and $\Psi$ is exactly the $s$ lowest-order terms in this polynomial. 
\end{lemma}
\begin{proof}
Recall the matrix exponential can be expressed as a Taylor series, $\e^{-\mathcal{L}} := \sum_{\ell=0}^\infty \frac{(-\mathcal{L})^\ell}{\ell!}$,
and, from above, $\Phi$ corresponding to explicit Runge-Kutta schemes is exactly a truncated matrix exponential for all
schemes of order less than 5. For higher orders, there are methods (but not all) such that this is also the case. 

Thus, let $\Phi$ correspond to a truncated matrix exponential of power $s$, with time step $h_t$,
\begin{align}
\Phi = I - h_t\mathcal{L} + ... \pm \frac{(h_t \mathcal{L})^s}{s!} = \sum_{\ell=0}^s \frac{(-h_t\mathcal{L})^\ell}{\ell!}.\label{eq:phi_series}
\end{align}
Then, consider truncating the matrix exponential with time step $kh_t$ at power $s$, corresponding to $\Psi$.
First, note that breaking the series into two parts and applying the binomial theorem yields
\begin{align}
e^{-kh_t\mathcal{L}} & = \left(e^{-h_t\mathcal{L}}\right)^k \nonumber\\
& = \left( \sum_{\ell=0}^s \frac{(-h_t \mathcal{L})^\ell}{\ell!} + \sum_{\ell=s+1}^\infty \frac{(-h_t \mathcal{L})^\ell}{\ell!} \right)^k \nonumber\\
& = \sum_{j=0}^k \binom{k}{j} \left[\sum_{\ell=0}^s \frac{(-h_t \mathcal{L})^\ell}{\ell!}\right]^{k-j} \left[\sum_{\ell=s+1}^\infty \frac{(-h_t \mathcal{L})^\ell}{\ell!}\right]^j.\label{eq:binom_s}
\end{align}
Notice that all terms of the second summation in \eqref{eq:binom_s} have power $> s$. It follows that when expressing
$e^{-kh_t\mathcal{L}}$ as a Taylor series centered at zero, terms of order $\leq s$ correspond to $j=0$ in \eqref{eq:binom_s},
\begin{align}
\left[\sum_{\ell=0}^s \frac{(-h_t \mathcal{L})^\ell}{\ell!}\right]^k.\label{eq:psi_s}
\end{align}
$\Psi$ is obtained by truncating this series at power $s$. Noticing that \eqref{eq:psi_s} is exactly $\Phi$ \eqref{eq:phi_series} raised
to the power of $k$ completes the proof.
\end{proof}

It is well-known that explicit integration schemes are limited in the parallel-in-time setting by the coarse-grid stability constraint,
often making them not immediately practical. However, it is interesting to see from Lemma \ref{lem:opt} that using the same
explicit discretization on the coarse grid is actually optimal in some sense. The following Corollary develops conditions under
which $\Psi - \Phi^k$ is invertible for explicit schemes, an assumption necessary for some of the results in \cite{BSSouthworth_2018a}
and also one that provides insight into how $\Psi$ approximates $\Phi^k$. In particular, it is unlikely that any eigenmode is
approximated exactly. 

\begin{corollary}\label{cor:invert}
Let $\Phi$ correspond to an explicit Runge-Kutta scheme of less than order 5, or of higher order with a stability function given
by a truncated Taylor series of the exponential, with $s$ stages and time step $h_t$. Let $\Psi$ correspond to the same
Runge-Kutta scheme, with time step $kh_t$. \revision{Then $\Psi - \Phi^k$ is not invertible if and only if
\begin{align}
(s!)^{k-1}\Gamma(s+1, kh_t\xi) - \Gamma(s+1, h_t\xi)^k & = 0, \label{eq:gamma}
\end{align}
for some eigenvalue $\xi$ of $\mathcal{L}$}, and where $\Gamma(n,x)$ denotes the incomplete Gamma function.
\end{corollary}
\begin{proof}
$\Psi - \Phi^k$ is not invertible if and only if it has a zero eigenvalue, and a zero eigenvalue $\lambda_0$ would correspond to a root
of the characteristic polynomial of $\Psi - \Phi^k$. Noting that eigenvalues of $\Phi$ and $\Psi$ are a function of eigenvalues of
$\mathcal{L}$, let $\{\xi\}$ denote eigenvalues of $\mathcal{L}$. Then, eigenvalues of $\Psi - \Phi^k$ are given by
\begin{align}
p(\xi) := \sum_{\ell=0}^s \frac{(-kh_t\xi)^\ell}{\ell!} -  \left[\sum_{\ell=0}^s \frac{(-h_t \xi)^\ell}{\ell!}\right]^k.\label{eq:p_xi}
\end{align}
Recognize the partial sum 
\begin{align}
\sum_{\ell=0}^s \frac{(-kh_t\xi)^\ell}{\ell!} = \frac{e^{-kh_tv}}{s!}\Gamma(s+1, -kh_t\xi)\label{eq:incomp}
\end{align}
as a scaling of the incomplete Gamma function. \revision{Plugging \eqref{eq:incomp} into \eqref{eq:p_xi} and setting $p(\xi) = 0$,
corresponding to zero eigenvalue of $\Psi - \Phi^k$, a zero eigenvalue exists if and only if}
\begin{align*}
(s!)^{k-1}\Gamma(s+1, -kh_t\xi) - \Gamma(s+1, -h_t\xi)^k & = 0.
\end{align*}
\end{proof}

\begin{remark}
As stated, Corollary \ref{cor:invert} is not immediately intuitive. However, it is straightforward to plug values of $s$ and $k$ into
\eqref{eq:gamma} and find the roots using symbolic software such as Matlab or Mathematica. Recall under the assumption that
$\Phi$ and $\Psi$ are stable, we must have $|\mu(kh_t\xi)|, |\lambda(h_t\xi)| < 1$. No combination of $k\geq 2$ and $1 \leq s \leq 4$
that we have tested has yielded a zero eigenvalue of $\Psi - \Phi^k$ for stable $h_t\xi$. Looking at the structure of roots of $\Psi-\Phi^k$
in the complex plane suggests there may be an analytical proof of this in general, but we have not identified it.
\end{remark}

Runge-Kutta time integration schemes have an extensive underlying mathematical theory, and above we have begun
to make some connections with parallel-in-time solvers. However, although the analysis of explicit Runge-Kutta schemes
can be tractable, effective parallel-in-time algorithms are generally more difficult to develop. Figure \ref{fig:erkerk}
demonstrates why, plotting eigenvalues $\lambda^k$ and $\mu_k$ as a function of $h_t\xi$, and the corresponding
convergence factor of two-level MGRIT with FCF-relaxation for each eigenmode. As expected, convergence (on the
fine-grid problem) is confined to the coarse-grid stability constraint.

\begin{figure}[!h]
  \centering
  \begin{subfigure}[b]{0.33\textwidth}
    \includegraphics[width=\textwidth]{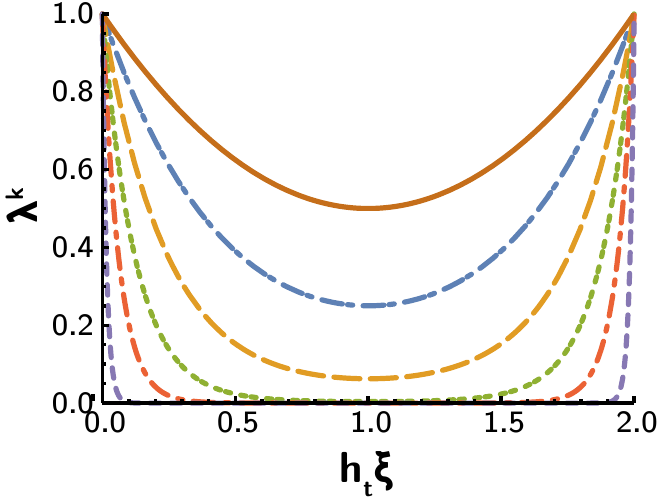}
  \end{subfigure}
   \begin{subfigure}[b]{0.33\textwidth}
    \includegraphics[width=\textwidth]{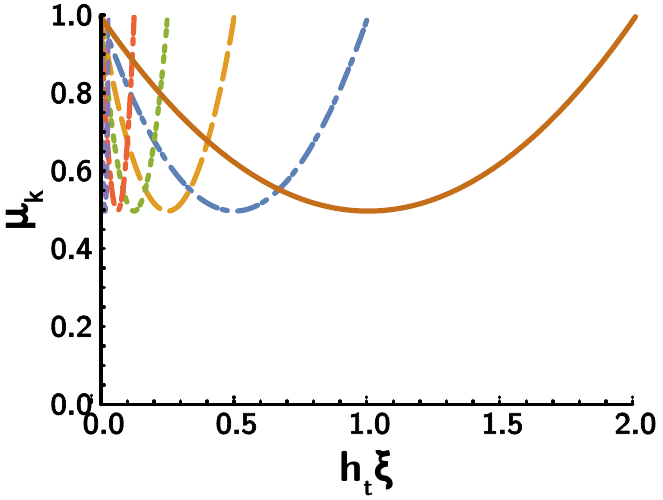}
  \end{subfigure}
  \begin{subfigure}[b]{0.33\textwidth}
    \includegraphics[width=\textwidth]{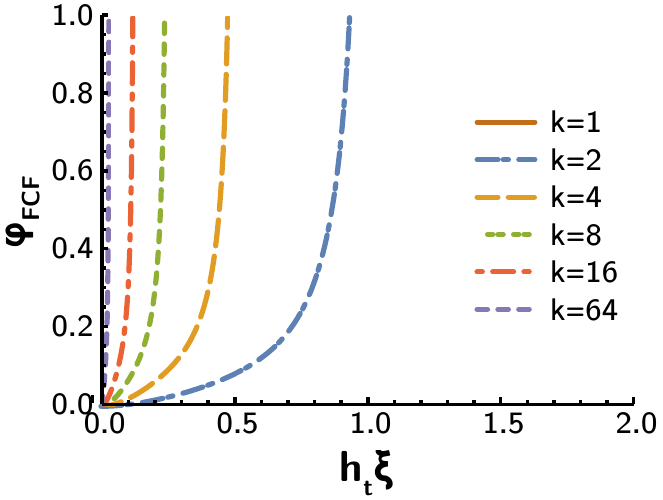}
  \end{subfigure}
\caption{2nd-order explicit Runge Kutta on the coarse and fine grid, for coarsening factors $k\in\{2,4,8,16,64\}$.}
  \label{fig:erkerk}
\end{figure}

An alternative option to achieve Parareal/MGRIT convergence using explicit Runge-Kutta schemes and time steps
approaching the \textit{fine-grid} stability limit is to use an implicit coarse-grid integration scheme\cite{RDFalgout_etal_2014}.
\revision{Figure \ref{fig:erkirk} demonstrates that this is reasonably effective in terms of convergence factor, as long as
a moderate or larger coarsening factor is used, indicating that parallel-in-time methods can be developed that are
convergent for explicit fine-grid schemes. Convergence is best for larger coarsening factors, but near the stability limit
convergence degrades rapidly.} This behavior results from the existence of a non-physical component in the discrete
solution that is oscillatory in time, as discussed in more detail in \cite{RDFalgout_etal_2014}. Since the existence of
this component is not physical, this is not a real restriction for using parallel-in-time methods. \revision{Note that although
it is more difficult to analyze, in practice it is likely best to do implicit coarse-grid solves approximately (for example,
one spatial multigrid V-cycle \cite{RDFalgout_etal_2014}), which makes the cost of implicit coarse-grid solves reasonable
in practice.}

\begin{figure}[!h]
  \centering
  \begin{subfigure}[b]{0.33\textwidth}
    \includegraphics[width=\textwidth]{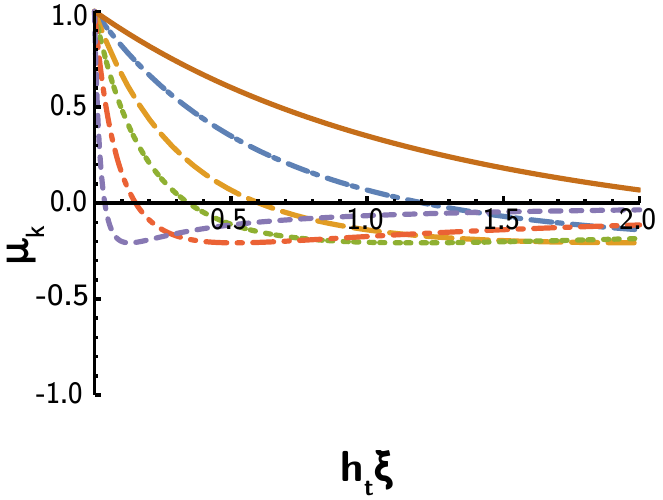}
  \end{subfigure}
   \begin{subfigure}[b]{0.33\textwidth}
    \includegraphics[width=\textwidth]{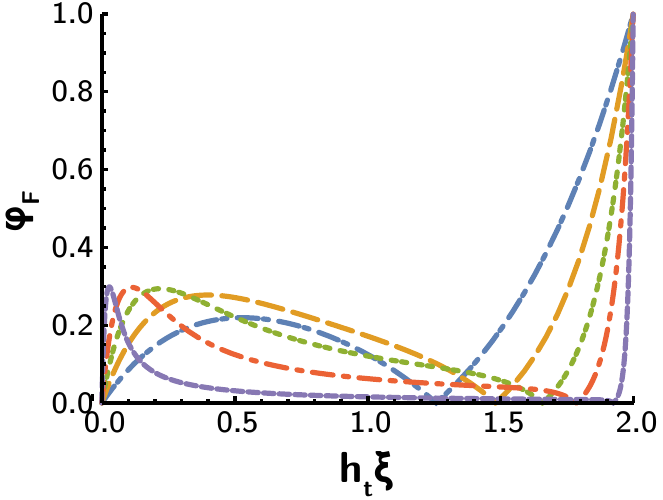}
  \end{subfigure}
  \begin{subfigure}[b]{0.33\textwidth}
    \includegraphics[width=\textwidth]{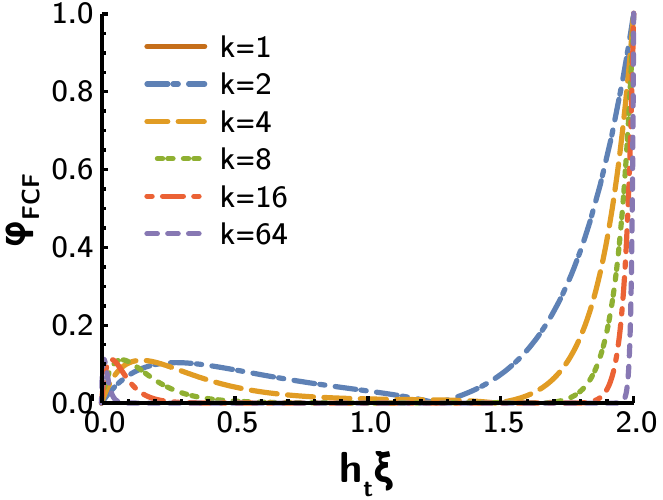}
  \end{subfigure}
\caption{2nd-order explicit Runge Kutta on the fine grid and backward Euler on the coarse grid, for coarsening factors $k\in\{2,4,8,16,64\}$.
Eigenvalues of $\Psi$ are shown on the left as a function of $k$ (to be compared with eigenvalues $\lambda^k$ in Figure \ref{fig:erkerk}),
followed by $\varphi_F$ and $\varphi_{FCF}$ as a function of $h_t\xi$.}
\label{fig:erkirk}
\end{figure}

\section{Implicit Runge-Kutta}\label{sec:irk}

Thus far, we have primarily introduced analysis relating convergence of Parareal and MGRIT to theory of Runge-Kutta methods. This section
gets into more practical information, studying convergence of Parareal/MGRIT applied to implicit integration schemes, a regime where
Parareal and MGRIT are generally more effective than explicit. In particular, for SPD spatial discretizations, we are interested
in achieving ``optimal'' Parareal and two-level MGRIT convergence, with convergence factors independent of time-step size, $h_t$, 
and spatial eigenvalue $\xi \geq 0$ (or, equivalently, independent of spatial mesh size). Here, we limit ourselves to
DIRK-type methods, as these are typically much more practical for PDEs than fully implicit Runge-Kutta methods,
but the methods of analysis used here easily extend to the general case as well.

One important point to draw from results in this section is that understanding the problem one is interested in solving
is critical to choosing an effective Parareal/MGRIT scheme. The first question is simply what range will $h_t\xi$ take? Often
$h_t\xi$ may not get substantially larger than one, in which case ``optimal'' convergence for arbitrarily large $h_t\xi$
may be less of a concern. However, as we will see in Section \ref{sec:irk:L}, even for $\mathcal{O}(h_t\xi) = 1$,
small changes in $h_t\xi$ can quickly move from very rapid convergence to divergence. In other cases, such as
using high-order time integration with a low-order spatial discretization or simply taking very large time steps for
some other reason, the case of $h_t\xi \gg 1$ may be of interest, in which case the optimal convergence is particularly
relevant. It is also worth pointing out that when considering full multilevel MGRIT, the coarse-grids take progressively
larger time steps, and two-grid convergence for $h_t\xi \gg1$ is likely a necessary condition for multilevel convergence,
although a formal connection between two-level and multilevel convergence is not yet understood \cite{mgrit19},
and, as will be seen in Section \ref{sec:ex}, the relationship between two-level and multilevel is not always so simple. 

Tight Parareal and MGRIT convergence bounds are provided for all $h_t\xi>0$ and various implicit Runge-Kutta schemes in
Section \ref{sec:irk:irk}. Sections \ref{sec:irk:L} and \ref{sec:irk:weight} proceed to apply the same analysis tools
to recently developed modified Parareal algorithms, showing the versatility of these bounds. Butcher tableaux for
the SDIRK Runge-Kutta schemes considered here can be found in Appendix \ref{app:sdirk}, Tables \ref{tab:1stage},
\ref{tab:2stage}, and \ref{tab:3stage}, and tableaux for the ESDIRK schemes can be found in Appendix \ref{app:esdirk},
Table \ref{tab:esdirk}.

\subsection{Parareal and MGRIT convergence bounds}\label{sec:irk:irk}

A number of papers have looked at optimal convergence of Parareal for specific implicit time-stepping schemes. The first 
results were from Mathew et al.\cite{TPMathew_etal_2010}, who proved $h$-independent convergence for Parareal with
 two combinations of coarse- and fine-grid integration schemes, independent of $k$. 
Wu\cite{SLWu_2015} expanded on these results, proving optimal convergence for $\Psi$ given by backward Euler, and
$\Phi$ given by either SDIRK22 or TR/BDF2$(\gamma), \gamma\in[0.043,0.977]$, an A-stable (L-stable if $\gamma = 2-\sqrt{2}\cite{Bank_etal_1985})$ scheme combining trapezoid rule (TR) and BDF2 \cite{Bank_etal_1985,Hosea_Shampine_1996}.
Each of these proofs relied on the L-stable nature of the integration schemes, which, as we saw in 
Proposition \ref{prop:L}, offers unique benefits for very large time steps. \revision{Wu\cite{SLWu_2015} also proved that, \textit{using
convergence bounds from} \cite{MJGander_SVandewalle_2007}, convergence cannot be obtained for all $h_t\xi$ when
$\Phi$ corresponds to the trapezoid rule and $\Psi$ backward Euler. However, it should be pointed out that,
because said bounds were not shown to be tight or \textit{necessary} conditions, the proof in \cite{SLWu_2015} does
not rigorously prove that $\Phi$ corresponding to the trapezoid rule and $\Psi$ to backward Euler does not converge
independent of $h_t$.} Finally, Wu and Zhou\cite{SLWu_TZhou_2015} proved $h$-independent convergence assuming
relatively large coarsening factors for three additional fine-grid schemes, Gauss4, a three-stage DIRK method, and TR.

\revision{A comparison of our results with previously derived convergence bounds
\cite{TPMathew_etal_2010,SLWu_2015,SLWu_TZhou_2015} is shown in Table \ref{tab:comp}. For the
comparison, we have computed our bounds for coarsening factors $k\geq 2$ and report the maximum bound over these
coarsening factors. For all L-stable schemes, i.\,e., for backward Euler, TR/BDF2 with $\gamma=2-\sqrt{2}$, and SDIRK22,
results are in (almost) perfect agreement with the literature. For SDIRK23 on the fine grid and backward Euler on the coarse grid,
our analysis results in a maximum convergence factor for even coarsening factors $k\geq 4$ that is slightly better
than the bound reported by Wu and Zhou. Moreover, our bound also applies to odd coarsening factors, with a maximum
convergence factor of 0.392 for all $k\geq 3$. When using Gauss4 on the fine grid, our bound is also slightly tighter
than results in \cite{SLWu_TZhou_2015}, and $\hat{z}_{max}(k)$ can be computed for each coarsening factor $k$.
Finally, for linear problems, Crank-Nicolson integration is equivalent to the implicit midpoint and two-stage RK trapezoidal
methods. In \cite{TPMathew_etal_2010}, a bound of $\overline{\varphi}_F<0.503$ is claimed (but the derivation is not
shown). Table \ref{tab:comp} shows that our bounds disagree with that result. A full analytic derivation for the case
of $k=2$ is given in Lemma \ref{lem:crank} in the Appendix to show that Crank-Nicolson can observe arbitrarily slow
convergence for large $h_t\xi$, suggesting there is an error in \cite{TPMathew_etal_2010}. This is consistent with
statements in \cite{SLWu_2015} stating that the trapezoidal method can converge arbitrarily slowly, although, as
pointed out previously, motivation for such in \cite{SLWu_2015} was not rigorous due to a lack of \textit{necessary}
conditions for convergence. 
\begin{table}[tp]
\renewcommand{\tabcolsep}{0.195cm}
\centering
\begin{tabular}{ |c | c | c | c | c | c | c | }\Xhline{1.25pt}
$\Phi\sim~$ & BWE & CN & TR/BDF2 & SDIRK22 & SDIRK23 & Gauss4 \\ \hline\hline\multirow{2}{*}{$\overline{\varphi}_F$} & \multirow{2}{*}{0.298\cite{TPMathew_etal_2010}} & \multirow{2}{*}{0.503\cite{TPMathew_etal_2010}} &
	\multirow{2}{*}{1/3\cite{SLWu_2015}} & \multirow{2}{*}{0.316\cite{SLWu_2015}} & 1/3\cite{SLWu_TZhou_2015}  & 1/3 ($h_t\xi < z_\text{max}$; \\
& & &
	& & (if $k\geq 4$, even) & $k$ even, large)\cite{SLWu_TZhou_2015} \\ \hline
\multirow{2}{*}{$\varphi_F$} & \multirow{2}{*}{0.298} & \multirow{2}{*}{1} & \multirow{2}{*}{0.316} &  \multirow{2}{*}{0.316} & 0.301 for $k\geq 4$, &  \multirow{2}{*}{0.298 ($h_t\xi < \hat{z}_\text{max}$)} \\
&&&& &0.392 for $k\geq3$ & \\
\Xhline{1.25pt}
\end{tabular} 
\caption{\revision{Convergence bounds $\overline{\varphi}_F$ (reported in \cite{TPMathew_etal_2010,SLWu_2015,SLWu_TZhou_2015}) and maximum convergence factors $\varphi_F$ (where the maximum is taken over all coarsening factors $k\geq 2$) for Parareal/two-level MGRIT with F-relaxation, various fine-grid time-stepping schemes, and $\Psi\sim$ backward Euler. Here, we consider TR/BDF2 with the free parameter
$\gamma=2-\sqrt{2}$, which yields a number of desirable properties \cite{SLWu_2015,Hosea_Shampine_1996} and can be written as
an ESDIRK method. For Gauss4, the true $\hat{z}_\text{max}$ is a function of coarsening factor $k$ -- convergence is bounded $<0.298$ until it begins to asymptote to 1 as $h_t\xi \to \infty$; the value at which this behavior begins and convergence is larger than $0.298$ depends on $k$. }
}
\label{tab:comp}
\end{table}
}

The analysis section for most of those papers was fairly involved. Here, we want to emphasize that results from
\cite{BSSouthworth_2018a} (see Theorem \ref{th:diag_tight}) can easily be used to test such theory for arbitrary
Runge-Kutta schemes on the coarse and fine grid. Table \ref{tab:imp} gives a sample of results for nine different
schemes, using the same integrator on coarse and fine grids, and coarsening factors 2--64. Results include the
maximum convergence factor for F- or FCF-relaxation, the value $h_t\xi$ at which this maximum is obtained, or
the value $h_t\xi$ such that convergence is observed for all lesser values. It is worth pointing out that the algebra
behind deriving bounds that hold for all $k$ as well is still not trivial, but computing bounds for any specific value of $k$
is straightforward. 

\begin{table}[tp]
\renewcommand{\tabcolsep}{0.195cm}
\centering
\begin{tabularx}{0.9\textwidth}{ |c| *{6}{Y|} }\Xhline{1.25pt}
Type & Stages & Order & Stage-order & Stability & Stiff acc.  \\ \hline\hline\hline
SDIRK & 1 & 1 & 1 & L & Yes \\\hline
BWE & \multicolumn{5}{c |}{
\begin{tabular}[t]{@{}c | cccccc@{}}
	$k$ & 2 & 4 & 8 & 16 & 32 & 64 \\ \hline
	$\max \varphi(h_t\xi)$ & 0.13/0.05 & 0.20/0.08 & 0.25/0.10 & 0.27/0.10 & 0.28/0.11 & 0.29/0.11\\
	arg$\max \varphi(h_t\xi)$ & 1/0.33 & 0.48/0.16 & 0.23/0.08 & 0.11/0.04 & 0.06/0.02 & 0.03/0.01
\end{tabular}} \\\hline\hline
SDIRK & 1 & 2 & 2 & A & Yes \\\hline
SDIRK-12 & \multicolumn{5}{c |}{
\begin{tabular}[t]{@{}c | cccccc@{}}
	$k$ & 2 & 4 & 8 & 16 & 32 & 64 \\ \hline
	$\max \varphi(h_t\xi)$ & $\infty$/$\infty$ & $\infty$/$\infty$ & $\infty$/$\infty$ & $\infty$/$\infty$ & $\infty$/$\infty$ & $\infty$/$\infty$  \\
	$h_t\xi$ $|$ $\varphi < 1$ & 2.87/6.35 & 1.50/7.75 & 0.75/10.5 & 0.37/15.5 & 0.18/24.3 & 0.09/39.7
\end{tabular}} \\\hline\hline
ESDIRK & 2$^*$ & 2 & 2 & A & No \\\hline
TR & \multicolumn{5}{c |}{
\begin{tabular}[t]{@{}c | cccccc@{}}
	$k$ & 2 & 4 & 8 & 16 & 32 & 64 \\ \hline
	$\max \varphi(h_t\xi)$ & $\infty$/$\infty$ & $\infty$/$\infty$ & $\infty$/$\infty$ & $\infty$/$\infty$ & $\infty$/$\infty$ & $\infty$/$\infty$  \\
	$h_t\xi$ $|$ $\varphi < 1$ & 2.87/6.36 & 1.50/7.76 & 0.75/10.5 & 0.37/15.5 & 0.19/24.3 & 0.09/39.7 
\end{tabular}} \\\hline\hline
SDIRK & 2 & 2 & 1 & L & Yes  \\\hline
SDIRK-22 & \multicolumn{5}{c |}{
\begin{tabular}[t]{@{}c | cccccc@{}}
	$k$ & 2 & 4 & 8 & 16 & 32 & 64 \\ \hline
	$\max \varphi(h_t\xi)$ & 0.29/0.008 & 0.26/0.01 & 0.26/0.01 & 0.26/0.01 & 0.26/0.01 & 0.26/0.01 \\
	arg$\max \varphi(h_t\xi)$ & 5.0/0.70 & 2.1/0.36 & 1.0/0.17 & 0.51/0.10 & 0.25/0.05 & 0.13/0.02
\end{tabular}} \\\hline\hline
SDIRK & 2 & 3 & 1 & A & No  \\\hline		SDIRK-23 & \multicolumn{5}{c |}{
\begin{tabular}[t]{@{}c | cccccc@{}}
	$k$ & 2 & 4 & 8 & 16 & 32 & 64 \\ \hline
	$\max \varphi(h_t\xi)$ &  ($>$1)/($>$1) &  ($>$1)/($>$1) &  ($>$1)/0.25 &  ($>$1)/0.02 &  ($>$1)/0.013 &  ($>$1)/0.013 \\
	arg$\max \varphi(h_t\xi)$ & $\infty$/$\infty$ & $\infty$/$\infty$ & $\infty$/$\infty$ & $\infty$/$\infty$ & $\infty$/0.05 & $\infty$/0.025 \\
	$h_t\xi$ $|$ $\varphi < 1$ &4.43/17.6 & 2.61/257 & 1.31/NA & 0.65/NA & 0.33/NA & 0.16/NA
\end{tabular}} \\\hline\hline
ESDIRK & 3$^*$ & 2 & 2 & L & Yes  \\\hline		ESDIRK-32 & \multicolumn{5}{c |}{
\begin{tabular}[t]{@{}c | cccccc@{}}
	$k$ & 2 & 4 & 8 & 16 & 32 & 64 \\ \hline
	$\max \varphi(h_t\xi)$ & 0.29/0.008 & 0.26/0.01 & 0.26/0.011 & 0.26/0.011 & 0.26/0.011 & 0.26/0.011 \\
	arg$\max \varphi(h_t\xi)$ & 5.01/0.70 & 2.06/0.36 & 1.02/0.18 & 0.51/0.089 & 0.26/0.045 & 0.13/0.022
\end{tabular}} \\\hline\hline
ESDIRK & 3$^*$ & 3 & 2 & A & Yes  \\\hline		ESDIRK-33 & \multicolumn{5}{c |}{
\begin{tabular}[t]{@{}c | cccccc@{}}
	$k$ & 2 & 4 & 8 & 16 & 32 & 64 \\ \hline
	$\max \varphi(h_t\xi)$ & ($>$1)/($>$1) &  ($>$1)/ ($>$1) &  ($>$1)/0.25 &  ($>$1)/0.019 &  ($>$1)/0.013 &  ($>$1)/0.013 \\
	arg$\max \varphi(h_t\xi)$ & 	$\infty$/$\infty$ & $\infty$/$\infty$ & $\infty$/$\infty$ & $\infty$/$\infty$ & $\infty$/0.05 & $\infty$/0.026 \\
	$h_t\xi$ $|$ $\varphi < 1$ & 4.43/17.6 & 2.6/257 & 1.31/NA & 0.65/NA & 0.33/NA & 0.16/NA	
\end{tabular}} \\\hline\hline
SDIRK & 3 & 3 & 1 & L & Yes  \\\hline		SDIRK-33 & \multicolumn{5}{c |}{
\begin{tabular}[t]{@{}c | c@{\hspace{\tabcolsep}} c@{\hspace{\tabcolsep}} c@{\hspace{\tabcolsep}} c@{\hspace{\tabcolsep}} c@{\hspace{\tabcolsep}} c@{}}
	$k$ & 2 & 4 & 8 & 16 & 32 & 64 \\ \hline
	$\max \varphi(h_t\xi)$ & 0.16/0.004 & 0.15/0.005 & 0.15/0.005 & 0.15/0.005 & 0.15/0.005 & 0.15/0.005 \\
	arg$\max \varphi(h_t\xi)$ & 4.84/0.85 & 2.07/0.43 & 1.03/0.22 & 0.51/0.11 & 0.26/0.05 & 0.13/0.027
\end{tabular}} \\\hline\hline
SDIRK & 3 & 4 & 1 & A & Yes  \\\hline		SDIRK-34 & \multicolumn{5}{c |}{
\begin{tabular}[t]{@{}c | c@{\hspace{\tabcolsep}} c@{\hspace{\tabcolsep}} c@{\hspace{\tabcolsep}} c@{\hspace{\tabcolsep}} c@{\hspace{\tabcolsep}} c@{}}
	$k$ & 2 & 4 & 8 & 16 & 32 & 64 \\ \hline
	$\max \varphi(h_t\xi)$ &  ($>$1)/0.75 &  ($>$1)/0.19 &  ($>$1)/0.019 &  ($>$1)/0.007 &  ($>$1)/0.007 &  ($>$1)/0.007 \\
	arg$\max \varphi(h_t\xi)$ & $\infty$/$\infty$ & $\infty$/$\infty$ & $\infty$/$\infty$ & $\infty$/0.13 & $\infty$/0.066 & $\infty$/0.033 \\
	$h_t\xi$ $|$ $\varphi < 1$ & 7.55/NA & 6.21/NA & 3.23/NA & 1.62/NA & 0.81/NA & 0.40/NA
\end{tabular}} \\\hline\hline

\Xhline{1.25pt}
\end{tabularx} 
\caption{Bounds $\varphi_F$ and $\varphi_{FCF}$ for various SDIRK and ESDIRK schemes up to three stages, as a function of $h_t\xi$, for spatial
eigenvalue $\xi$ and time step $h_t$. Here, ``$\max \varphi(h_t\xi)$'' corresponds to the maximum value obtained for $h_t\xi \in(0,\infty)$,
``arg$\max \varphi(h_t\xi)$'' provides the value of $h_t\xi$ at which this maximum is obtained, and ``$h_t\xi$ $|$ $\varphi < 1$'' gives $\hat{x}$ such
that $\varphi < 1$ (that is, convergent) for all $h_t\xi < \hat{x}$. The sign $(>1)$ indicates that $\varphi > 1$ for some $h_t \xi$, but $\varphi < \infty$
for all $h_t \xi$, ``NA'' means not applicable, and a $*$ superscript indicates one explicit stage, which does not require a linear solve.
}
\label{tab:imp}
\end{table}

When choosing a (serial) time integration scheme, properties of the method often considered include order, stage order,
stability, stiffly accurate, and cost (in terms of number of linear solves). For parallel-in-time, convergence is an obvious concern,
but the properties of schemes that yield fast Parareal/MGRIT convergence are not always consistent with the properties desired
in an integration scheme. Next we consider Parareal and MGRIT with a desirable implicit scheme on the fine grid, the A-stable ESDIRK-33
scheme, and consider the effects of coarse-grid and coarsening factor on convergence. True DIRK methods (with nonzero
diagonals in Butcher tableaux) can have stage-order at most one\cite{JCButcher_2016}. Fully implicit SDIRK (necessary
for stage-order of three or more) are often not practical, because the linear systems are too complex to solve. EDIRK and
ESDIRK methods though, consisting of one explicit stage, followed by implicit stages, can have stage-order two with
minimal additional overhead cost (one explicit stage update)\cite{JCButcher_2016}. ESDIRK-33 is a 3rd-order, stiffly
accurate, A-stable, Runge-Kutta scheme, with stage-order two. Figure \ref{fig:esdirk} considers Parareal and two-level MGRIT using
ESDIRK-33 on the fine grid, and backward Euler, 2nd-order L-stable ESDIRK-32, and A-stable ESDIRK-33 on the
coarse grid.

\begin{figure}[!h]
  \centering
  \begin{subfigure}[b]{0.33\textwidth}
    \includegraphics[width=\textwidth]{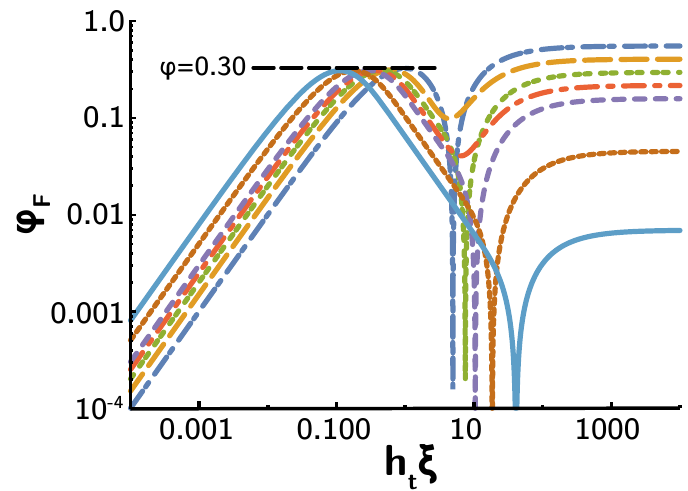}
  \end{subfigure}
   \begin{subfigure}[b]{0.33\textwidth}
    \includegraphics[width=\textwidth]{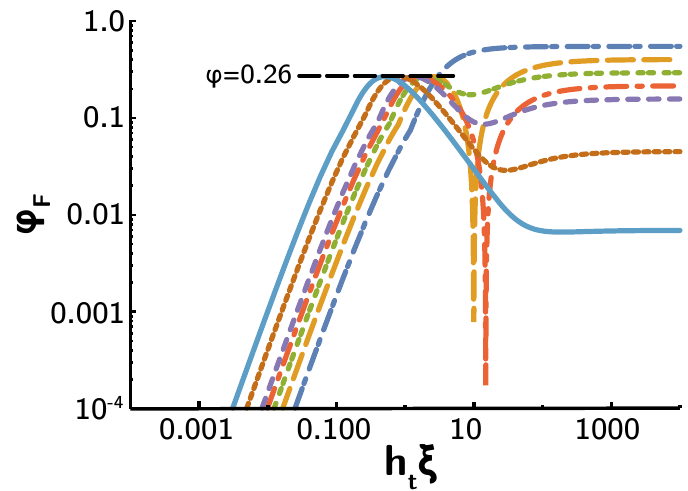}
  \end{subfigure}
  \begin{subfigure}[b]{0.33\textwidth}
    \includegraphics[width=\textwidth]{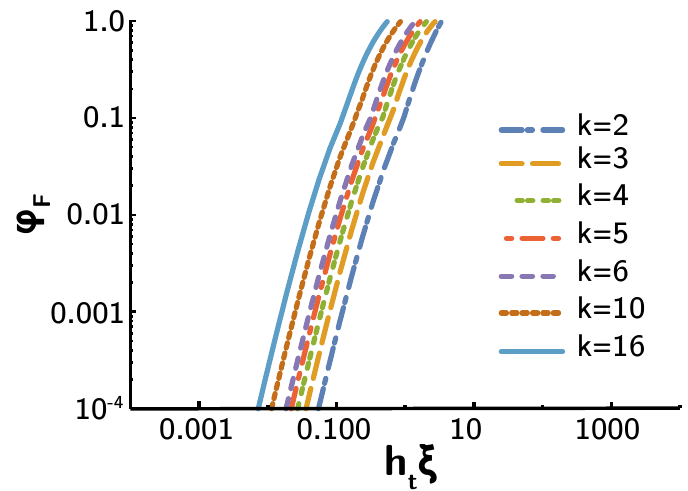}
  \end{subfigure}
\\
  \begin{subfigure}[b]{0.33\textwidth}
    \includegraphics[width=\textwidth]{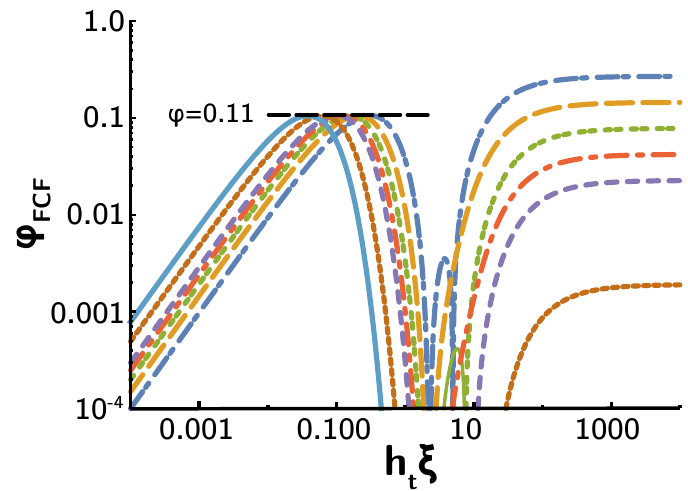}
    \caption{$\Psi \sim$ backward Euler}
    \label{fig:esdirk:be}
  \end{subfigure}
   \begin{subfigure}[b]{0.33\textwidth}
    \includegraphics[width=\textwidth]{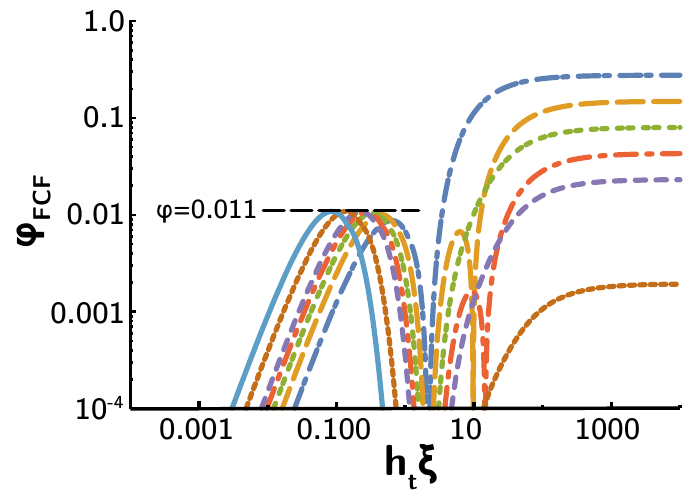}
    \caption{$\Psi \sim$ L-stable ESDIRK-32}
    \label{fig:esdirk:23}
  \end{subfigure}
  \begin{subfigure}[b]{0.33\textwidth}
    \includegraphics[width=\textwidth]{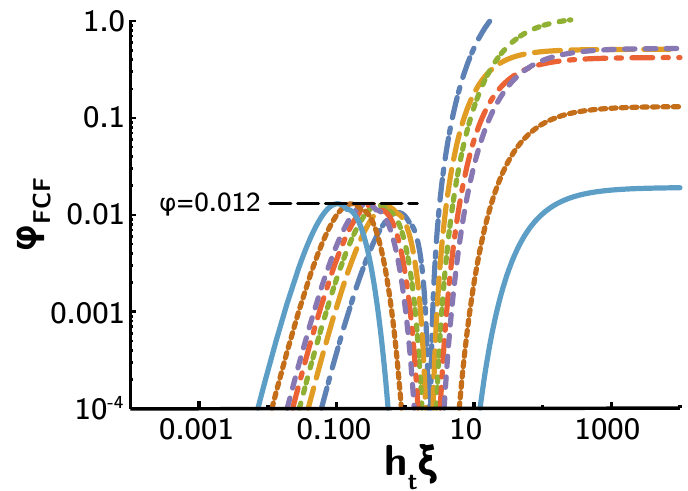}
    \caption{$\Psi \sim$ A-stable ESDIRK-33}
    \label{fig:esdirk:33}
  \end{subfigure}
\caption{Convergence bounds for Parareal and two-level MGRIT with $\Phi\sim$ A-stable ESDIRK-33, three different coarse-grid time stepping schemes,
and coarsening factors $k\in\{2,3,4,5,6,10,16\}$.}
  \label{fig:esdirk}
\end{figure}

The results in Figure \ref{fig:esdirk} show a number of interesting aspects on Parareal and MGRIT convergence:
\begin{enumerate}
\item \underline{Small changes in $h_t$ can dramatically affect convergence:} In the FCF row of Figures 
\ref{fig:esdirk:23} and \ref{fig:esdirk:33}, somewhere between $h_t\xi \in[5,12]$ (depending on $k$), a small
$\mathcal{O}(1)$ change in $h_t\xi$ can increase the worst-case convergence factor from $\varphi=0.01$ to 
as much as $15\times$ larger (again, depending on $k$). 

\item\revision{\underline{Size of $k$ can matter:} It is known that convergence of MGRIT and Parareal
generally improves with larger coarsening factors, because larger segments of the time domain are solved
exactly each iteration. However, Figure \ref{fig:esdirk:33} shows how sensitive this can be, where a modest
increase from coarsening by a factor of 2--4 to a factor of 16 changes from divergent iterations for moderate
to large $h_t\xi$ to a maximum convergence factor $\rho = 0.012$ over all $h_t\xi$.}

\item \underline{L-stability matters:} With F-relaxation and A-stable ESDIRK-33 on the fine grid, using an
L-stable coarse-grid operator is fundamental for convergence with moderate to large $h_t\xi$. L-stable schemes
(Figures \ref{fig:esdirk:be} and \ref{fig:esdirk:23}, as well as multiple other tested L-stable schemes) offer more
robust convergence than A-stable schemes (Figure \ref{fig:esdirk:33}, and multiple other tested A-stable schemes). 

\item \underline{FCF-relaxation is (sometimes) important:} With backward Euler as a coarse-grid operator
(Figure \ref{fig:esdirk:be}), FCF-relaxation is not particularly helpful, giving $\approx 1.8\times$ speedup in convergence
for roughly twice the work. Contrastingly, using L-stable ESDIRK-23 on the coarse-grid, FCF-relaxation can
yield a $3.4\times$ speedup in convergence (depending on $k$), while for A-stable ESDIRK-33 on the
coarse grid, FCF-relaxation can mean the difference between divergence and convergence. FCF-relaxation
is discussed in more detail in Section \ref{sec:irk:weight}.

\item \underline{Choice of coarse-grid integrator matters:} For twice as many linear solves on the coarse
grid ($\ll2\times$ the total work, depending on $k$), using ESDIRK-32 and ESDIRK-33 on the coarse grid can
result in convergence twice as fast with FCF-relaxation. Conversely, with just F-relaxation, the extra linear
solves associated with ESDIRK-32 and ESDIRK-33 result in either minor improvement in convergence
over backward Euler or divergence.

\item \underline{Parity of $k$ can matter:} Notice in the FCF row of Figure \ref{fig:esdirk:33} that for large $h_t\xi$, 
coarsening factors of $k=2$ and $k=4$ both diverge, while coarsening factors $k=3$ and $k=5$ converge. This is
due to the sign of $\lambda^k$. 

\end{enumerate}

These are just five interesting points on Parareal and MGRIT convergence that arise from this specific example.
The larger lesson is that convergence of Parareal/MGRIT can be complex, even in the case of unitarily
diagonalizable operators. However, simple a priori analysis can motivate algorithmic choices that
make the difference between convergence and divergence, or yield speedups of several times. 

\subsection{L-stability and coarse-/fine-grid propagators}\label{sec:irk:L}

Table \ref{tab:imp} shows that not all Runge-Kutta schemes are equal from the perspective of parallel-in-time. Indeed, some
schemes offer significantly better convergence for the same, for example, number of stages. This was observed in
\cite{SLWu_2016}, where it is noted that the implicit trapezoid method (that is, 2nd-order A-stable ESDIRK) and 4th-order
Gauss Runge-Kutta do not have convergence bounded independent of $h_t$ and $h_x$ and, in fact, can observe arbitrarily
slow convergence (consistent with Proposition \ref{prop:L}). 

This section starts by analyzing the implicit trapezoid method closer. Figure \ref{fig:trap} shows the worst-case convergence
of Parareal and MGRIT using F- and FCF-relaxation as a function of $h_t\xi$, where the implicit trapezoid rule is used on the
fine grid, and either the implicit trapezoid (Figure \ref{fig:trap:it}) or backward Euler (Figure \ref{fig:trap:be}) is used on the
coarse grid. First, note that using backward Euler as a coarse-grid propagation scheme is a perfectly reasonable method
with good convergence if $h_t\xi < \mathcal{O}(10)$. The arbitrarily slow convergence observed in \cite{SLWu_2016} occurs
for $h_t\xi > \mathcal{O}(10)$, but, in many practical applications, this regime is not of interest. Using the implicit
trapezoid method on the coarse grid with just F-relaxation tightens this bound, where divergence is observed for
$h_t\xi > \mathcal{O}(0.1)-\mathcal{O}(1)$ (depending on $k$). Indeed, this is a tighter constraint that is not often satisfied in practice.
However, using the implicit trapezoid method on the coarse grid with FCF-relaxation yields significantly faster convergence
than backward Euler if $h_t\xi< \mathcal{O}(10)$ ($\approx3.5\times$ faster convergence than backward Euler with F-relaxation
and $\approx1.8\times$ faster than backward Euler with FCF-relaxation), for a comparable cost to backward Euler with
FCF-relaxation. 

\begin{figure}[!ht]
  \centering
  \begin{subfigure}[b]{0.33\textwidth}
    \includegraphics[width=\textwidth]{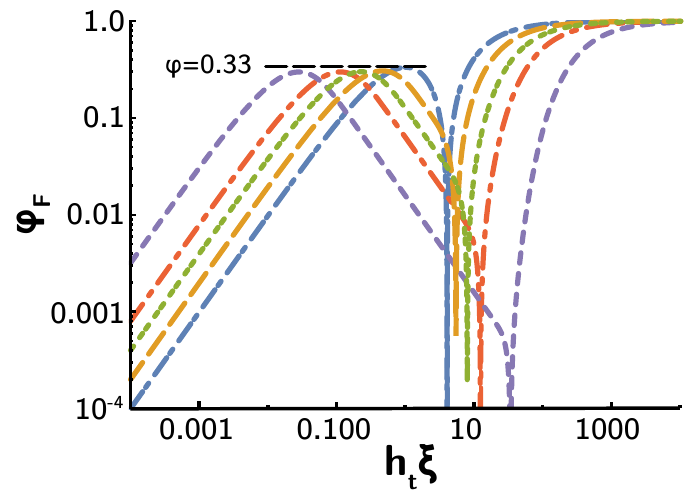}
  \end{subfigure}
  \begin{subfigure}[b]{0.33\textwidth}
    \includegraphics[width=\textwidth]{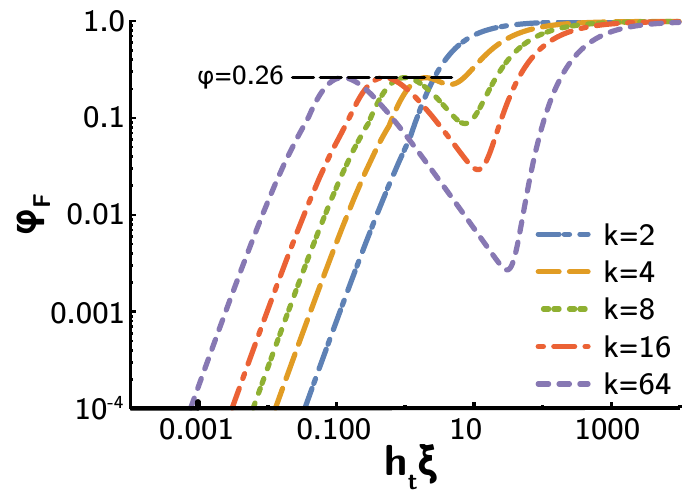}
  \end{subfigure}
  \begin{subfigure}[b]{0.33\textwidth}
    \includegraphics[width=\textwidth]{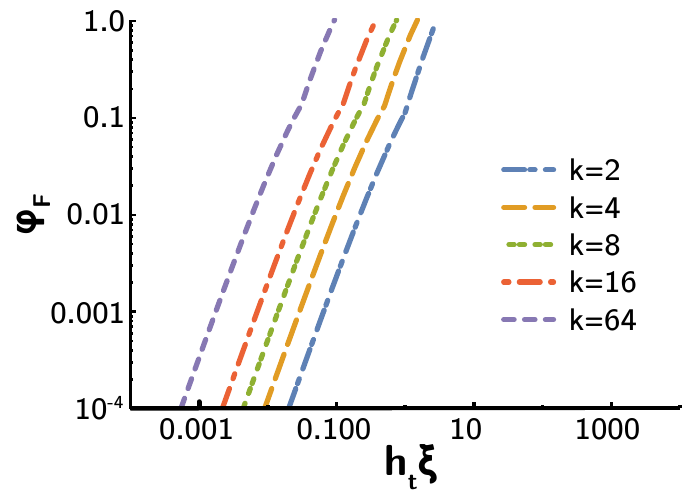}
  \end{subfigure}
  \\
    \begin{subfigure}[b]{0.33\textwidth}
    \includegraphics[width=\textwidth]{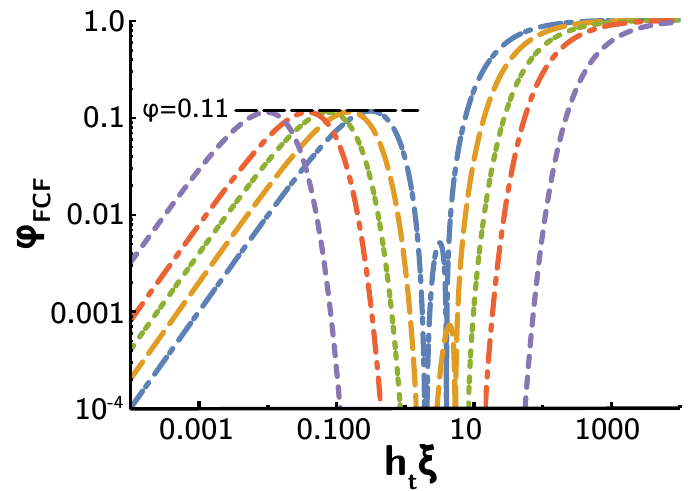}
    \caption{$\Psi \sim$ backward Euler}
    \label{fig:trap:be}
  \end{subfigure}
  \begin{subfigure}[b]{0.33\textwidth}
    \includegraphics[width=\textwidth]{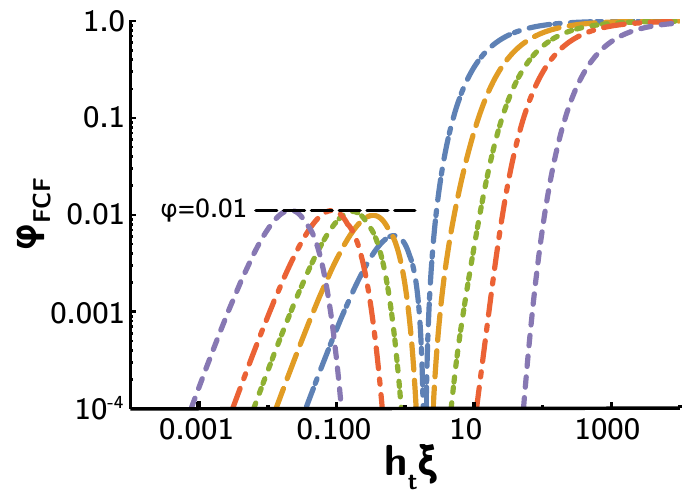}
        \caption{$\Psi \sim$ L-stable SDIRK-22}
    \label{fig:trap:sdirk}
  \end{subfigure}
  \begin{subfigure}[b]{0.33\textwidth}
    \includegraphics[width=\textwidth]{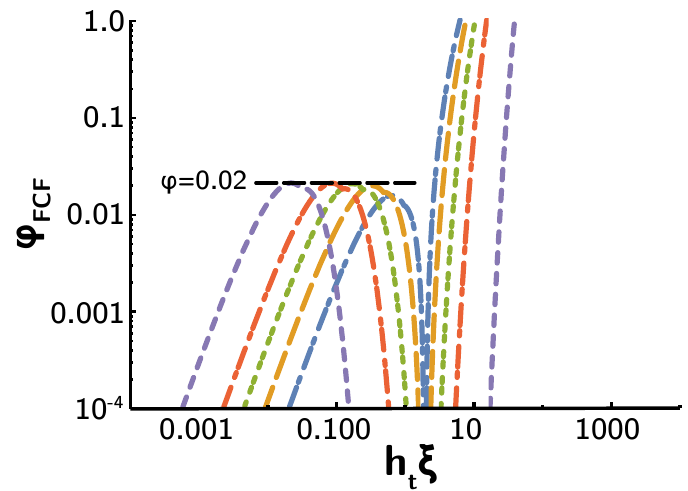}
    \caption{$\Psi \sim$ implicit trapezoid}
    \label{fig:trap:it}
  \end{subfigure}
\caption{Worst-case convergence factors for using the implicit trapezoid method on the fine grid, $\Phi$, and 
either backward Euler, implicit trapezoid, or 2nd-order L-stable SDIRK-22 on the coarse grid, $\Psi$, as a function
of $h_t \xi$, and for coarsening factors $k\in\{2,4,8,16,64\}$. }
  \label{fig:trap}
\end{figure}

In the case that $h_t\xi \gg 10$, Wu \cite{SLWu_2016} proposes a modification where $\Phi^k$ consists of two
steps using an L-stable scheme of the same order (SDIRK2 instead of the implicit trapezoid method or 4th-order 
Lobatto III-C instead of 4th-order Gauss) followed by $k-2$ steps using the implicit trapezoid method
(or 4th-order Gauss integrator). This does change the underlying problem being solved, that is, Parareal/MGRIT converges
to the solution based on alternating time steps, where two out of every $k$ time steps are now being integrated
with a different scheme, rather than the original implicit trapezoid or Gauss4 on every step. A detailed analysis
of this method is given in \cite{SLWu_2016}, however, the framework used here provides a simple conceptual
understanding as well. 

In considering convergence of an L-stable scheme for large $h_t\xi$, the key point is that
$\lim_{h_t\xi \to \infty} |\lambda_i| = 0$, for all eigenvalues $\{\lambda_i\}$ of $\Phi$. From Proposition \ref{prop:L},
using an L-stable scheme on both grids guarantees convergence for large $h_t\xi$, while such a result does
not (necessarily) hold in the A-stable case. However, suppose we change one step of the F-points to be an
L-stable scheme. Then, the $|\lambda_i^k|$ term that appears in convergence bounds now takes the form
$|\eta_i\lambda_i^{k-1}|$, where $\{\eta_i\}$ is the eigenvalue corresponding to the L-stable time step. If we take
this limit as $h_t\xi \to \infty$, because $|\lambda_i|$ is bounded, the limit is zero. For purposes
of Parareal and MGRIT, this makes the fine-grid appear L-stable, and Proposition \ref{prop:L} applies.

The implicit trapezoid rule is a fine-grid propagation scheme for which it is difficult to find a coarse-grid propagation
scheme that is effective for all $h_t$ and $h_x$. Indeed, in our testing, no other schemes appear substantially better
than results in Figure \ref{fig:trap}. For many practical applications, the regime of effective convergence is sufficient,
but in some cases large $h_t\xi$ may be important too.

\subsection{FCF-relaxation}\label{sec:irk:weight}

\subsubsection{High-Order integrators and L-stability}\label{sec:irk:weight:ho}

\revision{
From results presented so far, one might notice that FCF-relaxation is particularly useful with high-order integration
schemes, and either L-stable integration schemes or larger coarsening factors (see Figures \ref{fig:esdirk} and \ref{fig:trap} 
and Table \ref{tab:imp}). We start by discussing this phenomenon and the significance of high-order integrators.
Figure \ref{fig:fcf} plots the eigenvalue magnitude $|\lambda|$, convergence bounds for F-relaxation, and
convergence bounds for FCF-relaxation, as a function of $h_t\xi$. For $h_t\xi \ll 0.1$, $|\lambda| \approx 1$
and FCF-relaxation will not offer significant improvement in convergence for these eigenmodes (because
$|\lambda^k|\approx 1$). However, in this case, just F-relaxation is typically sufficient for rapid convergence when $h_t\xi \ll 1$,
likely due to $\Psi$ and $\Phi^k$ both being an accurate approximation to the exponential (i.e., exact solution) in
this regime.

For $h_t\xi \gg 1$, L-stable schemes yield $|\lambda|\to0$ as $h_t\xi \to\infty$, and FCF-relaxation guarantees
rapid convergence. Conversely, using A-stable schemes, often divergence is observed for F-relaxation 
as $h_t\xi \to \infty$, while A-stability only ensures $|\lambda|\leq 1$. The implicit midpoint method (i.e., Crank-Nicolson/trapezoid
method) is actually unbounded as $h_t\xi\to\infty$, but other high-order RK schemes asymptote to some constant. 
Furthermore, we observe that $\lim_{h_t\xi\to\infty} |\lambda|$ decreases for A-stable schemes as integration order increases,
for all SDIRK and ESDIRK schemes we have considered, with limiting value less than one.
Then, raising $|\lambda|^k$ for sufficiently large $k$ can overcome the (bounded) divergence observed with F-relaxation
on large $h_t\xi$, and provide a convergent method (see Table \ref{tab:imp} and Figures \ref{fig:esdirk} and
\ref{fig:fcfA}), with better convergence associated with higher-order integration or larger $k$. 

Finally, consider the middle regime of $h_t\xi \approx 1$, give or take 1--2 orders of magnitude. Here, convergence factors 
of just F-relaxation is typically moderately large, and FCF-relaxation offers some benefit, but not as much as for large
$h_t\xi \gg1$. However, observe that the convergence behavior of Parareal/MGRIT with F-relaxation shifts
to the right with $h_t\xi$ (in the sense of plots in Figure \ref{fig:fcf}) with increasing order of integration (see Figures \ref{fig:fcfL}
and \ref{fig:fcfA}). Although the shift appears minor, the effect of this is that the maximum convergence is obtained
at larger $h_t\xi$ with increasing integration order. For L-stable schemes and many A-stable schemes, the
corresponding $\lambda(h_t\xi)$ is then smaller, and $|\lambda^k|$ provides a more
effective reduction of the maximum convergence factor with higher-order integration. This effect can be seen for orders 2--4 A- and L-stable
schemes in Figure \ref{fig:fcf}. A plausible explanation for the shift in $h_t\xi$ with higher-order integration is
that high-order schemes are better approximations to the exponential, and $\Psi$ will approximate $\Phi^k$ well
over a slightly larger interval of $h_t\xi,\in(0,\hat{z}_\text{max})$, although this is speculation. 
}
\begin{figure}[!h]
  \centering
  \begin{subfigure}[b]{\textwidth}
    \includegraphics[width=0.33\textwidth]{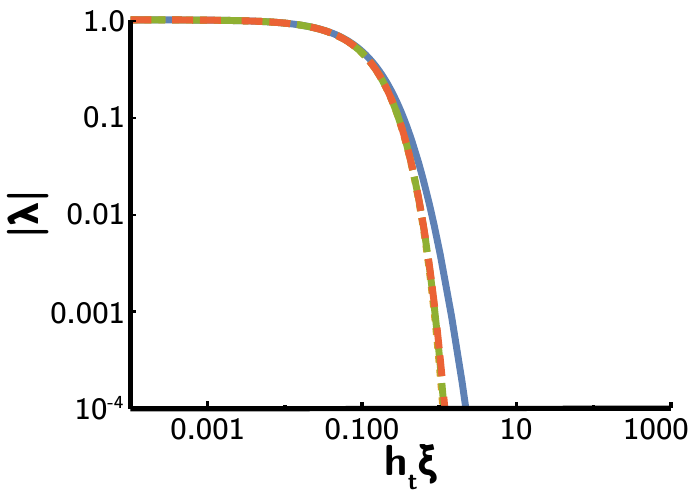}
    \includegraphics[width=0.33\textwidth]{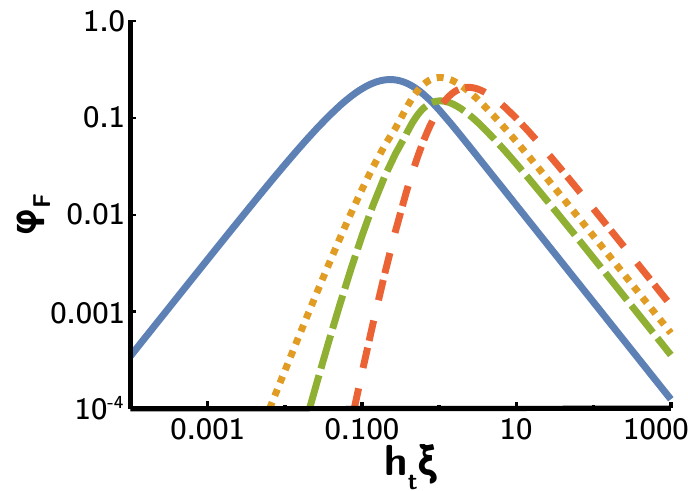}
    \includegraphics[width=0.33\textwidth]{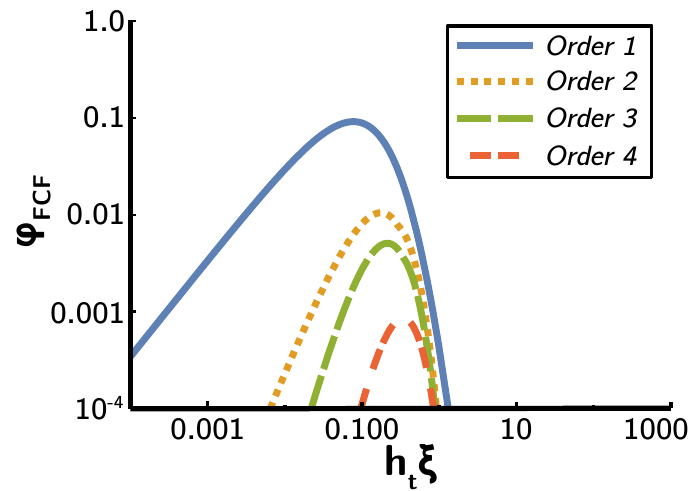}
    \caption{L-stable SDIRK schemes.}
    \label{fig:fcfL}
  \end{subfigure}
\\
  \begin{subfigure}[b]{\textwidth}
    \includegraphics[width=0.33\textwidth]{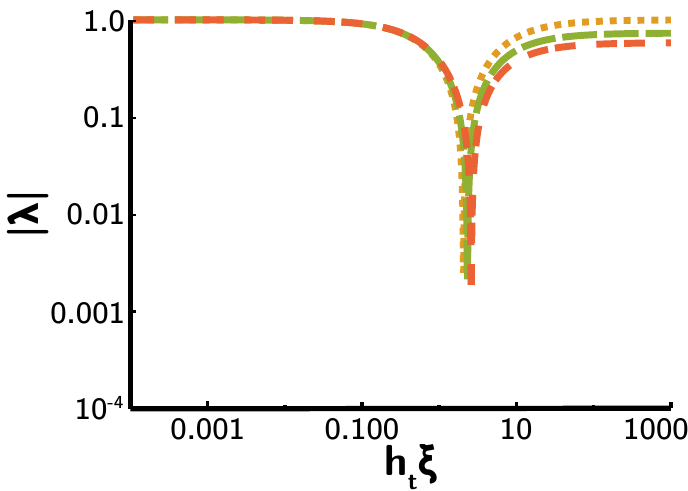}
    \includegraphics[width=0.33\textwidth]{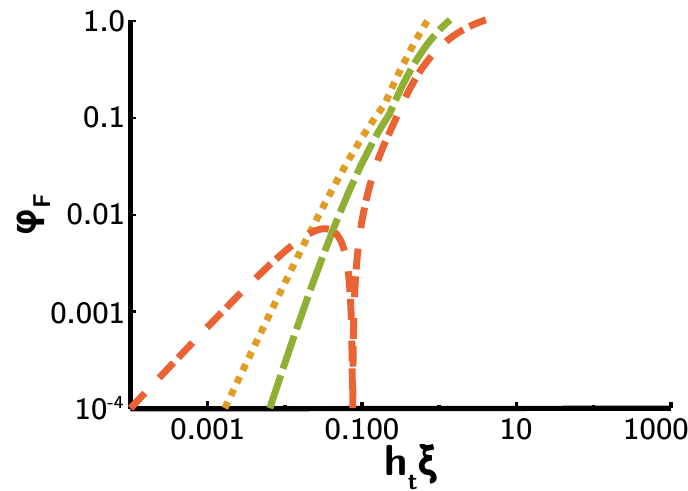}
    \includegraphics[width=0.33\textwidth]{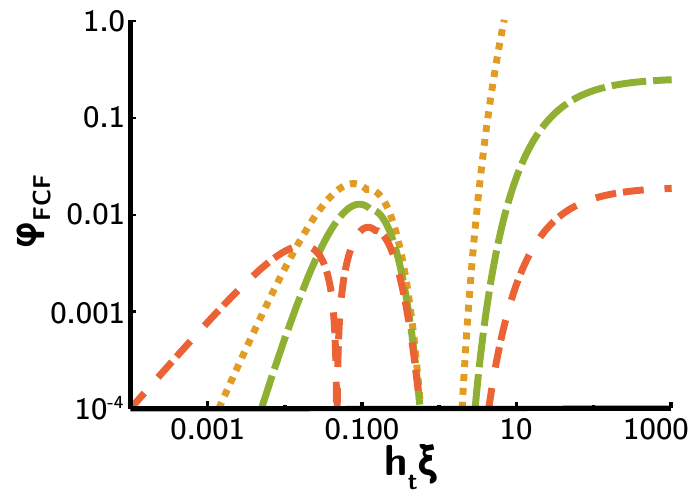}
    \caption{A-stable SDIRK schemes.}
    \label{fig:fcfA}
  \end{subfigure}
\caption{\revision{Eigenvalue $\lambda$, convergence bounds for F-relaxation, and convergence bounds for FCF-relaxation,
 as a function of $h_t\xi$, for A- and L-stable SDIRK schemes of order 1--4 (there is no 1st-order scheme that is A-stable
 but not L-stable), and coarsening factor $k = 8$.}}
  \label{fig:fcf}
\end{figure}

\subsubsection{Weighted iterations}\label{sec:irk:weight:weight}

Sometimes in iterative-type methods, a weighting is applied to the correction to improve convergence, such
as the damping factor in Jacobi or weight in successive over relaxation. In principle, the same concept can be applied in the 
Parareal/MGRIT context as well. In a (linear) Parareal framework, one can consider 
\begin{align*}
\mathbf{u}_{n+1}^{\nu+1} & = \Phi^\nu \mathbf{u}_n^\nu + \theta_n^\nu\left[\Psi \mathbf{u}_{n+1}^{\nu+1} - \Psi \mathbf{u}_{n+1}^\nu\right],
\end{align*}
where the subscript $n$ denotes approximate $\mathbf{u}$ at the $n$th fine-grid time point and superscript $\nu$ denotes iteration
number. This is the $\theta$-Parareal scheme introduced in \cite{GAriel_etal_2018a}, where $\theta$ is some scaling factor for the coarse-grid
operator that may be a scalar or linear operator. Here, we assume it is a scalar, and derive an interesting relation to
FCF-relaxation.

For scalar $\theta$, it can be absorbed into the $\Psi$ operator, and the convergence theory for two-level linear MGRIT/Parareal
applies. Thus, consider the tight bounds on Parareal/two-level MGRIT with F-relaxation as a function of $\theta$, given by
\begin{align*}
\varphi_F(\theta) & = \max_{h_t\xi} \frac{|\theta\mu(h_t\xi) - \lambda(h_t\xi)^k|}{1 - \theta|\mu(h_t\xi)|} : = \max_{h_t\xi} \mathcal{M}(\theta,h_t\xi),
\end{align*}
for $\theta \in[0,1]$ and $h_t\xi$ such that $|\theta\mu|,|\lambda| < 1$ for all $\xi$, that is, the coarse-
and fine-grid propagators are stable. It turns out, this weighting has an interesting relation to FCF-relaxation. In the case that $\theta$
is an operator, it is likely the case that $\Phi$ and $\Psi$ are no longer diagonalizable under the same eigenvectors, in which case the 
more general theory developed in \cite{BSSouthworth_2018a} is necessary to perform a tight two-grid analysis. 

For scalar $\theta\in[0,1]$, note that
\begin{align*}
\lim_{h_t\xi\to0} \frac{|\theta\mu(\xi) - \lambda(h_t\xi)^k|}{1 - \theta|\mu(h_t\xi)|}  = 
	\lim_{h_t\xi\to 0} \frac{|\lambda(h_t\xi)^k||\mu(h_t\xi) - \lambda(h_t\xi)^k|}{1 - |\mu(h_t\xi)|} = \begin{cases} 1 & \theta \in[0,1) \\ 0 & \theta = 1\end{cases}.
\end{align*}
That is, for any $\theta \neq 0$, iterations lose the desirable property proved in Corollary \ref{cor:limits}, where $\Psi$
exactly approximates $\Phi^k$ as $h_t\xi \to 0$. Moreover, there are often many eigenvalues $h_t\xi\approx 0$, 
corresponding to small spatial eigenvalues. Using $\theta \neq 1$ results in Parareal/MGRIT convergence factors being $\approx 1$
when applied to these eigenmodes.\footnote{Effectively the same results hold for the exact bounds in \eqref{eq:v_bounds},
but $\theta\neq1$ yields $\approx N_c/(1+N_c) \approx 1$.}

However, there is still something advantageous about iterations $\theta \neq 1$, particularly \textit{after} the first iteration. 
Note that $\theta = 1$ corresponds to a CF-relaxation, that is,
\begin{align*}
\frac{|\theta\mu(h_t\xi) - \lambda(h_t\xi)^k|}{1 - \theta|\mu(h_t\xi)|} \hspace{3ex}\mapsto_{\theta = 0}\hspace{3ex} |\lambda(h_t\xi)^k|,
\end{align*}
which is exactly the additional factor that FCF-relaxation adds to error propagation over just F-relaxation \eqref{eq:v_bounds}.
From derivations in Section 5.3 of Southworth\cite{BSSouthworth_2018a}, the maximum over $h_t\xi$ of multiple Parareal/MGRIT iterations
can be taken as a maximum
over the product. It follows that applying two iterations of Parareal/two-level MGRIT with F-relaxation, first
with $\theta=1$ followed by $\theta=0$, is exactly equivalent to one iteration of two-level MGRIT with FCF-relaxation.

\begin{figure}[!h]
  \centering
  \begin{subfigure}[b]{0.33\textwidth}
    \includegraphics[width=\textwidth]{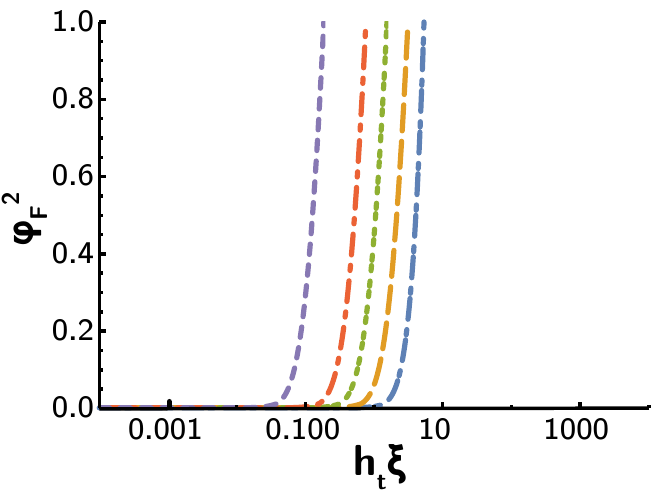}
    \caption{$\theta_1=1$, $\theta_2=1$}
  \end{subfigure}
   \begin{subfigure}[b]{0.33\textwidth}
    \includegraphics[width=\textwidth]{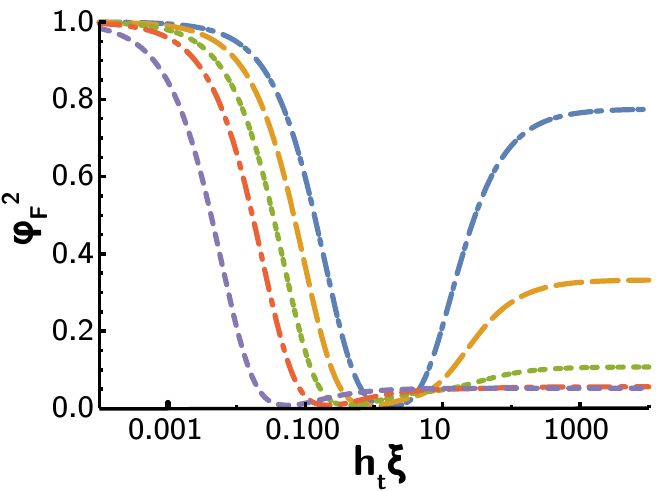}
    \caption{$\theta_1=0.25$, $\theta_2=0.25$ }\label{subfig:b}
  \end{subfigure}
  \begin{subfigure}[b]{0.33\textwidth}
    \includegraphics[width=\textwidth]{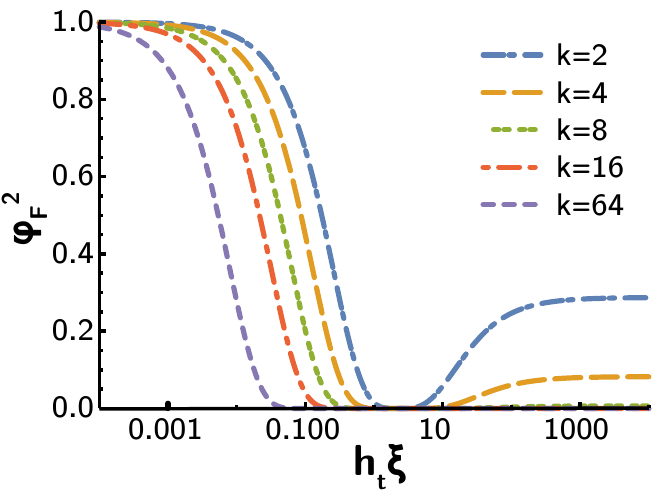}
    \caption{$\theta_1=0$, $\theta_2=0$}\label{subfig:c}
  \end{subfigure}
  \\\vspace{3ex}
    \begin{subfigure}[b]{0.33\textwidth}
    \includegraphics[width=\textwidth]{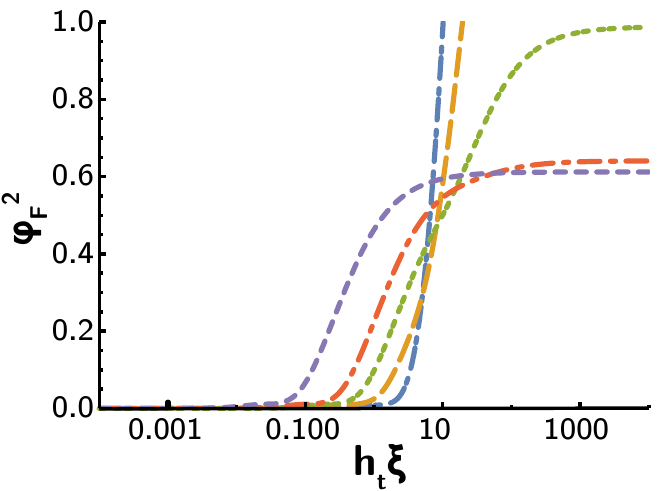}
    \caption{$\theta_1=1$, $\theta_2=0.25$}
  \end{subfigure}
   \begin{subfigure}[b]{0.33\textwidth}
    \includegraphics[width=\textwidth]{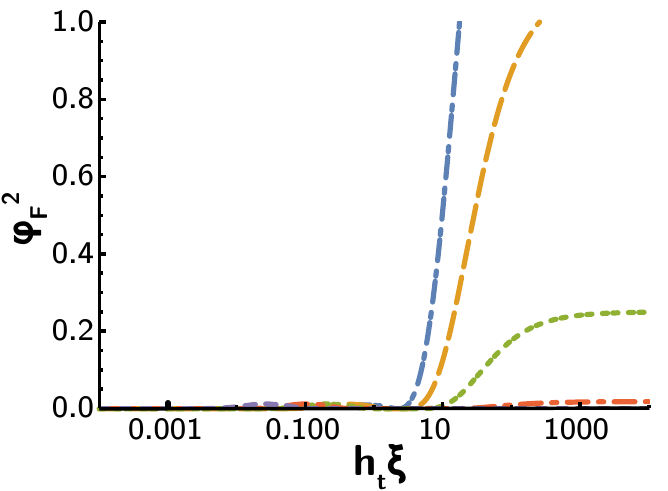}
    \caption{$\theta_1=1$, $\theta_2=0$ (i.e., one iter w/ FCF) }
  \end{subfigure}
\caption{Worst case (total) convergence factors for two successive iterations of $\theta$-Parareal with F-relaxation,
with weights $\theta_1$ and $\theta_2$, respectively, and coarsening factors $2,4,8,16$, and $64$. Three-stage,
3rd-order A-stable ESDIRK is used for coarse- and fine-grid propagation.}
  \label{fig:theta}
\end{figure}

In this case and a number of other tested examples, a second iteration with $\theta = 0$, that is, FCF-relaxation,
appears to be optimal in some sense. It is also advantageous from a computational perspective, because adding
a CF-relaxation is significantly cheaper than performing a full Parareal/MGRIT iteration with F-relaxation. Nevertheless, it
is interesting to see the relation of FCF-relaxation with the $\theta$-weighting. In particular, it raises the question
of whether a polynomial-type relaxation could be used with successive weights $\{\theta_0,\theta_1,...\}$ that
are ``optimal'' in some sense.

\begin{remark}
Recall for two-grid convergence, we are interested in the iteration (in the notation of \cite{BSSouthworth_2018a})
\begin{align*}
I - B_\Delta^{-1} A_\Delta & = I - \begin{bmatrix} I \\ \Psi & I \\ \vdots &  & \ddots \\ \Psi^{N_c-1} & ... & \Psi & I \end{bmatrix}
	\begin{bmatrix} I \\ -\Phi^k & I \\ & \ddots & \ddots \\ && -\Phi^k & I \end{bmatrix}.
\end{align*}
Here, $A_\Delta$ denotes the Schur complement corresponding to an exact two-level correction, and its inverse is
approximated by substituting $\Psi$ for $\Phi^k$. From an iterative methods perspective, arguably a more intuitive
weighting than above is to consider $I - \omega B_\Delta^{-1}A_\Delta$ for $\omega\in(0,2)$, that is, weighting the entire coarse-grid
correction. Working through the algebra, one arrives at similar bounds on two-grid convergence, which take the form
\begin{align*}
\varphi_{F,\omega} & = \max_{\xi} \left[ (1-\omega) + \omega\frac{|\mu(h_t\xi) - \lambda(h_t\xi)^k|}{1 - |\mu(h_t\xi)|} \right] , \\
\varphi_{FCF,\omega} & = \max_{\xi} \left[ (1-\omega)|\lambda(h_t\xi)^k| + \omega \frac{|\lambda(h_t\xi)^k||\mu(h_t\xi) - \lambda(h_t\xi)^k|}{1 - |\mu(h_t\xi)|} \right]. 
\end{align*}
Interestingly, for no tested Runge-Kutta schemes does such an approach appear to offer significant advantages
over $\omega = 1$ (unweighted). 
\end{remark}

\section{Numerical results}\label{sec:ex}

As an example, we consider the diffusion equation,
\begin{align}\label{eq:ex_diff}
	u_t - \Delta u = f(\mathbf{x},t), \quad (\mathbf{x},t) \in \Omega\times(0,1],
\end{align}
in a bounded domain $\Omega \subset \mathbb{R}^2$, with solution-independent forcing term 
\[
	f(\mathbf{x},t) = (\tau\cos(\tau t) + 2\kappa^2\sin(\tau t))\sin(\kappa x_1)\sin(\kappa x_2), \quad \mathbf{x} = (x_1,x_2)^T,
\]
and constants $\tau = 13\pi/6$, $\kappa = \pi$. We prescribe an initial condition for $u$ at $t = 0$, $u(\mathbf{x},0) = \sin(\kappa x_1)\sin(\kappa x_2)$, and impose the Neumann spatial boundary condition 
\[
	\mathbf{n}\cdot\nabla u = (\kappa\cos(\kappa x_1)\sin(\kappa x_2)\sin(\tau t), \kappa\sin(\kappa x_1)\cos(\kappa x_2)\sin(\tau t))^T, \quad (\mathbf{x},t)\in\partial\Omega \times (0,1].
\]
We discretize the spatial domain, $\Omega$, using continuous $Qp$ finite elements, $p = 1, 2, \ldots$. Denoting the
mass and stiffness matrix of the discretized Laplacian by $\mathcal{M}$ and $\mathcal{S}$, respectively, we obtain
\[
	\mathcal{M}\mathbf{u}'(t) - \mathcal{S}\mathbf{u}(t) = \mathbf{f}(t),
\]
which is of the form of our model problem \eqref{eq:problem} with operator $\mathcal{L} = -\mathcal{M}^{-1}\mathcal{S}$. For
consistency with our underlying assumption of an SPD operator, we used a lumped mass matrix and structured grid, in which
case $\mathcal{M}^{-1}$ is a constant scaling. The time interval is discretized using various SDIRK and ESDIRK schemes,
already considered in Table \ref{tab:imp}, with constant time-step $h_t$. Butcher tableaux can be found in the appendix.
Note that, for the diffusion problem, the values of the product $h_t\xi$, where $\xi$ denotes an eigenvalue of $\mathcal{L}$,
are between zero and a constant of order $\mathcal{O}(h_t/h^2)$, where $h$ represents the maximum of the step sizes in both
spatial dimensions.

In the following, we report on tests of solving the diffusion problem \eqref{eq:ex_diff}, as implemented in the driver \texttt{drive-diffusion}
from the open-source package XBraid \cite{xbraid-package}, and using the modular finite element library MFEM \cite{mfem-library}.
We consider two space-time domains, both with time interval $(0,1]$, but with different spatial domains. More precisely, we choose
either a star-shaped spatial domain or a beam, defined by (refinements of) the meshes \texttt{star.mesh} and \texttt{beam-quad.mesh}
from MFEM's \texttt{data} directory; see Figure \ref{fig:ex_diff_meshes}. Convergence is measured by the factor $\rho$, defined as
the maximum ratio of space-time residual norms over two successive iterations. Iterations are carried out until the space-time residual
norm is smaller than $10^{-13}$.

\begin{figure}[h!t]
	\centerline{\includegraphics[width=.3\textwidth]{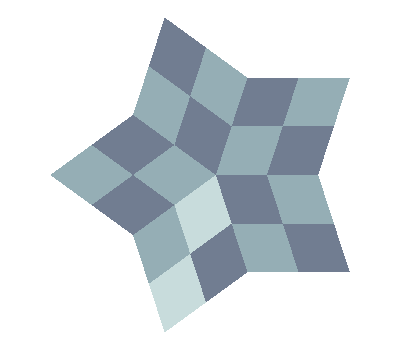}\quad\includegraphics[width=.3\textwidth]{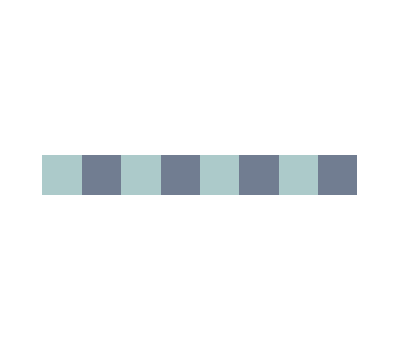}}
	\caption{MFEM's \texttt{star.mesh} (left) and \texttt{beam-quad.mesh} (right).\label{fig:ex_diff_meshes}}
\end{figure}

\subsection{L-stable time integration methods}

As shown in the previous sections, using L-stable schemes leads to robust and rapid Parareal and two-level MGRIT convergence. In this section,
we first aim at confirming this observation in a practical setting, and then compare with full multilevel MGRIT V-cycles on the star-shaped
domain. The time interval $[0,1]$ is discretized using various SDIRK schemes or the ESDIRK-32 method with a time-step size of
$h_t = 1/4096$. In space, either Q1 elements of size $0.06\times 0.06$, or Q2 or Q3 elements of size $0.12\times 0.12$ are used,
with the order of the finite-element discretization chosen such that it matches the order of the SDIRK or ESDIRK scheme. Although
these are smaller spatial problems than typically considered in practice, as discussed previously, the only aspect that matters for
convergence is the ratio of $h_t$ and $h_x$. Figure \ref{fig:evals_ex_diff} shows the distribution of the eigenvalues $\{\xi\}$ of the
operator $\mathcal{L}$ for the three spatial discretizations. In the case of bilinear elements, eigenvalues are distributed in the
interval $(0,6{,}793)$, and in the interval $(0, 8{,}491)$ or $(0, 24{,}077)$ when using biquadratic or bicubic elements, respectively.
Thus, on the fine grid the product $h_t\xi$ takes values between 0 and about 1.66, 2.08, or 5.88, respectively.

Table \ref{tab:ex_diff_star} shows convergence factors of two-level and full multilevel MGRIT V-cycles using $\lfloor\log_k(4096)\rfloor$
multigrid levels and even coarsening factors between two and 32. Two-level results are in perfect agreement with the theoretical bounds in
Table \ref{tab:imp}. In the case of SDIRK-22+Q2, for example, theory predicts that when using a coarsening factor of two or four, the
maximum value $\varphi_F$ is obtained for $h_t\xi = 5$ or $h_t\xi = 2.1$, respectively. In the tests, $h_t\xi$ takes values between 0
and 2.07 for this discretization and, thus, the maxima are not obtained, resulting in faster convergence than the maximum bounds.
Multilevel results show that convergence degrades, especially for small coarsening factors and for F-relaxation. Note, however, that 
for factor-2, factor-4, or factor-8 coarsening, the multilevel hierarchy consists of 11, six, or four grid levels, respectively, whereas for
the other coarsening factors, only three-level V-cycles are performed.

Figure \ref{fig:mg_conv} details multilevel convergence for backward Euler and the ESDIRK-32 scheme, plotting the maximum
convergence factors as functions of the number of grid levels. \revision{Interestingly, multilevel convergence of MGRIT with F-relaxation
asymptotes (regarding increasing grid level) significantly slower than the worst-case convergence of the two-level method in the case of backward 
Euler, while MGRIT with FCF-relaxation as well as MGRIT with F- or FCF-relaxation for ESDIRK-32 observe multilevel convergence factors 
approximately equal to (theoretical) worst-case two-level convergence (see Table \ref{tab:imp}). }The latter results make sense to some extent 
because coarse grids take progressively larger time steps, but for 
these (L-stable) schemes, (two-grid) convergence factors are smallest for very large time steps. However, it is not yet understood exactly
why this holds, nor why it does not hold for backward Euler with F-relaxation.

\begin{table}[h!t]
\renewcommand{\tabcolsep}{0.195cm}
\renewcommand{\arraystretch}{1.2}
\centering
\begin{tabular}{|>{\centering}m{.05\textwidth}|>{\centering}m{.15\textwidth}|>{\centering}m{.05\textwidth}|>{\centering}m{.09\textwidth}|>{\centering}m{.09\textwidth}|>{\centering}m{.09\textwidth}|>{\centering}m{.09\textwidth}|>{\centering\arraybackslash}m{.09\textwidth}|}
	\hline\multicolumn{3}{|r|}{$k = ~$} & 2 & 4 & 8 & 16 & 32\\ 
	\hline\hline
	& BWE + $Q1$ & $\rho$ & 0.12/0.05 & 0.20/0.08 & 0.24/0.09 & 0.27/0.10 & 0.28/0.10 \\
	2- & SDIRK-22 + $Q2$ & $\rho$ & 0.15/0.008 & 0.21/0.009 & 0.25/0.01 & 0.21/0.01 & 0.19/0.01 \\
	level & ESDIRK-32 + $Q2$ & $\rho$ & 0.15/0.007 & 0.22/0.008 & 0.25/0.008 & 0.25/0.007 & 0.25/0.008 \\
        & SDIRK-33 + $Q3$ & $\rho$ & 0.12/0.003 & 0.13/0.004 & 0.11/0.005 & 0.12/0.005 & 0.13/0.004 \\
	\hline\hline
	& BWE + $Q1$ & $\rho$ & 0.50/0.10 & 0.36/0.09 & 0.30/0.08 & 0.25/0.10 & 0.27/0.10 \\
	V- & SDIRK-22 + $Q2$ & $\rho$ & 0.30/0.03 & 0.28/0.003 & 0.24/0.007 & 0.22/0.01 & 0.25/0.009 \\
	cycle & ESDIRK-32 + $Q2$ & $\rho$ & 0.29/0.004 & 0.26/0.003 & 0.25/0.01 & 0.24/0.01 & 0.25/0.009 \\
	& SDIRK-33 + $Q3$ & $\rho$ & 0.24/0.05 & 0.17/0.008 & 0.15/0.004 & 0.14/0.004 & 0.14/0.004 \\
	\hline
\end{tabular}
\caption{Maximum convergence factors $\rho$ for various SDIRK schemes and ESDIRK-32. Convergence factors for using F-relaxation are given first, followed by convergence factors of iterations using FCF-relaxation.
}
\label{tab:ex_diff_star}
\end{table}

\begin{figure}[h!t]
	\begin{subfigure}[b]{0.49\textwidth}
    		\includegraphics[width=\textwidth]{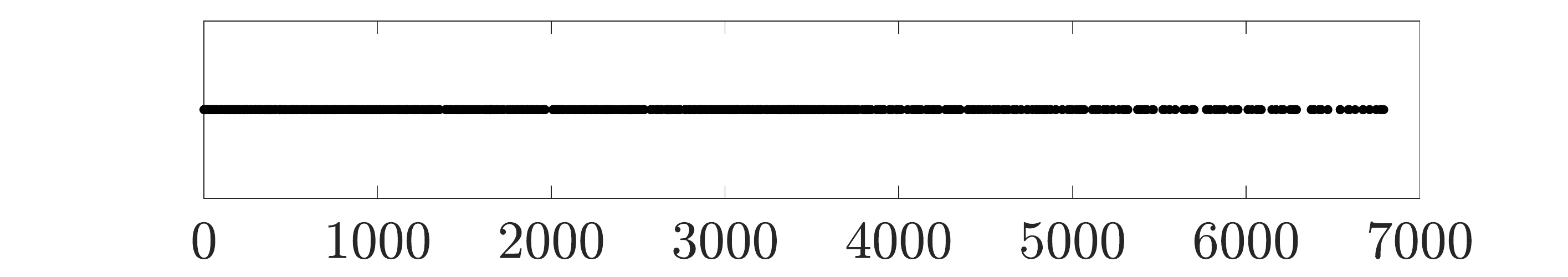}
    		\caption{$Q1$ elements of size $0.06\times 0.06$}
 	 \end{subfigure}~\begin{subfigure}[b]{0.49\textwidth}
    		\includegraphics[width=\textwidth]{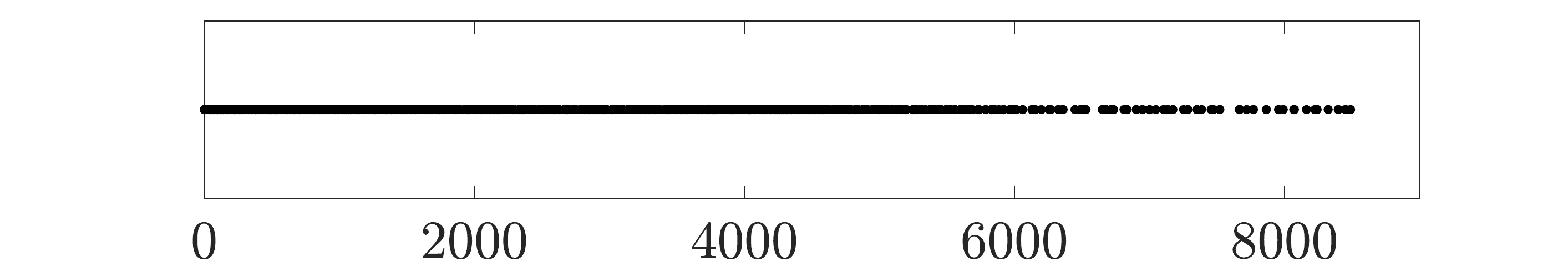}
    		\caption{$Q2$ elements of size $0.12\times 0.12$}
 	 \end{subfigure}\\[.5\baselineskip]
	  \begin{subfigure}[b]{\textwidth}
    		\centerline{\includegraphics[width=.7\textwidth]{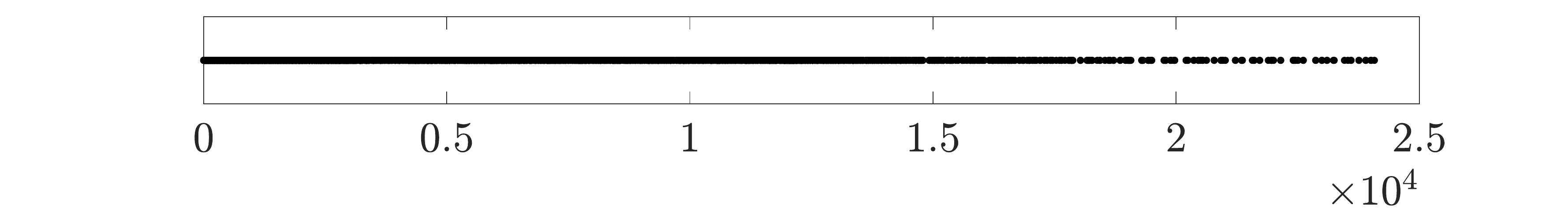}}
    		\caption{$Q3$ elements of size $0.12\times 0.12$}
 	 \end{subfigure}	
	\caption{Distribution of eigenvalues of the operator $\mathcal{L}$ for the diffusion problem discretized using continuous $Q1$, $Q2$, or $Q3$ elements.\label{fig:evals_ex_diff}}
\end{figure}

\begin{figure}[h!t]
	\centerline{
	\includegraphics[width=.53\textwidth]{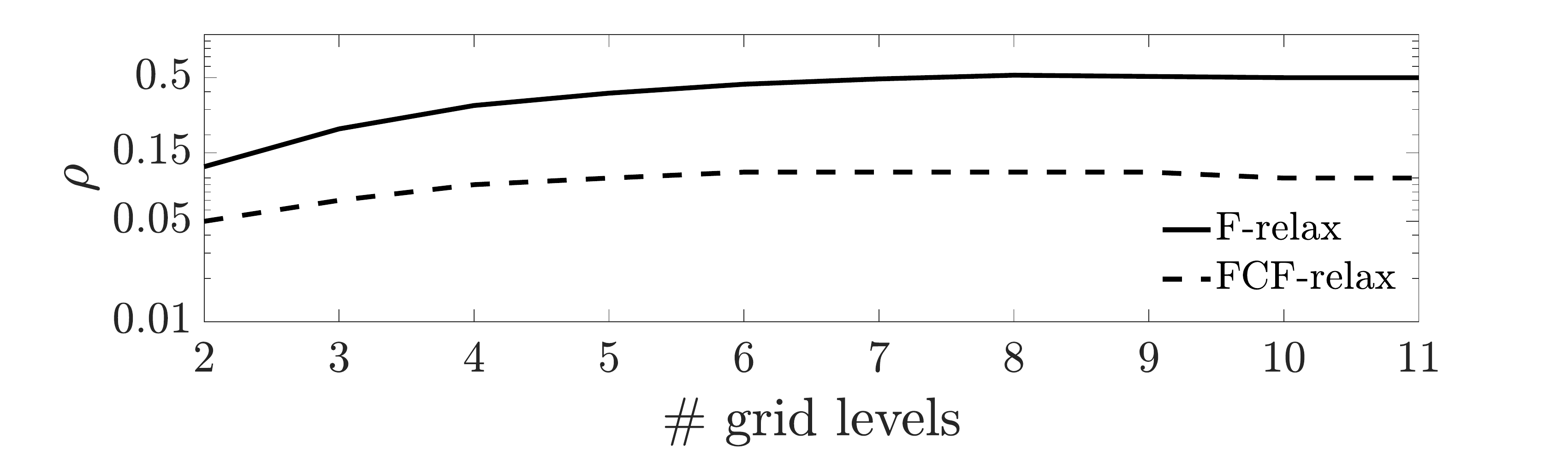}
	\includegraphics[width=.53\textwidth]{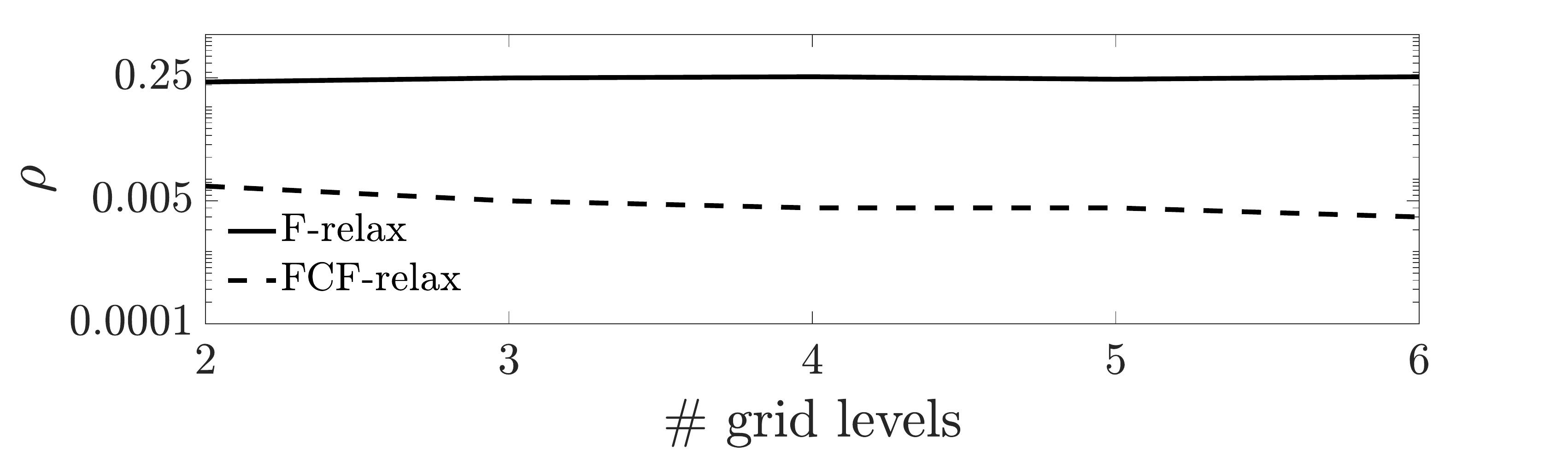}}
	\caption{Maximum convergence factors $\rho$ of MGRIT V-cycles with F- and FCF-relaxation. At left, MGRIT with factor-2 coarsening for backward Euler in time and $Q1$ elements in space and at right, MGRIT with factor-4 coarsening for ESDIRK-32 in time and $Q2$ elements in space. \label{fig:mg_conv}}
\end{figure}

\subsection{A-stable time integration methods}

The SDIRK schemes considered in the previous section yield fast MGRIT convergence. However, the cost (in terms of number
of linear solves) grows with the order of the method. Some ESDIRK schemes such as ESDIRK-33 or the trapezoid method are
not only computationally cheaper, but also offer properties desired in an integration scheme such as being stiffly accurate. However,
the methods are only A-stable. The theory presented in Section \ref{sec:irk:irk} indicates that while using L-stable Runge-Kutta schemes
results in fast convergence, regardless of the mesh sizes in space and time, this is not necessarily the case for A-stable schemes.
Especially when using F-relaxation, convergence depends heavily on the coarsening factor and on the values of the product $h_t\xi$.
To determine the impact of these restrictions in a practical setting, we consider the diffusion problem \eqref{eq:ex_diff} on the beam
mesh, discretized using bilinear quadrilateral elements in space and the A-stable schemes ESDIRK-33 or the trapezoid method in
time. In this setting, local Fourier analysis \cite{ABrandt_1977b} provides accurate predictions for the eigenvalues of the operator
$\mathcal{L}$ for arbitrary spatial mesh sizes $h$. For $h=1/4$, for example, the maximum eigenvalue of $\mathcal{L}$ is $\xi=384$,
and it increases by a factor of four when decreasing the mesh size by a factor of two. In a first experiment, we consider fixing $h=1/16$
(i.\,e., $\xi\in(0,6144)$) and varying the time-step size from $h_t = 1/8192$ to $h_t = 1/512$. For the smallest time-step size,
$h_t = 1/8192$, the product $h_t\xi$ takes values between 0 and 3/4, and for the largest time-step size, $h_t = 1/512$, $h_t\xi\in(0,12)$.
Results using ESDIRK-33 are shown in Table \ref{tab:ex_esdirk33}. Consistent with theoretical results, we see that two-level MGRIT
with FCF-relaxation is much more robust than Parareal/two-level MGRIT with F-relaxation. While F-relaxation does not lead to a convergent
two-level method for large values of $h_t\xi$, convergence of two-level MGRIT with FCF-relaxation only degrades for small coarsening
factors as $h_t\xi$ gets larger. Indeed, the factor of $|\lambda_i|^k\ll 1$ that arises in FCF-relaxation convergence bounds appears to
be fundamental to the success of this scheme.

\begin{table}[h!t]
\renewcommand{\tabcolsep}{0.195cm}
\renewcommand{\arraystretch}{1.2}
\centering
\begin{tabular}{|>{\raggedright}m{.23\textwidth}|>{\centering}m{.03\textwidth}|>{\centering}m{.09\textwidth}|>{\centering}m{.09\textwidth}|>{\centering}m{.09\textwidth}|>{\centering}m{.09\textwidth}|>{\centering}m{.09\textwidth}|>{\centering\arraybackslash}m{.09\textwidth}|}
	\hline\multicolumn{2}{|r|}{$k = ~$} & 2 & 3 & 4 & 5 & 8 & 16 \\ 
	\hline\hline
	$h_t = 1/8192 ~(h_t\xi \in (0,3/4))$ & $\rho$ & 0.01/0.01 & 0.03/0.007 & 0.07/0.009 & 0.09/0.009 & 0.50/0.02 & >1/0.01 \\
	$h_t = 1/4096 ~(h_t\xi \in (0,3/2))$ & $\rho$ & 0.04/0.007 & 0.18/0.03 & 0.50/0.02 & 0.69/0.01 & >1/0.01 & >1/0.01 \\
	$h_t = 1/2048 ~(h_t\xi \in (0,3))$ & $\rho$ & 0.52/0.02 & 0.85/0.01 & >1/0.01 & >1/0.01 & >1/0.01 & >1/0.01 \\
	$h_t = 1/1024 ~(h_t\xi \in (0,6))$ & $\rho$ & >1/0.02 & >1/0.01 & >1/0.009 & >1/0.01 & >1/0.01 & >1/0.008 \\
	$h_t = 1/512 ~(h_t\xi \in (0,12))$ & $\rho$ & >1/0.6 & >1/0.02 & >1/0.01 & >1/0.01 & >1/0.08 & >1/0.01 \\
	\hline
\end{tabular}
\caption{Maximum two-level convergence factors $\rho$ for using ESDIRK-33 in time and $Q1$ elements of size $h\times h$,
$h=1/16$ in space. Convergence factors for using F-relaxation are given first, followed by convergence factors 
using FCF-relaxation; the notation $>1$ indicates no convergence to the desired tolerance.
}
\label{tab:ex_esdirk33}
\end{table}

The results in Table \ref{tab:ex_esdirk33} show that convergence of two-level MGRIT with F- or FCF-relaxation is generally quite
good for $h_t\xi\in (0,1)$. For large $h_t\xi$, FCF-relaxation is needed to obtain a convergent two-level method when using
ESDIRK-33 on the fine and on the coarse grid. The theory in Section \ref{sec:irk:irk} indicates that for $h_t\xi$ on the order
of 10--100, in addition to FCF-relaxation, odd or large coarsening factors may be needed for (good) convergence. Table
\ref{tab:ex_esdirk33_even_odd} considers a similar experiment as in Table \ref{tab:ex_esdirk33}, but with larger values of
$h_t\xi$ and only MGRIT with FCF-relaxation. More precisely, we consider the same setting as in Table \ref{tab:ex_esdirk33},
but we fix $h=1/32$ and consider $h_t=1/512$, $h_t=1/256$ and $h_t=1/128$. Results do not exactly agree with the theory,
but are consistent on a conceptual level. The last row is particularly representative ($h_t = 1/128$), where divergence is
observed for $k=2$, $k=3$ converges more than $2\times$ as fast as $k=4$, and for $k=5$ and greater, very fast convergence
is obtained. In general, convergence of MGRIT with factor-16 coarsening does not degrade as $h_t\xi$ increases, consistent with
Figure \ref{fig:esdirk}.

\begin{table}[h!t]
\renewcommand{\tabcolsep}{0.195cm}
\renewcommand{\arraystretch}{1.2}
\centering
\begin{tabular}{|>{\centering}m{.22\textwidth}|>{\centering}m{.03\textwidth}|>{\centering}m{.07\textwidth}|>{\centering}m{.07\textwidth}|>{\centering}m{.07\textwidth}|>{\centering}m{.07\textwidth}|>{\centering}m{.07\textwidth}|>{\centering\arraybackslash}m{.07\textwidth}|}
	\hline\multicolumn{2}{|r|}{$k = ~$} & 2 & 3 & 4 & 5 & 8 & 16\\ 
	\hline\hline
	$h_t = 1/512 ~(h_t\xi \in (0,48))$ & $\rho$ & >1 & 0.31 & 0.39 & 0.01 & 0.008 & 0.01 \\
	$h_t = 1/256 ~(h_t\xi \in (0,96))$ & $\rho$ & >1 & 0.41 & 0.65 & 0.03 & 0.01 & 0.009 \\
	$h_t = 1/128 ~(h_t\xi \in (0,192))$ & $\rho$ & >1 & 0.40 & 0.70 & 0.02 & 0.01 & $0.006^*$ \\
	\hline
\end{tabular}
\caption{Results similar to those of Table \ref{tab:ex_esdirk33} for larger values of $h_t\xi$ and only FCF-relaxation; the $\ast$ superscript indicates that results may be inaccurate due to the exactness property of the MGRIT algorithm, i.\,e., convergence to the fine-grid solution in $N/2k$ iterations.
}
\label{tab:ex_esdirk33_even_odd}
\end{table}

Figure \ref{fig:trap} shows that MGRIT convergence is even more sensitive to small changes in $h_t\xi$ when the trapezoid
method is used. We demonstrate this in a similar setting as considered in Table \ref{tab:ex_esdirk33}. Again, we fix $h = 1/16$
and consider three time-step sizes, $h_t = 1/2048$, $h_t = 1/1024$, and $h_t = 1/512$. From the results in Table \ref{tab:ex_tr},
we see that a simple doubling of the time-step size can dramatically affect convergence, in the worst case changing from
convergence factors of $0.02$ to divergence (observed for factor-4 coarsening and $h_t = 1/1024$).

\begin{table}[h!t]
\renewcommand{\tabcolsep}{0.195cm}
\renewcommand{\arraystretch}{1.2}
\centering
\begin{tabular}{|>{\centering}m{.22\textwidth}|>{\centering}m{.03\textwidth}|>{\centering}m{.07\textwidth}|>{\centering}m{.07\textwidth}|>{\centering}m{.07\textwidth}|>{\centering\arraybackslash}m{.07\textwidth}|}
	\hline\multicolumn{2}{|r|}{$k = ~$} & 2 & 4 & 8 & 16\\ 
	\hline\hline
	$h_t = 1/2048 ~(h_t\xi \in (0,3))$ & $\rho$ & 0.01 & 0.02 & 0.02 & 0.02 \\
	$h_t = 1/1024 ~(h_t\xi \in (0,6))$ & $\rho$ & 0.82 & 0.02 & 0.02 & 0.02 \\
	$h_t = 1/512 ~(h_t\xi \in (0,12))$ & $\rho$ & >1 & >1 & 0.41 & 0.02 \\
	\hline
\end{tabular}
\caption{Results similar to those of Table \ref{tab:ex_esdirk33} for using the trapezoid method instead of ESDIRK-33 and only FCF-relaxation and even coarsening factors.\label{tab:ex_tr}}
\end{table}

Observing that two-level convergence heavily depends on the value of $h_t\xi$ raises the question of how full multilevel
convergence can be affected by this value. In Figure \ref{fig:tr_conv}, we consider multilevel MGRIT convergence for three
cases from Table \ref{tab:ex_tr} for which two-level convergence is fast, namely factor-2 and factor-4 coarsening for
$h_t = 1/2048$ and factor-4 coarsening for $h_t = 1/1024$. The plots show that although $h_t\xi$ gets larger with every
coarse grid level (i.\,e., $h_t\xi$ doubles for $k=2$ and quadruples for $k=4$), multilevel convergence is still good in most
cases. For both cases shown in the left plot, multilevel convergence is hardly affected, whereas convergence degrades from
about 0.02 to 0.4 when considering more than two grid levels in the case of $h_t = 1/1024$ and $k = 4$. This is somewhat
different than previously mentioned results (for example, Figure \ref{fig:mg_conv}), and indicates that further investigation
into the multilevel setting is needed. 

\begin{figure}[h!t]
	\centerline{
	\includegraphics[width=.53\textwidth]{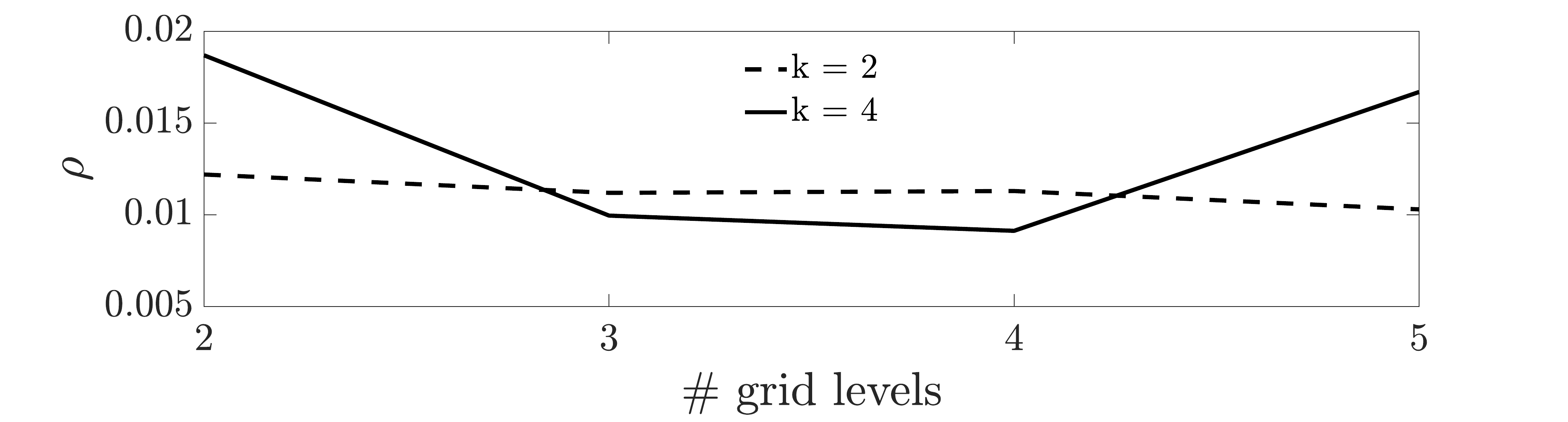}
	\includegraphics[width=.53\textwidth]{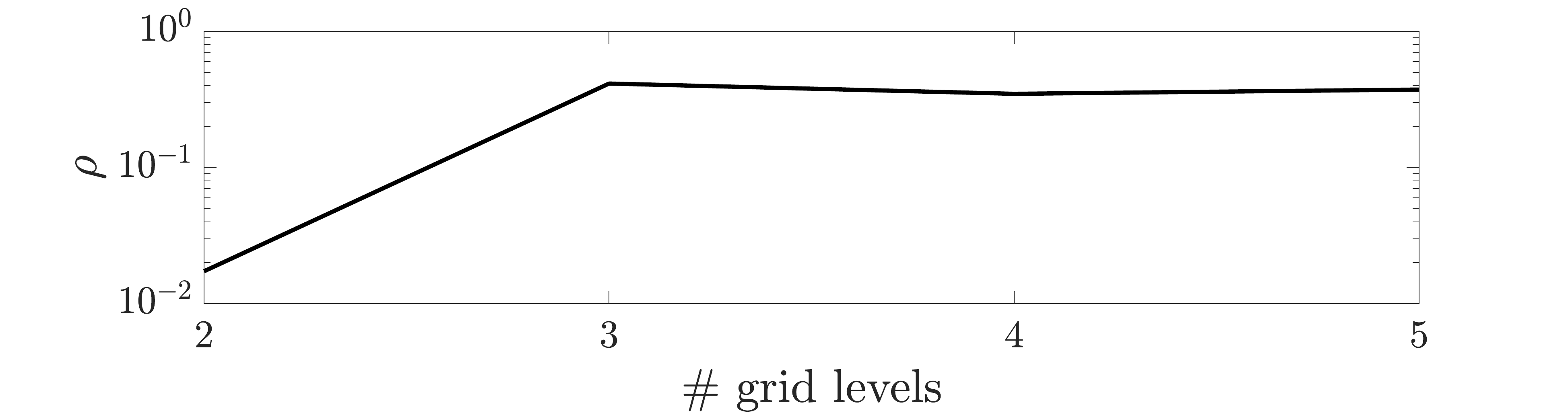}}
	\caption{Maximum convergence factors $\rho$ of MGRIT V-cycles with FCF-relaxation for trapezoid method in
	time and $Q1$ elements in space. At left, $h_t = 1/2048$ and at right, $h_t = 1/1024$. \label{fig:tr_conv}}
\end{figure}

\subsection{Combination of A- and L-stable time integration methods}

One particularly interesting result from Sections \ref{sec:irk:irk} and \ref{sec:irk:L} is that convergence of Parareal and two-level MGRIT
with ESDIRK-33 or with the trapezoid method can be improved or even ``fixed'' by using an L-stable scheme on the coarse grid.
Here, we consider this modification for ESDIRK-33. In particular, we test using the two L-stable time-stepping schemes backward
Euler and ESDIRK-32 on the coarse grid. Furthermore, to allow for comparison, we choose the same parameters as considered
in Table \ref{tab:ex_esdirk33}, that is, we fix $h=1/16$ and test with the time-step sizes $h_t=1/1024$ and $h_t=1/512$. Table
\ref{tab:ex_esdirk33_cg} shows results of these experiments for two-level MGRIT with F- and FCF-relaxation. We see that the
experimentally measured maximum convergence factors correspond well to the convergence bounds of Figure \ref{fig:esdirk}.
Comparing with the results in Table \ref{tab:ex_esdirk33}, Parareal/MGRIT with F-relaxation now converges in all tests. For FCF-relaxation,
while using backward Euler is not beneficial, ESDIRK-32 offers fast and more robust convergence. Moreover, experiments with
multilevel V-cycles (shown in Figure \ref{fig:esdirk33_cg_conv}) demonstrate that the strategy of using L-stable coarse-grid
time-integration methods can be applied successfully in a multilevel setting, especially when using FCF-relaxation.

\begin{table}[h!t]
\renewcommand{\tabcolsep}{0.195cm}
\renewcommand{\arraystretch}{1.2}
\centering
\begin{tabular}{|>{\centering}m{.1\textwidth}|>{\centering}m{.12\textwidth}|>{\centering}m{.03\textwidth}|>{\centering}m{.09\textwidth}|>{\centering}m{.09\textwidth}|>{\centering}m{.09\textwidth}|>{\centering}m{.09\textwidth}|>{\centering}m{.09\textwidth}|>{\centering\arraybackslash}m{.09\textwidth}|}
	\hline\multicolumn{3}{|r|}{$k = ~$} & 2 & 3 & 4 & 5 & 8 & 16 \\ 
	\hline\hline
	 & $h_t = 1/1024$ & $\rho$ & 0.31/0.10 & 0.29/0.11 & 0.30/0.11 & 0.30/0.10 & 0.29/0.10 & 0.28/0.11 \\
	ESDIRK-33 & $(h_t\xi \in (0,6))$ & & & & & & & \\
	+ BWE & $h_t = 1/512$ & $\rho$ & 0.31/0.11 & 0.29/0.09 & 0.29/0.10 & 0.30/0.10 & 0.28/0.11 & 0.29/0.09 \\
	& $(h_t\xi \in (0,12))$ & & & & & & & \\
	\hline\hline
	ESDIRK-33 & $h_t = 1/1024$ & $\rho$ & 0.24/0.006 & 0.24/0.007 & 0.24/0.01 & 0.24/0.01 & 0.24/0.009 & 0.24/0.01 \\
	+ & $(h_t\xi \in (0,6))$ & & & & & & & \\
	ESDIRK-32 & $h_t = 1/512$ & $\rho$ & 0.38/0.04 & 0.25/0.009 & 0.24/0.009 & 0.25/0.007 & 0.24/0.01 & 0.23/0.01 \\
	& $(h_t\xi \in (0,12))$ & & & & & & & \\
	\hline
\end{tabular}
\caption{Results similar to those of Table \ref{tab:ex_esdirk33} for using ESDIRK-33 only on the fine grid and BWE or ESDIRK-32 on the coarse grid.}
\label{tab:ex_esdirk33_cg}
\end{table}

\begin{figure}[h!t]
	\centerline
	{\includegraphics[width=.53\textwidth]{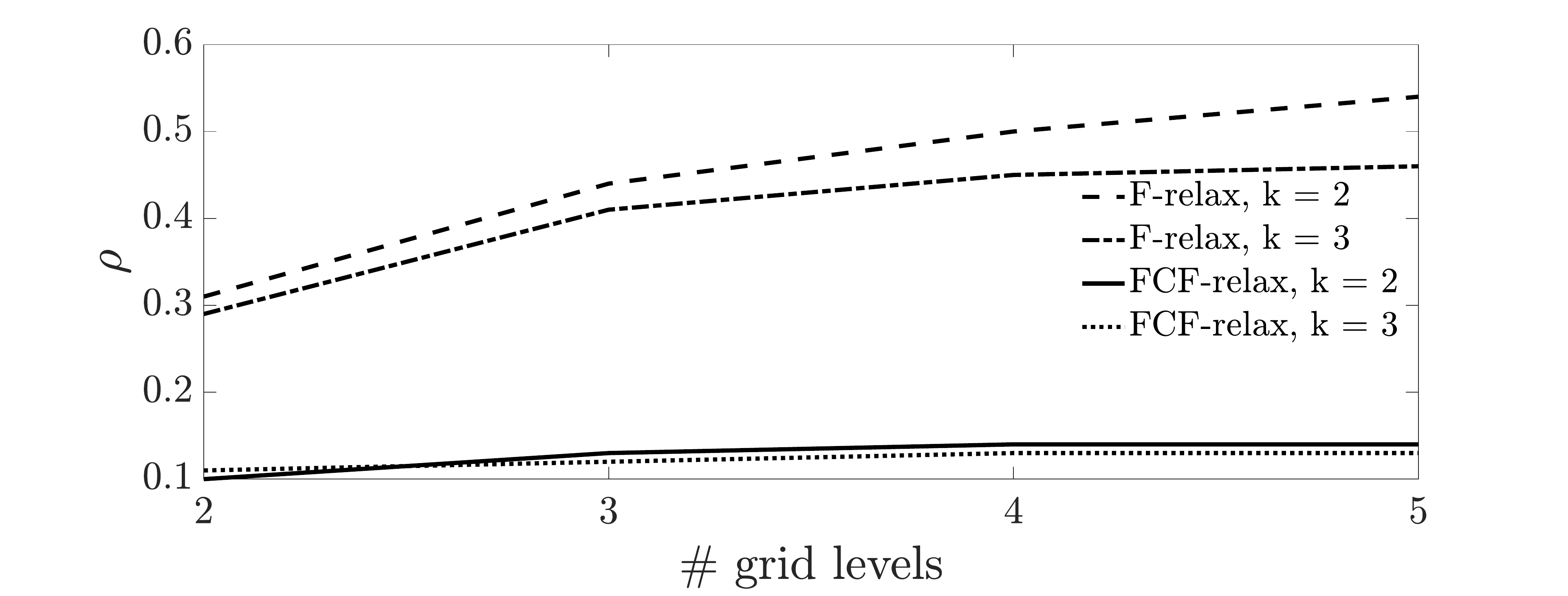}
	\includegraphics[width=.53\textwidth]{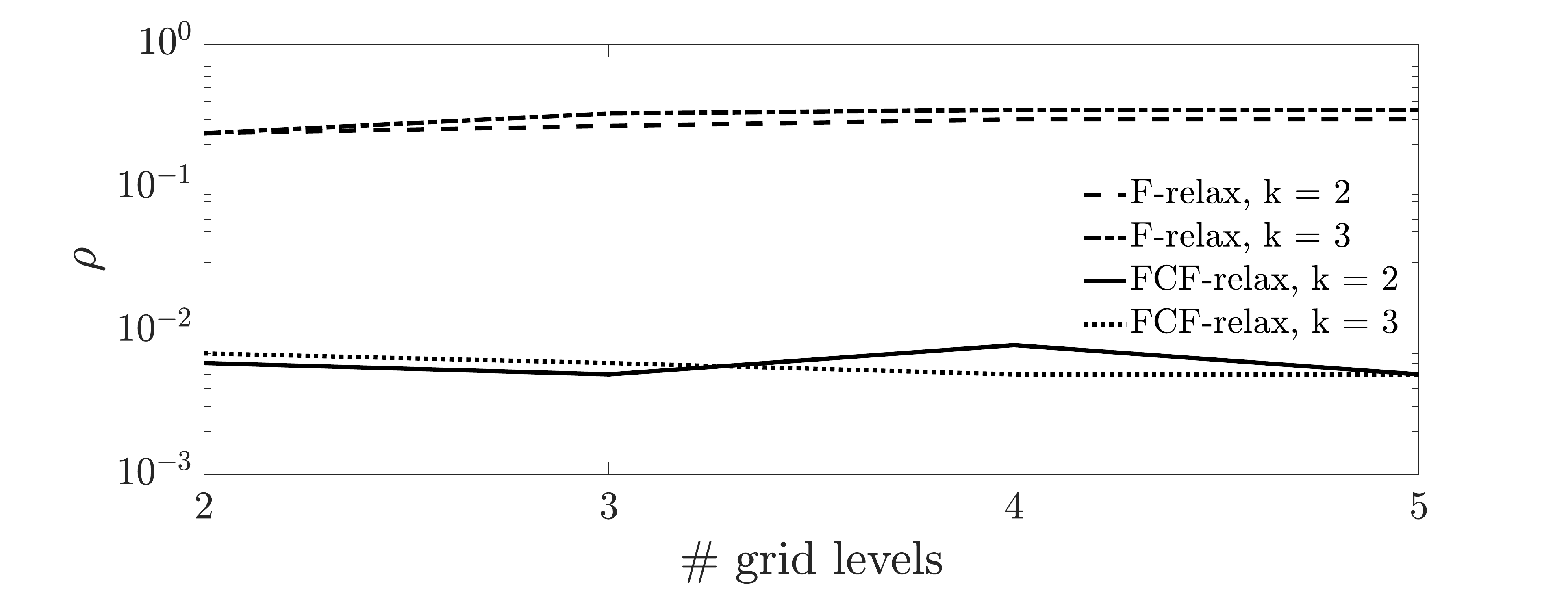}}
	\caption{Maximum convergence factors $\rho$ of MGRIT V-cycles with ESDIRK-33 on the fine grid and backward Euler
	(left) or ESDIRK-32 (right) on the coarse grid; $Q1$ elements with $h = 1/16$ in space, time-step size $h_t = 1/1024$.
	\label{fig:esdirk33_cg_conv}}
\end{figure}

\begin{remark}
	The numerical experiments in this section were carried out on grids with relatively large spatial mesh sizes. Note, however,
	that Parareal/MGRIT convergence only depends on the value of $h_t\xi$, or, equivalently, on a function of the mesh sizes in space
	and time. For the diffusion problem, this function only depends on the ratio $h_t/h^2$. Results on finer grids show the same
	behavior, but require higher computational costs. Scaling tests in $h_t$ and $h_x$ can be found, for example, in Falgout et al.\cite{RDFalgout_etal_2014}
\end{remark}

\section{Skew-symmetric operators}\label{sec:imag}

{

So far we have only considered SPD spatial operators. The primary reason for this is that SPD operators are
unitarily diagonalizable, in which case the theory in \cite{BSSouthworth_2018a} provides tight upper and lower
bounds on convergence of the error and residual of two-level MGRIT and Parareal in the $\ell^2$-norm. This yields
necessary and sufficient conditions, where we can prove that an RK scheme either does or does not observe
convergence independent of spatial and temporal problem size, in two of the standard measures of convergence
for iterative methods. 

In general, an operator is unitarily diagonalizable if and only if it is normal. Limiting ourselves to real-valued
operators, the other general subclass of normal matrices (in addition to symmetric matrices) are skew-symmetric
matrices, $A = -A^T$. Skew symmetric matrices have purely imaginary eigenvalues and can arise in the
discretization of hyperbolic PDEs (for example, a central finite difference scheme for advection is skew symmetric,
or the second-order wave equation in first-order form is almost skew symmetric and has purely imaginary
eigenvalues \cite{optimal20}). This section considers convergence of Parareal/MGRIT applied to a spatial operator with purely
imaginary eigenvalues, and using various implicit and explicit Runge-Kutta schemes. Convergence bounds
for Parareal/MGRIT with F-relaxation are shown in Figure \ref{fig:evalsimag} as a function of $h_t\xi$, 
where $\xi$ is purely imaginary. There are a few interesting things to note:
\begin{enumerate}
\item Backward Euler appears to be the only Runge-Kutta scheme for which one can obtain ``optimal''
convergence independent of $N_c$ for skew-symmetric operators.
All other Runge-Kutta schemes we have tested, including those shown
in Figure \ref{fig:evalsimag} as well as others not published here, stop converging for $h_t\xi \lessapprox 1$
as $N_c\to\infty$. We hypothesize there is a rigorous explanation for this based on looking at Runge-Kutta
schemes as rational approximations to the exponential, but we have been unable to identify a proof. Nevertheless,
this again highlights the usefulness of this work: it is straightforward to evaluate and compare tight convergence
bounds for arbitrary RK schemes, without thorough analyses for each. 

\item The CFL stability limit on the coarse temporal grid approximately corresponds to where the coarse-grid
eigenvalue $|\mu|$ (solid black line) increases above one for explicit schemes (left column of plots). This
occurs between $h_t\xi \in [0.1,0.5]$ depending on the scheme. Convergence of MGRIT,
regardless of implicit/explicit or $N_c$, requires \textit{at least} CFL ``stability'' on the coarse grid
for convergence and typically a stronger stability $h_t\xi \ll 0.5$. The benefit of implicit is that the
underlying RK scheme is stable on arbitrarily coarse grids, which has allowed for some success with
multilevel MGRIT \cite{mgrit19}. Multilevel is generally not feasible using explicit schemes anywhere 
close to the CFL limit as the residual will explode when the CLF limit is violated on some grid. 

\item Plots for FCF-relaxation are nearly identical for all schemes shown. Recall all RK schemes are
a rational approximation to the exponential, and the exponential evaluated at imaginary values always
has magnitude one. Thus for any reasonable approximation, the eigenvalues of $\Phi$, particularly
anywhere moderately close to the origin, satisfy $|\lambda_i|\approx 1$, and the additional constant
in convergence bounds provided by FCF-relaxation, $|\lambda_i|^k$, is of minimal benefit.

\item Even backward Euler is limited by coarsening factor, $k$. Figure \ref{fig:evalsimag} uses
$k=4$, and backward Euler observes worst-case convergence of $\varphi_F = \varphi_{FCF} = 0.75$,
independent of $N_c$. However, this bound increases to $\varphi_F = \varphi_{FCF} = 0.875, 0.9375,$
and $0.96875$ for $k=8,16$ and $32$, respectively (not shown in plots), while the Parareal/MGRIT
stability limit on $h_t\xi$ to observe convergence steadily decreases.

\end{enumerate}

\begin{figure}[!ht]
    \centering
    \begin{center}
        \begin{subfigure}[c]{0.285\textwidth}
            \includegraphics[width=\textwidth]{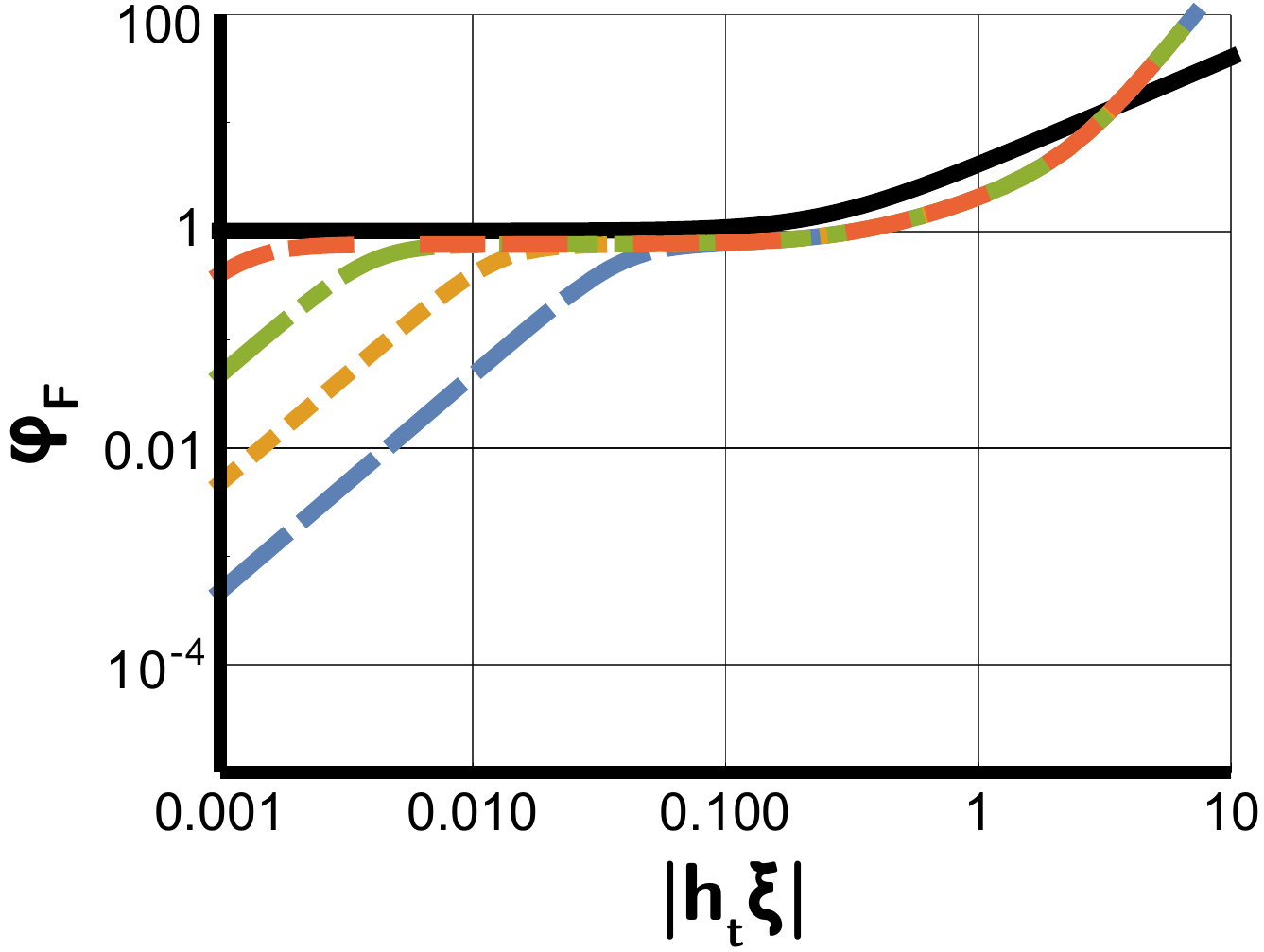}
            \caption{Forward Euler}
        \end{subfigure}
        \begin{subfigure}[c]{0.285\textwidth}
            \includegraphics[width=\textwidth]{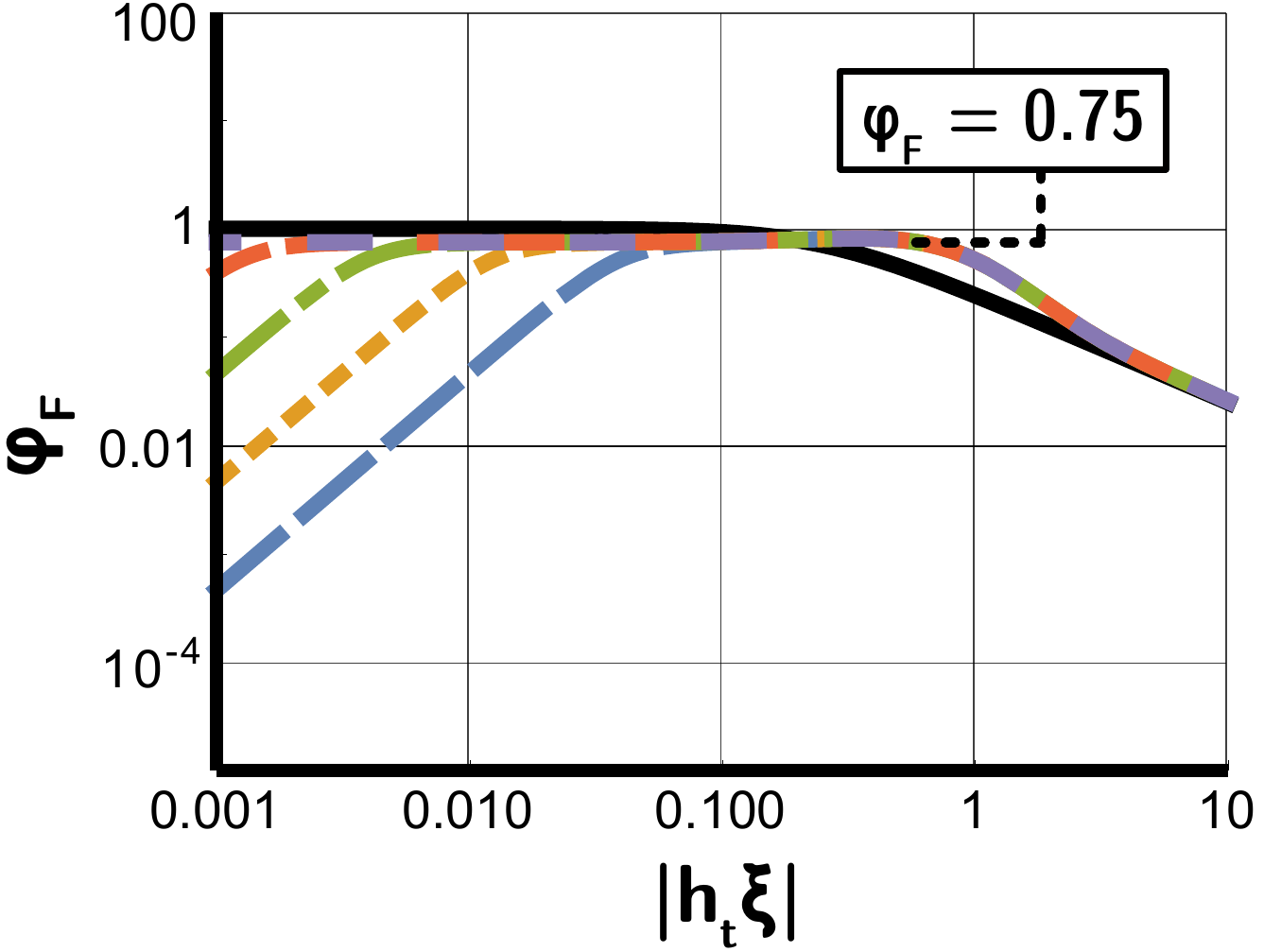}
            \caption{Backward Euler (L)}
        \end{subfigure}
        \hspace{7ex}
        \begin{subfigure}[c]{0.18\textwidth}
            \includegraphics[width=\textwidth]{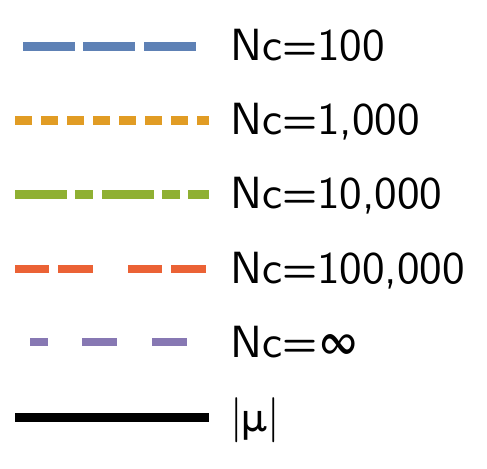}
            \vspace{2ex}
        \end{subfigure}
\\
        \begin{subfigure}[c]{0.285\textwidth}
            \includegraphics[width=\textwidth]{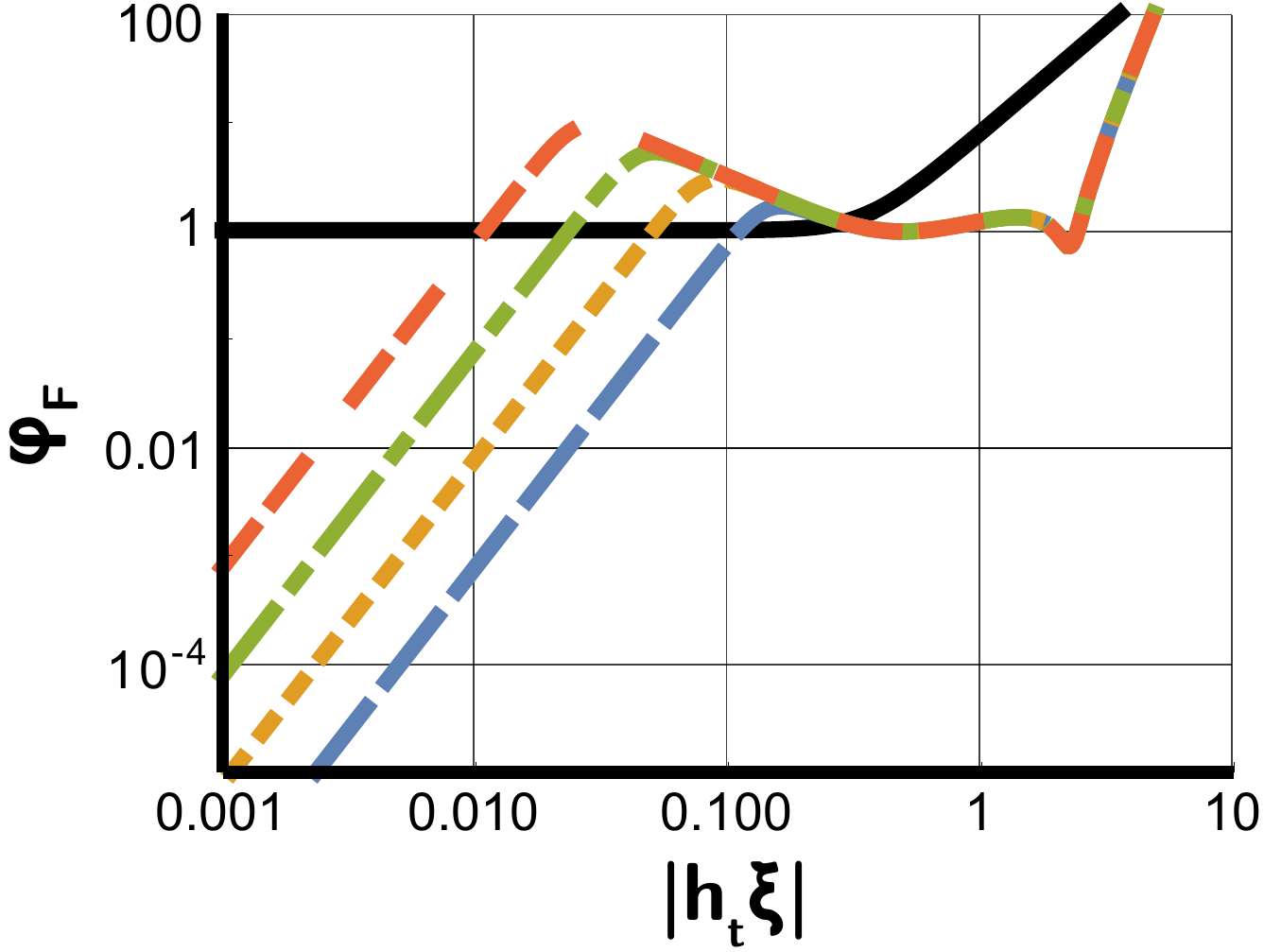}
            \caption{ERK2}
        \end{subfigure}
        \begin{subfigure}[c]{0.285\textwidth}
            \includegraphics[width=\textwidth]{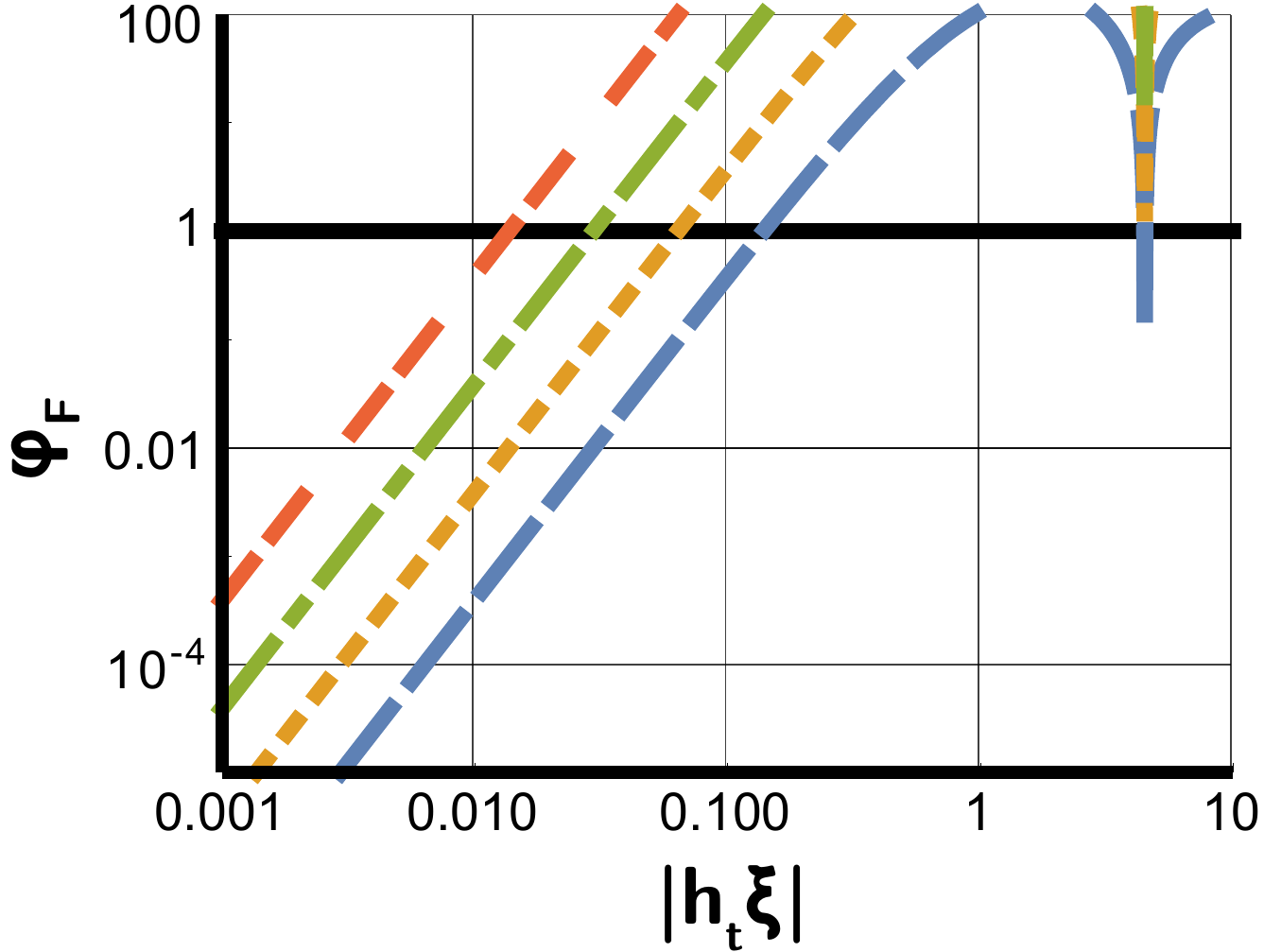}
            \caption{Trapezoid rule (A)}
        \end{subfigure}
        \begin{subfigure}[c]{0.275\textwidth}
            \includegraphics[width=\textwidth]{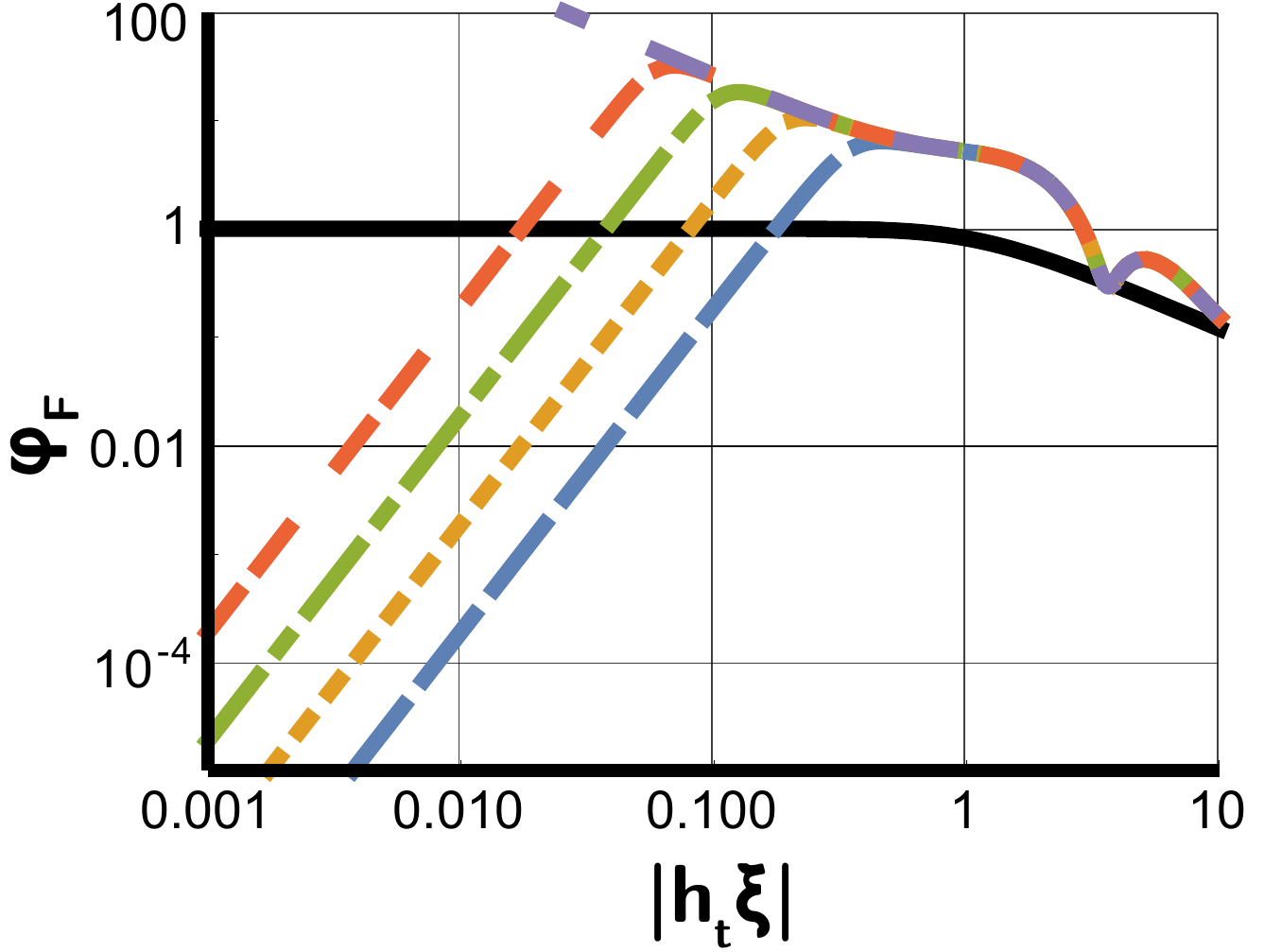}
            \caption{SDIRK2 (L)}
        \end{subfigure}
\\
        \begin{subfigure}[c]{0.285\textwidth}
            \includegraphics[width=\textwidth]{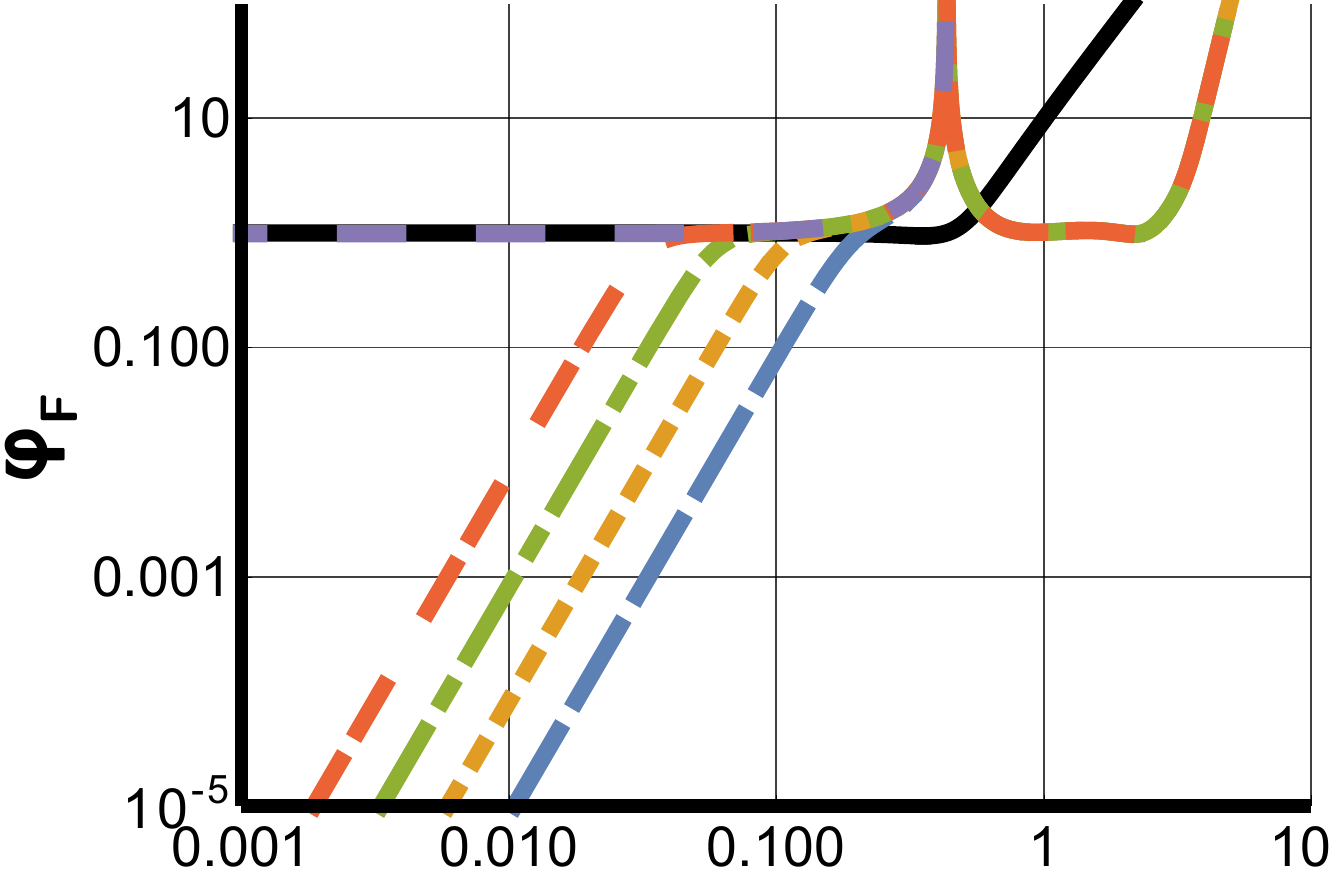}
            \caption{ERK3}
        \end{subfigure}
        \begin{subfigure}[c]{0.285\textwidth}
            \includegraphics[width=\textwidth]{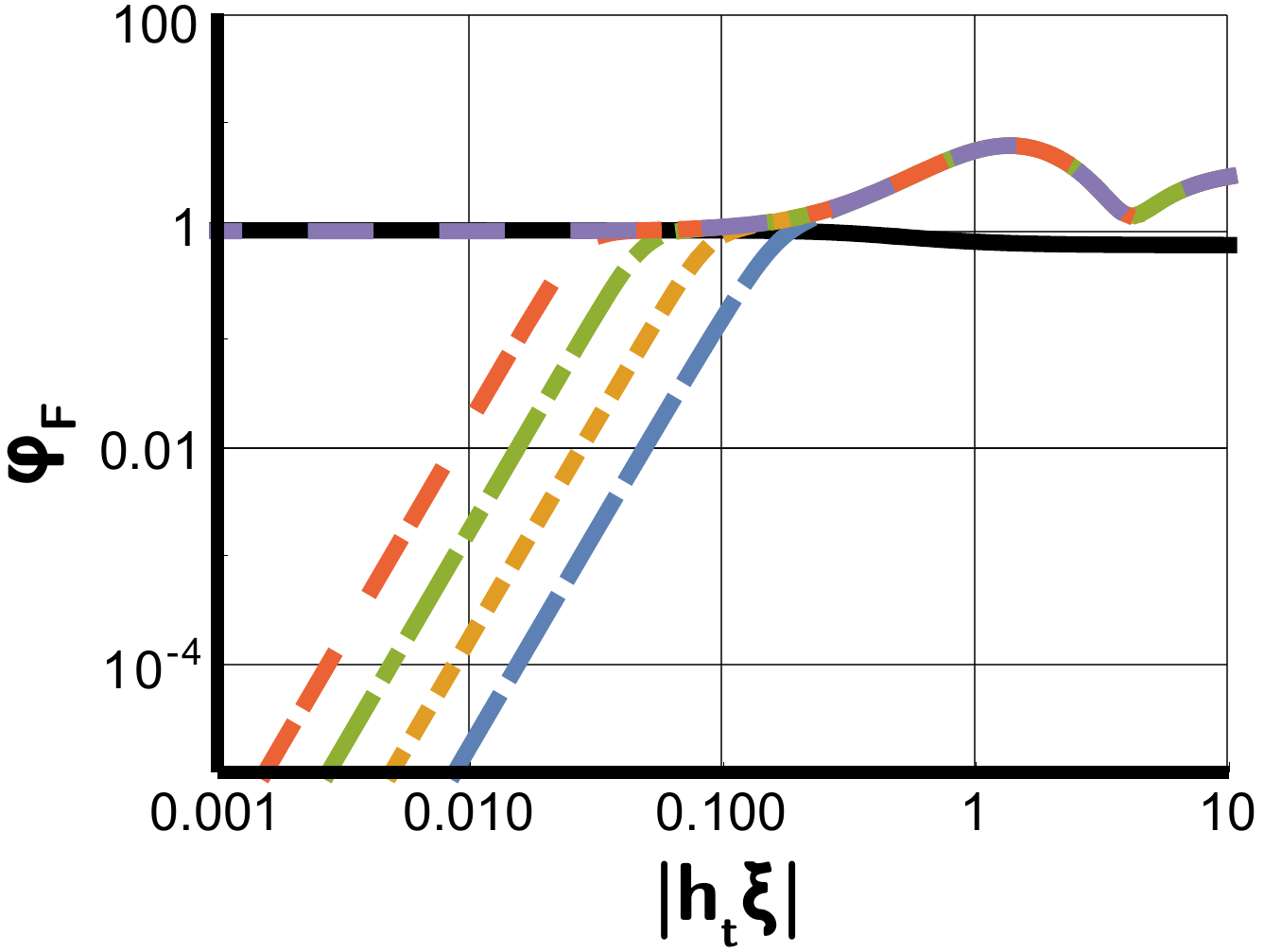}
            \caption{SDIRK3 (A)}
        \end{subfigure}
        \begin{subfigure}[c]{0.285\textwidth}
            \includegraphics[width=\textwidth]{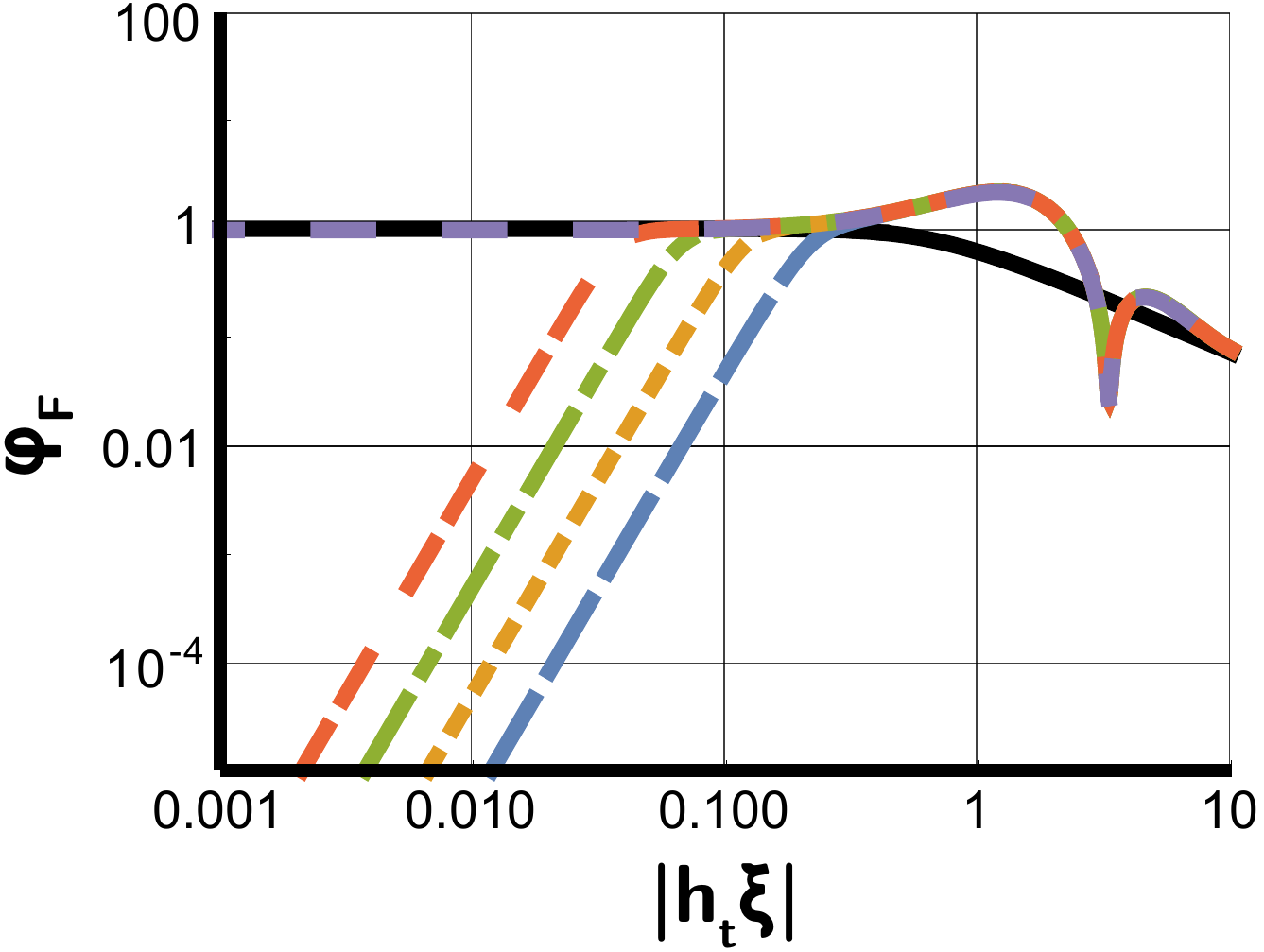}
            \caption{SDIRK3 (L)}
        \end{subfigure}
    \end{center}
    \caption{Convergence bounds of two-level MRGIT/Parareal with F-relaxation and coarsening
    factor $k=4$, as a function of $h_t|\xi|$, for purely imaginary spatial eigenvalue $\xi$. Results
    are shown for various numbers of coarse-grid time steps, with $N_c=\infty$ denoting the upper
    bound on convergence with no dependence on $N_c$.}
    \label{fig:evalsimag}
\end{figure}

It is worth briefly commenting on the distinction between skew-symmetric and SPD operators and
the corresponding convergence estimates. Recall when considering SPD operators and real
eigenvalues, we ignored the $O(N_c)$ term in the eigenvalue convergence bounds. 
Indeed, it is easily verified (via, for example, plotting the functions) that for real eigenvalues,
convergence bounds are virtually identical for $N_c = 100$ or $N_c = \infty$ for any of the
schemes considered in this paper. The key to this was proven in Corollary \ref{cor:limits}, where
it is shown that $\varphi_F$ and $\varphi_{FCF}$ \textit{continuously} limit to zero as $h_t\xi\to0$. 
Furthermore, for real-valued spatial eigenvalues, $|\mu_i|$ is only close to one when $h_t\xi$ is
close to zero, and $|\mu_i|\approx 1$ is the only situation when the $\mathcal{O}(N_c)$ term
in \eqref{eq:v_bounds} and \eqref{eq:v_bounds2} matters.

Complex eigenvalues violate both
of these properties. As mentioned in Remark \ref{rem:complex}, the continuous limit is not
well-defined in the complex plane, so even though convergence may be perfect (zero) at
$h_t\xi = 0$, convergence arbitrarily close to $h_t\xi = 0$ may be poor. Moreover, for any RK scheme,
eigenvalues of $\Phi$ and $\Psi$ will have magnitude close to one for $h_t\xi$ along a significant
portion of the imaginary axis (approximating the exponential), and the $\mathcal{O}(N_c)$ term
in convergence bounds is significant for all such eigenvalues. Thus, for purely imaginary (and
generally complex) eigenvalues, problem size matters. In the more general complex setting,
it is even more complicated to consider convergence independent of problem size because the
structure of spatial eigenvalues in the complex plane can change as $h$ changes, in addition
to the fact that convergence bounds are generally more complicated (see the following remark). 

\begin{remark}
This work has only considered convergence of Parareal/MGRIT applied to SPD or skew-symmetric
operators so that optimal convergence can be proven for error and residual propagation
in the $\ell^2$-norm. The theoretical results in \cite{BSSouthworth_2018a} are more general and
can be useful for the analysis of non-normal operators, however, proving optimal convergence in the
more general setting is more complicated. In particular, similar eigenvalue convergence bounds hold in a
$(UU^*)^{-1}$-norm, for spatial eigenvectors $U$. However, as the
spatial mesh is refined $UU^*$ is likely to become increasingly ill-conditioned, which (i) changes
the norm in which we are considering convergence, and (ii) may be decreasingly representative of
observed convergence in an $\ell^2$-sense. Analysis of convergence for non-normal operators and
specific RK schemes has been done in, for example, \cite{Wu:2017ih}. A similar scenario holds
for those results -- convergence is considered on an eigenvector-by-eigenvector basis, but if the
eigenvectors do not form an orthogonal basis for the space, the theory does not yield tight bounds
on convergence of arbitrary error modes. That is not to say such analyses are not useful, rather that
the focus of this paper is on rigorous analysis of convergence independent of problem size, and
eigenvalue bounds do not necessarily provide this for non-normal operators. Note, more general
theory that is not based on eigenvectors was also developed in \cite{BSSouthworth_2018a}, but
requires independent analysis for each problem, so is not considered here. 
\end{remark}

}

\section{Conclusions}\label{sec:conclusions}

This paper provides a theoretical framework for analyzing the convergence of two-level MGRIT and Parareal algorithms
for linear problems of the form $\mathbf{u}'(t) + \mathcal{L}\mathbf{u}(t) = f(t)$, where $\mathcal{L}$ is SPD \revision{or skew symmetric}
and Runge-Kutta time integration
is used. Tight convergence bounds are provided in terms of the product $h_t\xi$, where $h_t$ denotes the time-step size and $\xi$ are the
spatial eigenvalues of the operator $\mathcal{L}$. Several important observations come from the theory presented in this paper \revision{for SPD
operators}. First, two-level
MGRIT and Parareal using arbitrary Runge-Kutta schemes on the fine and coarse grid is guaranteed to converge rapidly for small values of
$h_t\xi$. Secondly, using an L-stable coarse-grid operator with an A-stable or L-stable fine-grid operator cannot result in divergence for large
values of $h_t\xi$. Third, the a priori analysis can motivate algorithmic choices such as the choice of the coarsening factor, relaxation scheme,
or coarse-grid integration method, that make the difference between convergence and divergence, or yield speedups of several times. A final
benefit is the versatility of the presented theory that allows for straightforward convergence analysis of arbitrary coarse- and fine-grid operators,
coarsening factors, and time-step sizes, and also encompasses analysis of various modified Parareal algorithms, such as the $\theta$-Parareal
method\cite{GAriel_etal_2018a} and modified A-/L-stable fine-grid propagators introduced in \cite{SLWu_2016}. For explicit schemes, it is shown
that using the same method on the coarse grid (with time-step size $kh_t$) as on the fine grid (with time step $h_t$) is optimal in a Taylor sense.
Thus, the analysis can be tractable, but developing effective parallel-in-time algorithms for explicit methods is generally more difficult than for
implicit methods. \revision{An interesting insight of the presented theory for the $\theta$-Parareal scheme is that applying two iterations of Parareal/
two-level  MGRIT with F-relaxation, first with weight $\theta=1$ followed by $\theta=0$ is exactly equivalent to one iteration of MGRIT with 
FCF-relaxation. For skew-symmetric operators, we show that backward Euler appears to be the only RK scheme that can provide
optimal convergence (proving many other standard schemes do not), but even there the convergence depends heavily on the coarsening factor.}

Future work includes extending the theory to the multilevel setting. The multilevel results presented in Section \ref{sec:ex}
are promising, where we observe fast and robust convergence in settings of fast and robust two-level convergence, and some theory has
been developed in \cite{mgrit19}, but a formal connection between two-level and full multilevel convergence is not yet understood. 
A Mathematica notebook that performs two-grid analysis for all Runge-Kutta schemes tested here, and to which
other schemes can easily be added, can be found in the \textit{a\_priori/} subfolder of the Github repository
\url{https://github.com/XBraid/xbraid-convergence-est}.

\appendix

\begin{lemma}[The trapezoid rule/Crank-Nicolson]\label{lem:crank}
Parareal/MGRIT with F-relaxation using the trapezoid rule/Crank-Nicolson on the fine grid and backward Euler
on the coarse grid can observe arbitrarily slow convergence for large $h_t\xi$. 
\end{lemma}
\begin{proof}
In \cite{TPMathew_etal_2010}, $\theta$-integration is considered on the fine grid, which takes the form
\begin{align*}
(M+\theta h_t\mathcal{L})\mathbf{u}_{i+1} & = (M - (1-\theta) h_t\mathcal{L})\mathbf{u}_i + \mathbf{f}, 
\end{align*}
for mass matrix $M$, spatial operator $\mathcal{L}$, and forcing function $\mathbf{f}$. Here we
focus on the case of $\theta = 1/2$. Rearranging, this corresponds to fine-grid time-integration
scheme $\hat\Phi := (M+ \tfrac{h_t}{2}\mathcal{L})^{-1}(M - \tfrac{h_t}{2}\mathcal{L})$. For a linear problem,
the implicit mid-point and two-stage RK trapezoidal method are equivalent and also referred
to as Crank-Nicolson. From \eqref{eq:phi} and Appendix \ref{app:sdirk}, the implicit midpoint
method yields $\Phi := I - h_t(I + \tfrac{h_t}{2}\mathcal{L})^{-1}\mathcal{L} 
= (I + \tfrac{h_t}{2}\mathcal{L})^{-1}(I - \tfrac{h_t}{2}\mathcal{L})^{-1}$. Letting $\mathcal{L} \mapsto M^{-1}\mathcal{L}^{-1}$,
we have $\Phi = \hat\Phi$. 

For simplicity, let us consider $M = I$ (for example, as in finite-difference discretizations). Then,
let $\Phi$ be given by Crank-Nicolson as above, $\Psi$ by backward Euler, and let $k=2$. Eigenvalues
$\lambda(h_t\xi)$ and $\mu(h_t\xi)$ are given by the stability functions,
\begin{align*}
\lambda(h_t\xi)^2 &: = \left(1 - \frac{h_t\xi}{1 + h_t\xi/2}\right)^2 = \frac{(h_t\xi-2)^2}{(h_t\xi+2)^2}, \\
\mu_2(h_t\xi) &: = 1 - \frac{2h_t\xi}{1+2h_t\xi} = \frac{1}{1+2h_t\xi}. 
\end{align*}
Convergence bounds for Parareal/MGRIT with F-relaxation in Theorem \ref{th:diag_tight}
are then straightforward to evaluate, 
\begin{align*}
\varphi_F & = \frac{ \left|\frac{(h_t\xi-2)^2}{(h_t\xi+2)^2} - \frac{1}{1+2h_t\xi}\right| }{ \sqrt{ \left(1 - 
 	\left|\frac{1}{1+2h_t\xi}\right|\right)^2 + \frac{\pi^2}{CN_c^2}\left|\frac{1}{1+2h_t\xi}\right|}}
= \frac{ \left|\frac{2(h_t\xi)^3  - 8(h_t\xi)^2}{2(h_t\xi)^3 + 9(h_t\xi)^2 + 12(h_t\xi) + 4}\right| }{ \sqrt{ \left(1 - 
 	\left|\frac{1}{1+2h_t\xi}\right|\right)^2 + \frac{\pi^2}{CN_c^2}\left|\frac{1}{1+2h_t\xi}\right|}}.
\end{align*}
Note that for (arbitrary) fixed $N_c$ and $C$, $\lim_{(h_t\xi)\to\infty} \varphi_F = 1$. From \cite{BSSouthworth_2018a},
$C = 1$ provides a lower bound on convergence and $C = 6$ provides an upper bound on convergence.
It follows that for any normal spatial operator and $\epsilon > 0$, there always exists a time-step $h_t$
such that Parareal/MGRIT with F-relaxation, using Crank-Nicolson on the fine grid and backward Euler
on the coarse grid, has an eigenvalue $1 - \epsilon$. 
\end{proof}

For completeness and to make this paper self-contained, in the following, we list the Butcher tableaus of the Runge-Kutta methods considered in this paper.

\section{Implicit Runge-Kutta Methods}
\subsection{Singly Diagonally Implicit Runge-Kutta (SDIRK) Methods}\label{app:sdirk}

\begin{center}
\begin{table*}[h!t]\caption{One-stage SDIRK methods\label{tab:1stage}}
\centering
\begin{tabular*}{350pt}{@{\extracolsep\fill}cc@{\extracolsep\fill}}
	\textbf{Backward Euler (order 1, L-stable)} & \textbf{Implicit midpoint rule (order 2, A-stable)} \\
	\begin{tikzpicture}
	\node (0,0) {$\renewcommand\arraystretch{1.8}\arraycolsep=8pt
	\begin{array}
	{c|c}
	1 & 1\\
	\cline{1-2}
	& 1
	\end{array}$};
	\end{tikzpicture} & \begin{tikzpicture}
	\node (0,0) {$\renewcommand\arraystretch{1.8}\arraycolsep=8pt
	\begin{array}
	{c|c}
	\tfrac{1}{2} & \tfrac{1}{2}\\
	\cline{1-2}
	& 1
	\end{array}$};
	\end{tikzpicture}
\end{tabular*}
\end{table*}
\end{center}

\begin{center}
\begin{table*}[h!t]\caption{Two-stage SDIRK methods (taken from Equations (221) and (223) of Kennedy and Carpenter\cite{CAKennedy_MACarpenter_2016})\label{tab:2stage}}
\centering
\begin{tabular*}{350pt}{@{\extracolsep\fill}cc@{\extracolsep\fill}}
	\textbf{ SDIRK22 (order 2, L-stable)} & \textbf{ SDIRK23 (order 3, A-stable)} \\
	\begin{tikzpicture}
	\node (0,0) {$\renewcommand\arraystretch{1.8}\arraycolsep=8pt
	\begin{array}
	{c|cc}
	\gamma & \gamma & 0\\
	1 & 1-\gamma & \gamma\\
	\cline{1-3}
	& 1-\gamma & \gamma
	\end{array}$};
	\end{tikzpicture} & \begin{tikzpicture}
	\node (0,0) {$\renewcommand\arraystretch{1.8}\arraycolsep=8pt
	\begin{array}
	{c|cc}
	\gamma & \gamma & 0\\
	1-\gamma & 1-2\gamma & \gamma \\ 
	\hline
	& \tfrac{1}{2} & \tfrac{1}{2}
	\end{array}$};
	\end{tikzpicture}\\
	$\hspace*{-3em}\gamma = \tfrac{2-\sqrt{2}}{2}$ & $\hspace*{-5em}\gamma = \tfrac{3+\sqrt{3}}{6}$
\end{tabular*}
\end{table*}
\end{center}

\begin{center}
\begin{table*}[h!t]\caption{Three-stage SDIRK methods (taken from Equations (229) and (232) of Kennedy and Carpenter\cite{CAKennedy_MACarpenter_2016})\label{tab:3stage}}
\centering
\begin{tabular*}{350pt}{@{\extracolsep\fill}cc@{\extracolsep\fill}}
	\textbf{ SDIRK33 (order 3, L-stable)} & \textbf{ SDIRK34 (order 4, A-stable)} \\
	\begin{tikzpicture}
	\node (0,0) {$\renewcommand\arraystretch{1.8}\arraycolsep=8pt
	\begin{array}
	{c|ccc}
	\gamma & \gamma & 0 & 0\\
	c & c-\gamma & \gamma & 0\\
	1 & b & 1-b-\gamma & \gamma\\
	\cline{1-4}
	& b & 1-b-\gamma & \gamma
	\end{array}$};
	\end{tikzpicture} & \begin{tikzpicture}
	\node (0,0) {$\renewcommand\arraystretch{1.8}\arraycolsep=8pt
	\begin{array}
	{c|ccc}
	\gamma & \gamma & 0 & 0\\
	\tfrac{1}{2} & \tfrac{1}{2}-\gamma & \gamma & 0\\
	1-\gamma & 2\gamma & 1-4\gamma & \gamma \\ 
	\cline{1-4}
	& b & 1-2b & b
	\end{array}$};
	\end{tikzpicture}\\
	\hspace*{.5em}$\gamma = 0.435866521508458999416019$ & $\gamma = \tfrac{3+2\sqrt{3}\cos(\pi/18)}{6}, ~b = \tfrac{1}{6(1-2\gamma)^2}$\\
	$b = 1.20849664917601007033648$\\
	\hspace*{.5em}$c = 0.717933260754229499708010$
\end{tabular*}
\end{table*}
\end{center}

\subsection{Explicit Singly Diagonally Implicit Runge-Kutta (ESDIRK) Methods}\label{app:esdirk}
\begin{center}
\begin{table*}[h!t]\caption{ESDIRK methods (taken from Sections 4.1.1 and 4.2.1 of Kennedy and Carpenter\cite{CAKennedy_MACarpenter_2016})\label{tab:esdirk}}
\centering
\begin{tabular*}{\textwidth}{@{\extracolsep\fill}ccc@{\extracolsep\fill}}
	\textbf{Trapezoidal rule(order 2, A-stable)} & \textbf{ESDIRK32 (order 2, L-stable)} & \textbf{ESDIRK33 (order 3, A-stable)}\\
	\begin{tikzpicture}
	\node (0,0) {$\renewcommand\arraystretch{2.5}\arraycolsep=8pt
	\begin{array}
	{c|cc}
	0 & 0 & 0\\
	1 & \tfrac{1}{2} & \tfrac{1}{2}\\
	\cline{1-3}
	& \tfrac{1}{2} & \tfrac{1}{2}
	\end{array}$};
	\end{tikzpicture} & \begin{tikzpicture}
	\node (0,0) {$\renewcommand\arraystretch{2.5}\arraycolsep=8pt
	\begin{array}
	{c|ccc}
	0 & 0 & 0 & 0\\
	2\gamma & \gamma & \gamma & 0\\
	1 & 1-b-\gamma & b & \gamma\\
	\cline{1-4}
	& 1-b-\gamma & b & \gamma
	\end{array}$};
	\end{tikzpicture} & \begin{tikzpicture}
	\node (0,0) {$\renewcommand\arraystretch{2.5}\arraycolsep=8pt
	\begin{array}
	{c|ccc}
	0 & 0 & 0 & 0\\
	2\gamma & \gamma & \gamma & 0\\
	1 & \tfrac{6\gamma-1}{4\gamma}-\gamma & \tfrac{1-2\gamma}{4\gamma} & \gamma\\
	\cline{1-4}
	& 1-b_2-b_3 & b_2 & b_3
	\end{array}$};
	\end{tikzpicture}\\
	& $\gamma = \tfrac{2-\sqrt{2}}{2}, ~b = \tfrac{1-2\gamma}{4\gamma}$ & $\gamma = \tfrac{3+\sqrt{3}}{6}, ~b_2 = \tfrac{1}{12\gamma(1-2\gamma)}, ~b_3 = \tfrac{1-3\gamma}{3(1-2\gamma)}$
\end{tabular*}
\end{table*}
\end{center}

\subsection{Diagonally Implicit Runge-Kutta (DIRK) Methods}\label{app:dirk}
\begin{center}
\begin{table*}[h!t]\caption{Gauss4 (taken from Table 5.1 of Hairer and Wanner\cite{EHairer_GWanner_2010}) and TR-BDF2 (taken from Hosea and Shampine \cite{Hosea_Shampine_1996})\label{tab:dirk}}
\centering
\begin{tabular*}{\textwidth}{@{\extracolsep\fill}cc@{\extracolsep\fill}}
	\textbf{Gauss4 (order 4,  A-stable)} & \textbf{TR-BDF2$\boldsymbol{(\theta)}$ (order 2, A-stable)}\\
	& (L-stable if $\theta = 2-\sqrt{2})$\\
	\begin{tikzpicture}
	\node (0,0) {$\renewcommand\arraystretch{1.8}\arraycolsep=8pt
	\begin{array}
	{c|cc}
	\tfrac{3-\sqrt{3}}{6} & \tfrac{1}{4} & \tfrac{3-2\sqrt{3}}{12}\\
	\tfrac{3+\sqrt{3}}{6} & \tfrac{3+2\sqrt{3}}{12} & \tfrac{1}{4} \\ 
	\hline
	& \tfrac{1}{2} & \tfrac{1}{2}
	\end{array}$};
	\useasboundingbox (0,-3.5) rectangle (2.5,.5);
	\end{tikzpicture} & \begin{tikzpicture}
	\node (0,0) {$\renewcommand\arraystretch{2.5}\arraycolsep=8pt
	\begin{array}
	{c|ccc}
	0 & 0 & 0 & 0\\
	\theta & \tfrac{\theta}{2} & \tfrac{\theta}{2} & 0\\
	1 & \tfrac{3\theta-\theta^2-1}{2\theta} & \tfrac{1-\theta}{2\theta} & \tfrac{\theta}{2}\\
	\cline{1-4}
	& \tfrac{3\theta-\theta^2-1}{2\theta} & \tfrac{1-\theta}{2\theta} & \tfrac{\theta}{2}
	\end{array}$};
	\end{tikzpicture}
\end{tabular*}
\end{table*}
\end{center}

\section*{Acknowledgments}
This work was performed under the auspices of the U.S. Department of Energy 
under grant number (NNSA) DE-NA0002376 and
Lawrence Livermore National Laboratory under contract B634212.
\bibliography{refs}

\end{document}